\pdfoutput=1
\documentclass[12pt,reqno]{amsart}

\usepackage{a4wide}
\usepackage{amsmath}
\usepackage{amssymb}
\usepackage{amsxtra}
\usepackage{amsthm}
\usepackage{mathtools}
\usepackage{extpfeil}
\usepackage{enumitem}
\usepackage[mathscr]{eucal}
\usepackage{graphicx}
\usepackage[T1]{fontenc}

\usepackage[all]{xy}
\SelectTips{cm}{}
\SilentMatrices
\ifx\pdfoutput\undefined
  \xyoption{dvips}
\fi

\usepackage{hyperref}
\usepackage{slashed}

\newcommand{\ifundef}[1]{\expandafter\ifx\csname#1\endcsname\relax}

\DeclareMathAlphabet{\mathbbe}{U}{bbold}{m}{n}

\makeatletter

\def\re@DeclareMathSymbol#1#2#3#4{%
    \let#1=\undefined
    \DeclareMathSymbol{#1}{#2}{#3}{#4}}

\ifundef{Top}
  \DeclareSymbolFont{tcSyC}{U}{txsyc}{m}{n}
  \SetSymbolFont{tcSyC}{bold}{U}{txsyc}{bx}{n}
  \DeclareFontSubstitution{U}{txsyc}{m}{n}

  \re@DeclareMathSymbol{\Top}{\mathord}{tcSyC}{120}
  \re@DeclareMathSymbol{\Bot}{\mathord}{tcSyC}{121}
\fi

\ifundef{righthalfcup}
  \DeclareFontFamily{U}{MnSymbolC}{}
  \DeclareSymbolFont{mnSyC}{U}{MnSymbolC}{m}{n}
  \SetSymbolFont{mnSyC}{bold}{U}{MnSymbolC}{b}{n}
  \DeclareFontShape{U}{MnSymbolC}{m}{n}{
      <-6>  MnSymbolC5
     <6-7>  MnSymbolC6
     <7-8>  MnSymbolC7
     <8-9>  MnSymbolC8
     <9-10> MnSymbolC9
    <10-12> MnSymbolC10
    <12->   MnSymbolC12}{}
  \DeclareFontShape{U}{MnSymbolC}{b}{n}{
      <-6>  MnSymbolC-Bold5
     <6-7>  MnSymbolC-Bold6
     <7-8>  MnSymbolC-Bold7
     <8-9>  MnSymbolC-Bold8
     <9-10> MnSymbolC-Bold9
    <10-12> MnSymbolC-Bold10
    <12->   MnSymbolC-Bold12}{}
  
  \re@DeclareMathSymbol{\righthalfcup}{\mathord}{mnSyC}{184}
  \re@DeclareMathSymbol{\lefthalfcap}{\mathord}{mnSyC}{185}
\fi

\DeclareFontFamily{U}{MnSymbolA}{}
\DeclareSymbolFont{mnSyA}{U}{MnSymbolA}{m}{n}
\SetSymbolFont{mnSyA}{bold}{U}{MnSymbolA}{b}{n}
\DeclareFontShape{U}{MnSymbolA}{m}{n}{
    <-6>  MnSymbolA5
   <6-7>  MnSymbolA6
   <7-8>  MnSymbolA7
   <8-9>  MnSymbolA8
   <9-10> MnSymbolA9
  <10-12> MnSymbolA10
  <12->   MnSymbolA12}{}
\DeclareFontShape{U}{MnSymbolA}{b}{n}{
    <-6>  MnSymbolA-Bold5
   <6-7>  MnSymbolA-Bold6
   <7-8>  MnSymbolA-Bold7
   <8-9>  MnSymbolA-Bold8
   <9-10> MnSymbolA-Bold9
  <10-12> MnSymbolA-Bold10
  <12->   MnSymbolA-Bold12}{}

\re@DeclareMathSymbol{\twoheadedswarrow}{\mathord}{mnSyA}{30}

\makeatother

\newcommand{\mlaux}[3]{\setbox0=\hbox{$\mathsurround=0pt #2{#3}$}%
  \dimen0=\dp0\advance\dimen0 by \ht0\lower#1\dimen0\box0}

\newcommand{\makellapm}[2]{\hbox to 0pt{\hss$\mathsurround=0pt #1{#2}$}}

\newcommand{\makerlapm}[2]{\hbox to 0pt{$\mathsurround=0pt #1{#2}$\hss}}

\newcommand{\makelapm}[2]{\hbox to 0pt{\hss$\mathsurround=0pt #1{#2}$\hss}}

\newcommand{\makeushort}[3]{%
	\setbox0=\hbox{$\mathsurround=0pt #2{#3}$}%
	\hbox to 1\wd0{\hss\underbar{\hbox to #1\wd0{\hss\box0\hss}}\hss}}

\def\makebigger#1#2#3{\scalebox{#1}{$\mathsurround=0pt #2{#3}$}}
\def\bigger#1#2{{\relax\mathpalette{\makebigger{#1}}{#2}}}

\def\scaleuphalf{1.0954}
\def\scaleupone{1.2}

\newcommand{\op}{^{\mathord{\text{\rm op}}}}
\newcommand{\co}{^{\mathord{\text{\rm co}}}}

\newcommand{\mapcat}{^\cattwo}

\renewcommand{\th}{^{\text{th}}}

\newcommand{\defeq}{\mathrel{:=}}

\makeatletter

\def\newmop{\@ifstar{\@newmop m}{\@newmop o}}
\def\@newmop#1{\@ifnextchar[{\@@newmop #1}{\@@@newmop #1}}
\def\@@newmop#1[#2]{\@declmathop #1#2}
\def\@@@newmop#1#2{\expandafter\@declmathop\expandafter #1\csname #2\endcsname{#2}}

\makeatother

\newdir{ |}{{}*!/-5pt/@{|}}

\newmop{im}
\newmop{coim}
\newmop{dom}
\newmop{cod}
\newmop{id}
\newmop{Map}

\newmop{obj}
\newmop{arr}
\newmop{sq}
\newmop{norm}

\newmop{el}

\newmop{ev}

\newmop{Ext}
\newmop{icon}
\newmop{pbk}

\newmop{cell}
\newmop{cof}
\newmop{fib}

\newmop{diag}

\newmop*{colim}
\newmop*{holim}

\newmop[\lan]{lan}
\newmop[\ran]{ran}

\newcommand{\comma}{\mathbin{\downarrow}}

\makeatletter
\newcommand{\rotatemath}[2]{\rotatebox[origin=c]{180}{$\m@th #1{#2}$}}
\makeatother

\newmop[\pr]{pr}
\newmop[\dm]{d} 

\newmop[\cpr]{in}

\newcommand{\pocorner}{\hbox to 8pt{{\vrule height8pt depth0pt width0.5pt}%
    \vbox to 8pt{{\hrule height0.5pt width7.5pt depth0pt}\vfill}}}
\newcommand{\poexcursion}{\save[]-<15pt,-15pt>*{\pocorner}\restore}
\newcommand{\pbcorner}{\vbox to 0pt{\kern 4pt\hbox to 0pt{\kern 4pt%
      \vbox{{\hrule height0.5pt width7.5pt depth0pt}}%
      {\vrule height8pt depth0pt width0.5pt}\hss}\vss}}
\newcommand{\pbexcursion}{\save[]+<5pt,-5pt>*{\pbcorner}\restore}

\newmop{cls}

\newcommand{\leib}[1]{\mathbin{\widehat{#1}}}

\newcommand{\category}[1]{\underline{\smash[b]{\text{\rm{#1}}}}}

\newcommand{\tcat}{\lcat}

\newcommand{\lcat}{\mathcal}
\newcommand{\scat}{\mathbf}
\newcommand{\stcat}{\scat}

\newcommand{\catthree}{{\bigger{1.12}{\mathbbe{3}}}}
\newcommand{\cattwo}{{\bigger{1.12}{\mathbbe{2}}}}
\newcommand{\catone}{{\bigger{1.16}{\mathbbe{1}}}}
\newcommand{\iso}{{\bigger\scaleuphalf{\mathbb{I}}}}

\makeatletter

\def\Del@Sym{{\bigger\scaleuphalf{\mathbbe{\Delta}}}}

\def\del@fn{\futurelet\del@next}
\def\del@dn{\def\del@next}

\def\parsedel@{%
  \ifx +\del@next \del@dn+{\Del@Sym_{\mathord{+}}}%
  \else \del@dn {\del@fn\parsedel@@}%
  \fi\del@next}

\def\parsedel@@{%
  \ifx\space@\del@next \expandafter\del@dn\space{\del@fn\parsedel@@}%
  \else\ifx [\del@next \del@dn[{\del@fn\parsedel@@@}%
  \else\ifx _\del@next \del@dn{\Delta}%
  \else\ifx ^\del@next \del@dn{\Delta}%
  \else \del@dn{\Del@Sym}%
  \fi\fi\fi\fi\del@next}

\def\parsedel@@@{%
  \ifx\space@\del@next \expandafter\del@dn\space{\del@fn\parsedel@@@}%
  \else\ifx t\del@next \del@dn t{\Del@Sym_\infty\del@fn\parsedel@@@@}%
  \else\ifx b\del@next \del@dn b{\Del@Sym_{-\infty}\del@fn\parsedel@@@@}%
  \else \del@dn{\errmessage{unexpected modifier}}%
  \fi\fi\fi\del@next}

\def\parsedel@@@@{%
  \ifx\space@\del@next \expandafter\del@dn\space{\del@fn\parsedel@@@@}%
  \else\ifx ]\del@next \del@dn]{}%
  \else \del@dn{\errmessage{expecting close of option block}}%
  \fi\fi\del@next}

\def\Del{\del@fn\parsedel@}

\makeatother

\newcommand{\Horn}{\Lambda}

\newcommand{\Set}{\category{Set}}
\newcommand{\Cat}{\category{Cat}}

\newcommand{\eCat}[1]{#1\text{-}\Cat} % Enriched categories

\newcommand{\sCat}{\sSet\text{-}\Cat}
\newcommand{\sSet}{\category{sSet}}
\newcommand{\qCat}{\category{qCat}}
\newcommand{\Adj}{\category{Adj}}
\newcommand{\Mnd}{\category{Mnd}}

\newcommand{\asSet}{\sSet_{\mathord{+}}}
\newcommand{\msSet}{\category{msSet}} % marked simplicial sets = strat

\newcommand{\twoCat}{\eCat{2}}

\newcommand{\pbshape}{{\mathord{\bigger\scaleupone\righthalfcup}}}
\newcommand{\poshape}{{\mathord{\bigger\scaleupone\lefthalfcap}}}

\newcommand{\face}{\delta}

\newcommand{\degen}{\sigma}
\newcommand{\aug}{\iota}
\newcommand{\tdegen}{\varsigma}

\newcommand{\fbv}[1]{\{{#1}\}}

\newcommand{\join}{\star}
\newcommand{\fatjoin}{\mathbin\diamond}
\newmop{dec}
\newmop{fatdec}
\newmop{slc}
\newmop{fatslc}
\newcommand{\fatslice}{\mathbin{\mkern-1mu{/}\mkern-5mu{/}\mkern-1mu}}
\newcommand{\fatslicel}[2]{\vphantom{#2}^{{#1}\fatslice{}}\mkern-2mu{#2}}
\newcommand{\fatslicer}[2]{{#1\mkern-1mu}_{{}\fatslice{#2}}}
\newcommand{\slice}{/}
\newcommand{\slicel}[2]{\vphantom{#2}^{{#1}\slice{}}\mkern-2mu{#2}}
\newcommand{\slicer}[2]{{#1\mkern-1mu}_{{}\slice{#2}}}

\newmop{ir} % the interval representation functor
\newmop{incl}

\newcommand{\nrv}{N}
\newcommand{\ho}{h}

\newcommand{\nrvhc}{\nrv}
\newcommand{\gC}{\mathfrak{C}}%left adjoint to the homotopy coherent nerve

\newcommand{\boundary}{\partial}

\def\reedyfilt#1_#2{#1_{\leq #2}}
\newmop{sk}
\newmop{cosk}
\newmop{res}

\newcommand{\mclass}{\mathcal}

\newmop{map}
\newmop{Ho}

\newcommand{\cpts}{\pi_0}

\newmop[\pth]{path}
\newmop{cyl}

\newmop[\const]{c}
  
\newmop{coll}
\newmop{wgt}

\newdir{ >}{{}*!/-7pt/@{>}}
\newdir{u(}{{}*!/-4pt/@^{(}}
\newdir{d(}{{}*!/-4pt/@_{(}}
\newdir{|>}{%
  !/4.5pt/@{|}*:(1,-.2)@^{>}*:(1,+.2)@_{>}*+@{}}

\makeatletter
\def\makeslashed#1#2#3#4#5{#1{\mathpalette{\sla@{#2}{#3}{#4}}{#5}}}

\def\@mathlower#1#2#3{\setbox0=\hbox{$\m@th#2#3$}\lower#1\ht0\box0}
\def\mathlower#1#2{\mathpalette{\@mathlower{#1}}{#2}}
\makeatother

\newcommand{\inc}{\hookrightarrow}

\newcommand{\tcof}{\hookrightarrow}
\newcommand{\tfib}{\twoheadrightarrow}

\newcommand{\trvfib}{\xtwoheadrightarrow{\smash{\mathlower{1.2}{\sim}}}}

\newcommand{\To}{\Rightarrow}

\makeatletter

\def\tens@fn{\futurelet\tens@next}
\def\tens@dn{\def\tens@nextcont}
\newtoks\tens@toks
\def\addtotens@toks#1{\tens@toks=\expandafter{\the\tens@toks#1}}

\def\parsetens@@{%
    \ifx\space@\tens@next \expandafter\tens@dn\space{\tens@fn\parsetens@@}%
    \else\ifx ^\tens@next \tens@dn ^##1{\parsetens@procsep^\addtotens@toks{##1}%
      \tens@fn\parsetens@@}%
    \else\ifx _\tens@next \tens@dn _##1{\parsetens@procsep_\addtotens@toks{##1}%
      \tens@fn\parsetens@@}%
    \else\tens@dn{\ifx *\tens@last \else\addtotens@toks\egroup\fi\the\tens@toks}%
    \fi\fi\fi\tens@nextcont}

\def\parsetens@procsep#1{%
  \ifx *\tens@last \addtotens@toks{#1}\addtotens@toks\bgroup%
  \else\ifx \tens@last\tens@next \addtotens@toks,%
  \else \addtotens@toks\egroup\addtotens@toks\bgroup%
    \addtotens@toks\egroup\addtotens@toks{#1}\addtotens@toks\bgroup% 
  \fi\fi\let\tens@last\tens@next}

\newcommand{\tn}[1]{\let\tens@last=*\tens@toks={#1}\tens@fn\parsetens@@}

\makeatother

\def\adjdisplay#1-|#2:#3->#4.{{%
    \xymatrix@R=0em@!C=2.5em{%
      *+[l]{#3} \ar@/_0.55pc/[rr]_-{#2} & {\bot} &
      *+[r]{#4}\ar@/_0.55pc/[ll]_-{#1}}}}

\def\adjdisplaytwo#1-|#2:#3->#4.{{% 
\xymatrix@=1.2em{
      {#3}\ar@/_1.5ex/[rr]_-{#2}^-{}="one"
      & & {#4}
      \ar@/_1.5ex/[ll]_-{#1}^-{}="two" 
      \ar@{}"one";"two"|{\bot}
    }}}

\def\tripleadjdisplay#1-|#2-|#3:#4->#5.{{%
\xymatrix@=2.4em{ 
{#4}\ar[r]|{#2} &
{#5} \ar@/_3ex/[l]_{#1}^{\bot} \ar@/^3ex/[l]_{\bot}^{#3}}
}}

\def\adjinline#1-|#2:#3->#4.{{#1}\dashv{#2}:#3\to #4}

\newcommand{\pent}[1]{
  \xybox{
    \POS (0,-15)*+{\a}="0", 
         (-14,-5)*+{\b}="1", 
         (-9,12)*+{\c}="2", 
         (9,12)*+{\d}="3", 
         (14,-5)*+{\e}="4"
    \POS"0" \ar "1"^{\labelstyle \ab}|{}="01"
    \POS"1" \ar "2"^{\labelstyle \bc}|{}="12"
    \POS"2" \ar "3"^{\labelstyle \cd}|{}="23"
    \POS"3" \ar "4"^{\labelstyle \de}|{}="34"
    \POS"0" \ar "4"_{\labelstyle \ae}|{}="04"
    \ifcase #1
    \POS"0" \ar "2"|{\labelstyle \ac}="02"
    \POS"0" \ar "3"|{\labelstyle \ad}="03"
    \POS"02";"1"**{}, ?(0.3) \ar@{=>} ?(0.7)^{\labelstyle \abc}
    \POS"03";"2"**{}, ?(0.25) \ar@{=>} ?(0.5)_{\labelstyle \acd}
    \POS"04";"3"**{}, ?(0.2) \ar@{=>} ?(0.4)_{\labelstyle \ade}
    \or
    \POS"1" \ar "3"|{\labelstyle \bd}="13"
    \POS"1" \ar "4"|{\labelstyle \be}="14"
    \POS"13";"2"**{}, ?(0.3) \ar@{=>} ?(0.7)_{\labelstyle \bcd}
    \POS"14";"3"**{}, ?(0.25) \ar@{=>} ?(0.5)_{\labelstyle \bde}
    \POS"04";"1"**{}, ?(0.25) \ar@{=>} ?(0.5)_{\labelstyle \abe}
    \or
    \POS"2" \ar "4"|{\labelstyle \ce}="24"
    \POS"0" \ar "2"|{\labelstyle \ac}="02"
    \POS"02";"1"**{}, ?(0.3) \ar@{=>} ?(0.7)^{\labelstyle \abc}
    \POS"04";"2"**{}, ?(0.2) \ar@{=>} ?(0.35)_{\labelstyle \ace}
    \POS"24";"3"**{}, ?(0.2) \ar@{=>} ?(0.6)^{\labelstyle \cde}
    \or
    \POS"1" \ar "3"|{\labelstyle \bd}="13"
    \POS"0" \ar "3"|{\labelstyle \ad}="03"
    \POS"04";"3"**{}, ?(0.2) \ar@{=>} ?(0.4)_{\labelstyle \ade}
    \POS"13";"2"**{}, ?(0.3) \ar@{=>} ?(0.7)_{\labelstyle \bcd}
    \POS"03";"1"**{}, ?(0.25) \ar@{=>} ?(0.5)^{\labelstyle \abd}
    \or
    \POS"2" \ar "4"|{\labelstyle \ce}="24"
    \POS"1" \ar "4"|{\labelstyle \be}="14"
    \POS"24";"3"**{}, ?(0.2) \ar@{=>} ?(0.6)^{\labelstyle \cde}
    \POS"04";"1"**{}, ?(0.25) \ar@{=>} ?(0.5)_{\labelstyle \abe}
    \POS"14";"2"**{}, ?(0.25) \ar@{=>} ?(0.5)^{\labelstyle \bce}
    \else\fi
  }
}

\newcommand{\pentofpent}[1]{
  \def\baselen{#1}
  \begin{xy}
    0;<\baselen,0mm>:
    *{\xybox{
        \POS(0,-4)*[o]{\pent 0}="zero"
        \POS(16,40)*[o]{\pent 3}="three"
        \POS(72,40)*[o]{\pent 1}="one"
        \POS(88,-4)*[o]{\pent 4}="four"
        \POS(44,-36)*[o]{\pent 2}="two"
        \ar@<1ex>"zero";"three"^-{\objectstyle\abcd}
        \ar@<1ex>"three";"one"^-{\objectstyle\abde}
        \ar@<1ex>"one";"four"^-{\objectstyle\bcde}
        \ar@<-1ex>"zero";"two"_-{\objectstyle\acde}
        \ar@<-1ex>"two";"four"_-{\objectstyle\abce}
        \ar@{=>}(44,-5);(44,+15)^{\objectstyle\abcde}
     }}
  \end{xy}
}

\setlist{}
\setenumerate{leftmargin=*,labelindent=\parindent}
\setitemize{leftmargin=*,labelindent=0.5\parindent}
\setdescription{leftmargin=1em}

\swapnumbers

\theoremstyle{plain}
\newtheorem{thm}{Theorem}[subsection]
\newtheorem{lem}[thm]{Lemma}
\newtheorem{cor}[thm]{Corollary}

\newtheorem{prop}[thm]{Proposition}

\theoremstyle{definition}
\newtheorem{defn}[thm]{Definition}
\newtheorem{ex}[thm]{Example}
\newtheorem{ntn}[thm]{Notation}

\theoremstyle{remark}
\newtheorem{obs}[thm]{Observation}
\newtheorem{rec}[thm]{Recall}
\newtheorem{rmk}[thm]{Remark}

\makeatletter
\let\c@equation\c@thm
\makeatother
\numberwithin{equation}{subsection}

\newcommand{\refI}[1]{I.\ref*{found:#1}}

\title{The 2-category theory of quasi-categories}
\date{$6^{\text{th}}$ May 2015}

\author[Riehl]{Emily Riehl}
\address{
  Department of Mathematics \\
  Harvard University \\
  Cambridge, MA 02138\\
  USA
}
\email{eriehl@math.harvard.edu}

\author[Verity]{Dominic Verity}
\address{
  Centre of Australian Category Theory \\
  Macquarie University \\
  NSW 2109 \\
  Australia
}
\email{dominic.verity@mq.edu.au}

\subjclass[2010]{%
  Primary  18G55, 55U35, 55U40; %
  Secondary 18A05, 18D20, 18G30, 55U10%18D35, 18F99
}

\begin{document}
  
  \ifpdf
  \DeclareGraphicsExtensions{.pdf, .jpg, .tif}
  \else
  \DeclareGraphicsExtensions{.eps, .jpg}
  \fi

\begin{abstract}
In this paper we re-develop the foundations of the category theory of quasi-categories (also called $\infty$-categories) using 2-category theory. We show that Joyal's strict 2-category of quasi-categories admits certain weak 2-limits, among them weak comma objects. We use these comma quasi-categories to encode universal properties relevant to limits, colimits, and adjunctions and prove the expected theorems relating these notions. These universal properties have an alternate form as absolute lifting diagrams in the 2-category, which we show are determined pointwise by the existence of certain initial or terminal vertices, allowing for the easy production of examples.

All the quasi-categorical notions introduced here are equivalent to the established ones but our proofs are independent and more ``formal''. In particular, these results generalise immediately to model categories enriched over quasi-categories.
\end{abstract}

\maketitle
\tableofcontents

%!TEX root = all.tex
% ******************************************************************
% ** Title:            The 2-category theory of quasi-categories
% **                  introduction
% ** Precis:        
% ** Author:           Emily Riehl and Dominic Verity
% ** Commenced:        2/3/2012
% ******************************************************************

\section{Introduction}

Quasi-categories, also called $\infty$-categories, were introduced by J.~Michael Boardman and Rainer Vogt under the name ``weak Kan complexes'' in their book \cite{Boardman:1973xo}. Their aim was to describe the weak composition structure enjoyed by homotopy coherent natural transformations between homotopy coherent diagrams. Other examples of quasi-categories include ordinary categories (via the nerve functor) and topological spaces (via the total singular complex functor), which are Kan complexes: quasi-categories in which every 1-morphism is invertible. Topological and simplicial (model) categories also have associated quasi-categories (via the homotopy coherent nerve). Quasi-categories provide a convenient model for $(\infty,1)$-categories: categories weakly enriched in $\infty$-groupoids or topological spaces. Following the program of Boardman and Vogt, many homotopy coherent structures naturally organise themselves into a quasi-category.

For this reason, it is desirable to extend the definitions and theorems of ordinary category theory into the $(\infty,1)$-categorical and specifically into the quasi-categorical context. As categories form a full subcategory of quasi-categories, a principle guiding the quasi-categorical definitions is that these should restrict to the classically understood categorical concepts on this full subcategory. In this way, we think of quasi-category theory as an extension of category theory---and indeed use the same notion for a category and the quasi-category formed by its nerve.

There has been significant work (particularly if measured by page count) towards the development of the category theory of quasi-categories, the most well-known being the articles and unpublished manuscripts of Andr\'{e} Joyal \cite{Joyal:2002:QuasiCategories,Joyal:2007kk,Joyal:2008tq} and the books of Jacob Lurie \cite{Lurie:2009fk,Lurie:2012uq}. Other early work includes the PhD thesis of Joshua Nichols-Barrer \cite{NicholsBarrer:2007oq}.  More recent foundational developments are contained in work of David Gepner and Rune Haugseng \cite{GepnerHaugseng:2013ec}, partially joint with Thomas Nikolaus \cite{GepnerHaugsengNikolaus:2015lc}. Applications of quasi-category theory, for instance to derived algebraic geometry, are already too numerous to mention individually.

Our project is to provide a second generation, {\em formal\/} category theory of quasi-categories, developed from the ground up. Each definition given here is equivalent to the established one, 
but we find our development to be more intuitive and the proofs to be simpler. Our hope is that this self-contained account will be more approachable to the outsider hoping to better understand the foundations of the quasi-category theory he or she may wish to use.

In this paper, we use 2-category theory to develop the category theory of quasi-categories.  The starting point is a (strict) 2-category of quasi-categories $\qCat_2$ defined as a quotient of the simplicially enriched category of quasi-categories $\qCat_\infty$.  The underlying category of both enriched categories is the usual category of quasi-categories and simplicial maps, here called simply ``functors''. We translate simplicial universal properties into 2-categorical ones: for instance, the simplicially enriched universal properties of finite products and the hom-spaces between quasi-categories imply that the 2-category $\qCat_2$ is cartesian closed. Importantly, equivalences in the 2-category $\qCat_2$ are precisely the (weak) equivalences of quasi-categories introduced by Joyal, which means that this 2-category appropriately captures the homotopy theory of quasi-categories.

Aside from finite products, $\qCat_2$ admits few strict 2-limits. However, it admits several important weak 2-limits of a sufficiently strict variety with which to develop formal category theory. Weak 2-limits in $\qCat_2$ are not unique up to isomorphism; rather their universal properties characterise these objects up to equivalence, exactly as one would expect in the $(\infty,1)$-categorical context. We show that $\qCat_2$ admits weak cotensors by categories freely generated by a graph (including, in particular, the walking arrow) and weak comma objects, which we use  to encode the universal properties associated to limits, colimits, adjunctions, and so forth.

A complementary paper \cite{RiehlVerity:2012hc} will showcase a corresponding ``internal'' approach to this theory. The basic observation is that the simplicial category of quasi-categories $\qCat_\infty$ is closed under the formation of weighted limits whose weights are projectively cofibrant simplicial functors. Examples include Bousfield-Kan style homotopy limits and a variety of weighted limits relating to homotopy coherent adjunctions.

In \cite{RiehlVerity:2012hc}, we show that any adjunction of quasi-categories can be extended to a {\em homotopy coherent adjunction}, by which we mean a simplicial functor whose domain is a particular cofibrant simplicial category that we describe in great detail. Unlike previous renditions of coherent adjunction data, our formulation is symmetric: in particular, a homotopy coherent adjunction restricts to a homotopy coherent monad and to a homotopy coherent comonad on the two quasi-categories under consideration. As a consequence of its cofibrancy, various weights extracted from the free homotopy coherent adjunction are projectively cofibrant simplicial functors. We use these to define the quasi-category of algebras associated to a homotopy coherent monad and provide a formal proof of the monadicity theorem of Jon Beck. More details can be found there.

\subsection{A generalisation}

In hopes that our proofs would be  more readily absorbed in familiar language, we have neglected to state our results in their most general setting, referencing only the simplicially enriched full subcategory of quasi-categories $\qCat_\infty$. Nonetheless, a key motivation for our project is that our proofs apply to more general settings which are also of interest.  

Consider a Quillen model category that is enriched as a model category relative to the Joyal model structure on simplicial sets and in which every fibrant object is also cofibrant. Then its full simplicial subcategory of fibrant objects is what we call an $\infty$-\emph{cosmos}; a simple list of axioms, weaker than the model category axioms, will be described in a future paper. Weak equivalences and fibrations between fibrant objects will play the role of the equivalences and isofibrations here. Examples of Quillen model categories which satisfy these conditions include Joyal's model category of quasi-categories and any model category of complete Segal spaces in a suitably well behaved model category. The canonical example \cite{Joyal:2007kk,Rezk:2001sf} is certainly included under this heading but we have in mind more general ``Rezk spaces'' as well. Given a well-behaved model category $\lcat{M}$, the localisation of the Reedy model structure on the category $\lcat{M}^{\Del\op}$ whose fibrant objects are complete Segal objects is enriched as a model category over the Joyal model structure on simplicial sets. All of the definitions that are stated and theorems that are proven here apply representably to any $\infty$-cosmos, being a simplicial category whose hom-spaces are quasi-categories and whose quotient 2-category admits the same weak 2-limits utilised here.

\subsection{Outline}

Our approach to the foundations of quasi-category theory is independent of the existing developments with one exception: we accept as previously proven the Joyal model structure for quasi-categories on simplicial sets and the model structure for naturally marked quasi-categories on marked simplicial sets. So that a reader can begin his or her acquaintance with the subject by reading this paper, we begin with a comprehensive background review in section \ref{sec:background}, where we also establish our notational conventions. 

In section \ref{sec:twocat}, we introduce the 2-category of quasi-categories $\qCat_2$ and investigate its basic properties. Of primary importance is the particular notion of weak 2-limit introduced here. Following \cite{Kelly:1989fk}, a strict 2-limit can be defined representably: the hom-categories mapping into the 2-limit are required to be naturally isomorphic to the corresponding 2-limit of hom-categories formed in $\Cat$. In our context, there is a canonical functor from the former category to the latter but it is not an isomorphism. Rather it is what we term a {\em smothering functor}: surjective on objects, full, and conservative. We develop the basic theory of these weak 2-limits and prove that $\qCat_2$ admits certain weak cotensors, weak 2-pullbacks, and weak comma objects.

In section \ref{sec:qcatadj}, we begin to develop the formal category theory of quasi-categories by introducing adjunctions between quasi-categories, which are defined simply to be adjunctions in the 2-category $\qCat_2$; this definition was first considered by Joyal. It follows immediately that adjunctions are preserved by pre- and post-composition, since these define 2-functors on $\qCat_2$. Any equivalence of quasi-categories extends to an adjoint equivalence, and that any adjunction between Kan complexes is automatically an adjoint equivalence. We describe an alternate form of the universal property of an adjunction which will be a key ingredient in the proof of the main existence theorem of \cite{RiehlVerity:2012hc}. Finally, we show that many of our adjunctions are in fact {\em fibred}, meaning that they are also adjunctions in the 2-category obtained as a quotient of the simplicial category of isofibrations over a fixed quasi-category. Any map between the base quasi-categories defines a pullback 2-functor, which then preserves fibred equivalences, fibred adjunctions, and so forth.

In section \ref{sec:limits}, we define limits and colimits in a quasi-category in terms of absolute right and left lifting diagrams in $\qCat_2$. A key technical theorem provides an equivalent definition as a fibred equivalence of comma quasi-categories. We prove the expected results relating limits and colimits to adjunctions: that right adjoints preserve limits, that limits of a fixed shape can be encoded as adjoints to constant diagram functors provided these exist, that limits and limit cones assemble into right Kan extensions along the join functor, and so on. As an application of these general results, we give a quick proof that any quasi-category admitting pullbacks, pushouts, and a zero object has a ``loops--suspension'' adjunction. This forms the basis for the notion of a {\em stable\/} quasi-category.

We conclude section \ref{sec:limits} with an example particularly well suited to our 2-categorical approach that will reappear in the proof of the monadicity theorem in \cite{RiehlVerity:2012hc}: generalising a classical result from simplicial homotopy theory, we show that if a simplicial object in a quasi-category admits an augmentation and ``extra degeneracies'', then the augmentation is its quasi-categorical colimit and also encodes the canonical colimit cone. Our proof is entirely 2-categorical. There exists an absolute left extension diagram in $\Cat$ involving $\Del$ and related categories and furthermore this 2-universal property is witnessed equationally by various adjunctions. Such universal properties are preserved by any 2-functor---for instance, homming into a quasi-category---and the result follows immediately.

Having established the importance of absolute lifting diagrams, which characterise limits, colimits, and adjunctions in the quasi-categorical context, it is important to develop tools which can be used to show that such diagrams exist in $\qCat_2$. This is the aim of section \ref{sec:pointwise}. In this section,  we show that a cospan $B \xrightarrow{f} A \xleftarrow{g} C$ admits an absolute right lifting of $g$ along $f$ if and only if for each object $c \in C$, the slice (or comma) quasi-category from $f$ down to $gc$ has a terminal object. In practice, this ``pointwise'' universal property is much easier to check than the global one encoded by the absolute lifting diagram. 

To illustrate, we use this theorem to show that any simplicial Quillen adjunction between simplicial model categories defines an adjunction of quasi-categories. The proof of this result is more subtle than one might suppose. The quasi-category associated to a simplicial model category is defined by applying the homotopy coherent nerve to the subcategory of fibrant-cofibrant objects---in general, the mapping spaces between arbitrary objects need not have the ``correct'' homotopy type. On account of this restriction, the point-set level left and right adjoints do not directly descend to functors between these quasi-categories so the quasi-categorical adjunction must be defined in some other way.

We conclude this paper with a technical appendix proving that the comma quasi-categories used here are equivalent to the slice quasi-categories introduced by Joyal \cite{Joyal:2002:QuasiCategories}. It follows that the categorical definitions introduced in this paper coincide with the definitions found in the existing literature.

\subsection{Acknowledgments}

During the preparation of this work the authors were supported by DARPA through the AFOSR grant number HR0011-10-1-0054-DOD35CAP and by the Australian Research Council through Discovery grant number DP1094883. The first-named author was also supported by an NSF postdoctoral research fellowship DMS-1103790. A careful reading by an anonymous referee has lead to numerous improvements in the exposition throughout this article. We would also like to extend personal thanks to Mike Hopkins without whose support and encouragement this work would not exist.

%!TEX root = all.tex
% ******************************************************************
% ** Title:            The 2-category theory of quasi-categories
% **                  Background
% ** Precis:        
% ** Author:           Emily Riehl and Dominic Verity
% ** Commenced:        2/3/2012
% ******************************************************************

  \section{Background on quasi-categories}\label{sec:background}

	We start by reviewing some basic concepts and notations. 

  \begin{obs}[size]\label{obs:size-conventions}
    In this paper matters of size will not be of great importance. However, for definiteness we shall adopt the usual conceit of assuming that we have fixed an inaccessible cardinal which then determines a corresponding Grothendieck universe, members of which will be called {\em sets\/}; we refer to everything else as {\em classes}.  A category is {\em small\/} if it has sets of objects and arrows; a category is {\em locally small\/} if each of its hom-sets is small. We shall write $\Set$ to denote the large and locally small category of all sets and functions between them.

    When discussing the existence of limits and colimits we shall implicitly assume that these are indexed by small categories. Correspondingly, completeness and cocompletess properties will implicitly reference the existence of small limits and small colimits.
     \end{obs}

    \subsection{Some standard simplicial notation}\label{subsect:simplicial.notation}

    \begin{ntn}[simplicial operators]\label{ntn:simp.op}
        As usual, we let $\Del+$ denote the algebraists' (skeletal) category of all finite ordinals and order preserving maps between them and let $\Del$ denote the topologists' full subcategory of non-zero ordinals. Following tradition, we write $[n]$ for the ordinal $n+1$ as an object of $\Del+$ and refer to arrows of $\Del+$ as {\em simplicial operators}.  We will generally use lower case Greek letters $\alpha,\beta,\gamma\colon[m]\to[n]$ to denote simplicial operators. We will also use the following standard notation and nomenclature throughout:
        \begin{itemize}
            \item The injective maps in $\Del+$ are called {\em face operators}. For each $j\in[n]$,  $\face_n^j\colon[n-1]\to[n]$ denotes the {\em elementary face operator\/} distinguished by the fact that its image does not contain the integer $j$.
            \item The surjective maps in $\Del+$ are called {\em degeneracy operators}. For each $j\in[n]$, we write $\degen_n^j\colon[n+1]\to[n]$ to denote the {\em elementary degeneracy operator} determined by the property that two integers in its domain map to the integer $j$ in its codomain.
      %  \item We write $\tdegen_n\colon[n]\to[0]$ and $\aug_n\colon[-1]\to[n]$ to denote the unique such simplicial operators.
        \end{itemize}
        Unless doing so would introduce an ambiguity, we tend to reduce notational clutter by dropping the subscripts of these elementary operators.
    \end{ntn}

    \begin{ntn}[(augmented) simplicial sets]\label{ntn:simplicial-sets}
        Let $\sSet$ denote the functor category $\Set^{\Del\op}$,  the category of all {\em simplicial sets\/} and {\em simplicial maps\/} between them. 

        If $X$ is a simplicial set then $X_n$ will denote its value at the object $[n]\in\Del$, called its set of $n$-simplices, and if $f\colon X\to Y$ is a simplicial map then $f_n\colon X_n\to Y_n$ denotes its component at $[n]\in\Del$.

        It is common to think of simplicial sets as being right $\Del$-sets and use the (right) action notation $x\cdot\alpha$ to denote the element of $X_n$ obtained by applying the image under $X$ of a simplicial operator $\alpha \colon [n] \to [m]$ to an element $x\in X_m$. Exploiting this notation,  the functoriality of a simplicial set $X$ may be expressed in terms of the familiar action axioms $(x\cdot\alpha)\cdot\beta = x\cdot(\alpha\circ\beta)$ and $x\cdot\id = x$ and the naturality of a simplicial map $f\colon X\to Y$ corresponds to the action preservation identity $f(x\cdot\alpha)=f(x)\cdot\alpha$. 

A subset $Y\subseteq X$ of a simplicial set $X$ is said to be a {\em simplicial subset\/} of $X$ if it is closed  under  right action by all simplicial operators. If $S$ is a subset of $X$ then there is a smallest simplicial subset of $X$ containing $S$, the simplicial subset of $X$ {\em generated by\/} $S$.

We adopt the same notational conventions for {\em augmented simplicial sets}, objects of the functor category $\Set^{\Del+\op}$, which we denote by $\asSet$.
    \end{ntn}

    \begin{rec}[augmentation]\label{rec:augmentation}
       There is a canonical forgetful functor $\asSet\to\sSet$ constructed by pre-composition with the inclusion functor $\Del\inc\Del+$. Rather than give this functor a name, we prefer instead to allow context to determine whether an augmented simplicial set should be regarded as being a simplicial set by forgetting its augmentation. 

        Left and right Kan extension along $\Del\inc\Del+$ provides left and right adjoints to this forgetful functor, both of which are fully faithful. The left adjoint gives a simplicial set $X$ the {\em initial augmentation\/} $X\to\cpts X$ by its set of path components. The right adjoint gives $X$ the {\em terminal augmentation\/} $X\to {*}$ by the singleton set. We say that an augmented simplicial set is {\em initially\/} (resp.~{\em terminally}) {\em augmented\/} if the counit (resp.~unit) of the appropriate adjunction is  an isomorphism.

        Each $(-1)$-simplex $x$ in an augmented simplicial set $X$ is associated with a terminally augmented sub-simplicial set consisting of those simplices whose $(-1)$-face is $x$. These components are mutually disjoint and their disjoint union is the whole of $X$, providing a canonical decomposition of $X$ as a disjoint union of terminally augmented simplicial sets.
    \end{rec}

    \begin{ntn}[some important (augmented) simplicial sets]

        We fix notation for some important (augmented) simplicial sets. 
        \begin{itemize}
            \item The {\em standard $n$-simplex\/} $\Del^n$ is defined to be the contravariant representable on the ordinal $[n]\in\Del+$. In other words, $\Del^n_m$ is the set of simplicial operators $\alpha\colon[m]\to[n]$ which are acted upon by pre-composition.
            \item The {\em boundary\/} of the standard $n$-simplex $\boundary\Del^n$ is defined to be the simplicial subset of $\Del^n$ consisting of those simplicial operators which are not degeneracy operators. This is the simplicial subset of $\Del^n$ generated by the set of its $(n-1)$-dimensional faces.
            \item The {\em $(n,k)$-horn\/} $\Horn^{n,k}$ (for $n\in\mathbb{N}$ and $0\leq k\leq n$) is the simplicial subset of $\Del^n$ generated by the set $\{\face^i_n\mid 0\leq i\leq n \text{ and } i\neq k\}$ of $(n-1)$-dimensional faces.  Alternatively, we can describe $\Horn^{n,k}$ as the simplicial subset of those simplicial operators $\alpha\colon[m]\to[n]$ for which $\im(\alpha)\cup\{k\}\neq[n]$.
            \item We say that $\Horn^{n,k}$ is an {\em inner horn\/} if $0<k<n$; if $k=0$ or $k=n$, it is an {\em outer horn}.
        \end{itemize}

We have overloaded our notation above to refer interchangeably to objects of $\sSet$ or $\asSet$. There is no ambiguity since in each case the underlying simplicial set of one of these objects in $\asSet$ is  the corresponding object in $\sSet$. As an augmented simplicial set each of the objects above is terminally augmented.

        When $\alpha\colon [n]\to [m]$ is a simplicial operator we use the same symbol to denote the corresponding simplicial map $\alpha\colon\Del^n\to\Del^m$ which acts by post-composing with $\alpha$. In particular, $\face^j_n\colon\Del^{n-1}\to\Del^n$, $\degen^j_n\colon\Del^{n+1}\to\Del^n$,  $\tdegen_n\colon\Del^n\to\Del^0$, and $\aug_n:\Del^{-1}\to\Del^n$ denote the simplicial maps corresponding to the simplicial operators introduced in~\ref{ntn:simp.op} above.
    \end{ntn}

    \begin{ntn}[faces of \protect{$\Del^n$}]\label{ntn:faces-by-vertices}
It is useful to identify a non-degenerate simplex in the standard $n$-simplex $\Delta^n$ simply by naming its vertices. We use the notation $\fbv{v_0,v_1,v_2,...,v_m}$ to denote the simplicial operator $[m]\to [n]$ which maps $i\in[m]$ to $v_i\in[n]$. Let $\Del^{\fbv{v_0,v_1,...,v_m}}$ denote the smallest simplicial subset of $\Del^n$ which contains the face $\fbv{v_0,v_1,...,v_m}$.

      %More generally, in the case where $X$ is a simplicial set whose simplices are {\em determined by their vertices}, in the sense that if $x, x'\in X$ are simplices with the same $0$-faces then $x=x'$, we shall adopt the convention of annotating an $n$-simplex of $X$ by listing its $0$-faces in order $\fbv{x_0,x_1,...,x_n}$. The nerves of pre-ordered sets have simplices which are determined by their vertices and the class of all such simplicial sets is closed under product and subset. Indeed, any simplicial set in this class may be obtained, up to isomorphism, as a simplicial subset of a nerve of a pre-ordered set.
    \end{ntn}
    
    \begin{ntn}[internal hom]\label{ntn:simplicial-hom-space}
    Like any presheaf category, the category of simplicial sets is cartesian closed. We write $Y^X$ for the exponential, equivalently the {\em internal hom\/} or simply {\em hom-space}, from $X$ to $Y$. By the defining adjunction and the Yoneda lemma, an $n$-simplex in $Y^X$ is a simplicial map $X \times \Del^n \to Y$. Its faces and degeneracies are computed by pre-composing with the appropriate maps between the representables.
    \end{ntn}
    
    \subsection{Quasi-categories}

    \begin{defn}[quasi-categories]
      A {\em quasi-category} is a simplicial set $A$ which possesses the {\em right lifting property\/} with respect to all {\em inner horn inclusions\/} $\Horn^{n,k}\inc\Del^n$ ($n \geq 2$, $0<k<n$). A simplicial map between quasi-categories will be called a {\em functor}. We write $\qCat$ for the full subcategory of $\sSet$ consisting of the quasi-categories and functors.
    \end{defn}

    \begin{rec}[the homotopy category]\label{rec:hty-category}
      Let $\Cat$ denote the category of all small categories and functors between them. There is an adjunction
      \begin{equation*}
        \adjdisplay \ho-|\nrv:\Cat->\sSet. 
      \end{equation*}
      given by the nerve construction and its left adjoint. Since the nerve construction is fully faithful, we typically regard $\Cat$ as being a full subcategory of $\sSet$ and elide explicit mention of the functor $\nrv$. The nerve of any category is a quasi-category, so we may equally well regard $\Cat$ as being a reflective full subcategory of $\qCat$.

 When $A$ is a quasi-category, $\ho{A}$ is sensibly called its {\em homotopy category}; it has:
 \begin{itemize}
 \item \textbf{objects} the 0-simplices of $A$,
\item \textbf{arrows} equivalence classes of 1-simplices of $A$ which share the same boundaries, and
\item \textbf{composition} determined by the property that $k = g f$ in $\ho{A}$ if and only if there exists a 2-simplex $a$ in $A$ with $a\cdot\face^0=g$, $a\cdot\face^2=f$ and $a\cdot\face^1=k$.
\end{itemize}
See, e.g., \cite[\S 1.2.3]{Lurie:2009fk}. To emphasise the analogy with categories, we draw a 1-simplex $f$ of $A$ as an arrow with domain $f \cdot \face^1$ and codomain $f \cdot \face^0$. With these conventions, a 2-simplex $a$ of $A$ witnessing the identity $k = g f$ in $\ho{A}$ takes the form:
  \begin{equation*}
  \xymatrix{ & \cdot \ar[dr]^g \ar@{}[d]|(.6){a} \\ \cdot \ar[ur]^f \ar[rr]_k & & \cdot}    
\end{equation*} 
      
  Identity arrows in $\ho{A}$ are represented by degenerate 1-simplices. Hence, the composition axiom defines what it means for a parallel pair of 1-simplices  $f,f' \colon x \to y$  to represent the same morphism in $\ho{A}$: this is the case if and only if there exist 2-simplices of each of  (equivalently, any one of) the following forms
  \begin{equation}\label{eq:homotopy-of-1-simplices} \xymatrix{ & y \ar[dr]^{y\cdot \degen^0} & & & y \ar[dr]^{y\cdot\degen^0}& & & x \ar[dr]^f & & & x \ar[dr]^{f'} \\ x \ar[ur]^f \ar[rr]_{f'} & & y & x \ar[ur]^{f'} \ar[rr]_f & & y & x \ar[ur]^{x \cdot\degen^0} \ar[rr]_{f'} & & y & x \ar[ur]^{x\cdot\degen^0} \ar[rr]_f & & y}\end{equation} 
    In this case, we say that $f$ and $f'$ are {\em homotopic relative to their boundary}.
      
Both of the functors $\ho$ and $\nrv$ are {\em cartesian}, preserving all finite products; see \cite[B.0.15]{Joyal:2008tq} or \cite[18.1.1]{Riehl:2014kx}.    
    \end{rec}

\begin{ntn}
 Let $\catone$, $\cattwo$, or $\catthree$  denote the one-point \fbox{$\bullet$}, \emph{generic arrow} \fbox{$\bullet\to\bullet$}, and \emph{generic composed pair} \fbox{$\bullet\to\bullet\to\bullet$} \emph{categories} respectively. Under our identification of categories with their nerves, these categories are identified with the standard simplices $\Del^0$, $\Del^1$, and $\Del^2$ respectively.
\end{ntn}

The terms {\em model category\/} and {\em model structure\/}  refer to closed model structures in the sense of Quillen~\cite{Quillen:1967:Model}.

    \begin{rec}[the model category of quasi-categories]\label{rec:qmc-quasicat}
    The quasi-categories are precisely the fibrant-cofibrant objects in a combinatorial model structure on simplicial sets due to Joyal, a proof of which can be found in \cite[\S 6.5]{Verity:2007:wcs1}.
       For our purposes here, it will be enough to recall that Joyal's model structure is completely determined by the fact that it has: 
      \begin{itemize}
      \item {\em weak equivalences\/}, which are  those simplicial maps $w\colon X\to Y$ for which each functor $\ho(A^w)\colon\ho(A^Y)\to\ho(A^X)$ is an equivalence of categories for all quasi-categories $A$,
      \item {\em cofibrations\/}, which are simply the injective simplicial maps. In particular all objects are cofibrant in this model structure, and
      \item {\em fibrations between fibrant objects\/}, which are those functors of quasi-categories which possess the right lifting property with respect to:
      \begin{itemize}
      \item all inner horn inclusions $\Horn^{n,k}\inc\Del^n$ ($n\geq 2$, $0<k<n$), and
      \item (either one of) the monomorphisms $\Del^0\inc \iso$, where $\iso$ denotes the {\em generic isomorphism category\/} $\bullet\cong\bullet$.
      \end{itemize}
      To emphasise the analogy with 1-category theory, we call the fibrations between fibrant objects {\em isofibrations}.
      \end{itemize}
    \end{rec}

The Joyal model structure for quasi-categories is \emph{cartesian}, the meaning of which requires the following construction.

\begin{rec}[Leibniz constructions]\label{rec:leibniz}
If we are given a bifunctor
$ \otimes \colon \lcat{K} \times \lcat{L} \to \lcat{M}$ 
  whose codomain  possesses all pushouts, then the {\em Leibniz construction\/} provides us with a bifunctor $ \leib\otimes \colon \lcat{K}\mapcat \times \lcat{L}\mapcat \to \lcat{M}\mapcat$
  between arrow categories, which carries a pair of objects $f \in \lcat{K}\mapcat$ and $g \in \lcat{L}\mapcat$ to an object $f\leib\otimes g \in \lcat{M}\mapcat$ defined to be the map induced by the universal property of the pushout in the following diagram:
  \begin{equation}
  \xymatrix@=2em{
    {K\otimes L} \ar[d]_{K\otimes g} \ar[r]^{f\otimes L} &
    {K'\otimes L} \ar[d] \ar@/^4ex/[ddr]^{K'\otimes g} & \\
    {K\otimes L'} \ar[r] \ar@/_4ex/[rrd]_{f\otimes L'} &
    {(K'\otimes L) \cup_{K\otimes L} (K\otimes L') } \poexcursion
    \ar@{-->}[dr]_{f\leib\otimes g} & \\
    & & {K'\otimes L'}
  }
  \end{equation}
  The action of this functor on the arrows of $\lcat{K}\mapcat$ and $\lcat{L}\mapcat$ is the canonical one induced by the functoriality of $\otimes$ and the universal property of the pushout in the diagram above. In the case where the bifunctor $\otimes$ defines a monoidal product, the Leibniz bifunctor $\leib\otimes$ is frequently called the {\em pushout product}.
In the context of a bifunctor $\hom \colon \lcat{K}\op \times \lcat{L} \to \lcat{M}$, the dual construction, defined using pullbacks in $\lcat{M}$, is preferred.  We refer the reader to \cite[\S\ref*{reedy:sec:Leibniz-Reedy}]{RiehlVerity:2013kx} for a full account of this construction and its properties. %In particular, in the few instances where Reedy category theory is invoked in proofs appearing below, we make use of the notational conventions established therein.  
\end{rec}

\begin{rec}[cartesian model categories]\label{rec:cart-modcat}
The cartesianness of the Joyal model structure may be formulated in the following equivalent forms:
  \begin{enumerate}
    \item If $i\colon X\tcof Y$ and $j\colon U\tcof V$ are both cofibrations (monomorphisms) then so is their {\em Leibniz product\/} $i\leib\times j\colon (Y\times U)\cup_{X\times U} (X\times V)\tcof (Y\times V)$. Furthermore, if $i$  or $j$ is a trivial cofibration then so is $i\leib\times j$.
    \item If $i\colon X\tcof Y$ is a cofibration (monomorphism) and $p\colon A\tfib B$ is a fibration then their {\em Leibniz hom\/} $\leib\hom(i,p)\colon A^Y\tfib B^Y\times_{B^X} A^X$ is also a fibration. Furthermore, if $i$ is a trivial cofibration or $p$ is a trivial fibration then $\leib\hom(i,p)$ is also a trivial fibration.
  \end{enumerate}
  In particular, if $A$ is a quasi-category then we may apply the second of these formulations to the unique isofibration $!\colon A\to 1$ and monomorphisms $\emptyset\inc X$ and $i\colon X\inc Y$ to show that $A^X$ is again a quasi-category and that the pre-composition functor $A^i\colon A^Y\tfib A^X$ is an isofibration. 
\end{rec}

\begin{obs}[closure properties of isofibrations]\label{obs:isofibration-closure}
As a consequence of \ref{rec:qmc-quasicat} and \ref{rec:cart-modcat}, the isofibrations enjoy the following closure properties:
\begin{itemize}
\item  The isofibrations are closed under products, pullbacks, retracts, and transfinite limits of towers (as fibrations between fibrant objects).
\item The isofibrations are also closed under the Leibniz hom $\leib\hom(i,-)$ for any monomorphism $i$ and, in particular, under exponentiation $(-)^X$ for any simplicial set $X$ (as fibrations between fibrant objects in a cartesian model category).
\end{itemize}
\end{obs}

\subsection{Isomorphisms and marked simplicial sets}

  \begin{defn}[isomorphisms in quasi-categories]\label{defn:equivalences} 
    When $A$ is a quasi-category, we say that a 1-simplex $a\in A_1$ is an {\em isomorphism\/} if and only if the corresponding arrow of its homotopy category $\ho{A}$ is an isomorphism in the usual sense.
\end{defn}

Others use the term ``equivalences'' for the isomorphisms in a quasi-category, but we believe our terminology is less ambiguous: no stricter notion of isomorphism exists. 

When working with isomorphisms in quasi-categories, it will sometimes be convenient to work in the category of {\em marked simplicial sets\/} as defined by Lurie~\cite{Lurie:2009fk}.

    \begin{defn}[marked simplicial sets]
      A {\em marked simplicial set} $X$ is a simplicial set equipped with a specified subset of \emph{marked} $1$-simplices $mX\subseteq X_1$ containing all the degenerate 1-simplices. A map of marked simplicial sets is a map of underlying simplicial sets that carries marked 1-simplices to marked 1-simplices. While the category $\msSet$ of marked simplicial sets is not quite as well behaved as $\sSet$ it is nevertheless a \emph{quasitopos}, which implies that it is complete, cocomplete, and (locally) cartesian closed (see \cite[Observation 11]{Verity:2007:wcs1} and \cite{Street:2003:WomCats}).

      The functor $\msSet\to\sSet$ which forgets markings has both a left and a right adjoint. This left adjoint, dubbed {\em flat\/} by Lurie, makes a simplicial set $X$ into a marked simplicial set $X^\flat$ by giving it the minimal marking in which only the degenerate $1$-simplices are marked. Conversely, this right adjoint, which Lurie calls {\em sharp}, makes $X$ into a marked simplicial set $X^\sharp$ by giving it the maximal marking in which all $1$-simplices are marked. If $X$ is already a marked simplicial set then we will use the notation $X^\flat$ and $X^\sharp$ for the marked simplicial sets obtained by applying the flat or sharp construction (respectively) to the underlying simplicial set of $X$.

      In general, we will identify simplicial sets with their minimally marked variants, allowing us to extend the notation introduced above to the marked context. Any variation to this rule will be commented upon as we go along.
    \end{defn}
    
        \begin{rmk}[stratified simplicial sets]
Earlier authors, including Roberts~\cite{Roberts:1978:Complicial}, Street~\cite{Street:1987:Oriental}, and Verity~\cite{Verity:2008sr,Verity:2007:wcs1}, have studied a more general notion of {\em stratification}. A {\em stratified simplicial set\/} is again a simplicial set $X$ equipped with a specified subset of simplices which, in that context, are said to be {\em thin}. A stratification may contain simplices of arbitrary dimension and it must again contain all degenerate simplices. Stratifications are used to build structures called {\em complicial sets\/}, which model homotopy coherent higher categories in much the way that quasi-categories model homotopy coherent categories.
    \end{rmk}

    \begin{rec}[products and exponentiation]\label{rec:marked-prod-exp}
The product in $\msSet$ of marked simplicial sets $X$ and $Y$ is formed by taking the product of underlying simplicial sets and marking those $1$-simplices $(x,y)\in X\times Y$ which have $x$ marked in $X$ and $y$ marked in $Y$.

      An exponential (internal hom) $Y^X$ in marked simplicial sets has $n$-simplices which correspond to maps $k\colon X\times\Del^n\to Y$ of marked simplicial sets and has marked $1$-simplices those $k$ which extend along the canonical inclusion $X\times\Del^1\inc X\times(\Del^1)^\sharp$ to give a (uniquely determined) map $k'$ \[\xymatrix{ X \times \Del^1 \ar[r]^-k \ar@{_(->}[d] & Y \\ X \times (\Del^1)^\sharp \ar[ur]_-{k'}}\] That is, a marked 1-simplex in $Y^X$ is a map $k'\colon X \times (\Del^1)^\sharp \to Y$ of marked simplicial sets; see \cite[\S 3.1.3]{Lurie:2009fk}.
     The only $1$-simplices which are not marked in $X\times\Del^1$ but are marked in $X\times(\Del^1)^\sharp$ are pairs of the form $(x,\id_{[1]})$ in which $x$ is marked in $X$. It follows that a marked simplicial map $k\colon X\times\Del^1\to Y$ extends along $X\times\Del^1\inc X\times(\Del^1)^\sharp$, and thus represents a marked $1$-simplex in $Y^X$, if and only if for all marked $1$-simplices $x$ in $X$ the $1$-simplex $k(x,\id_{[1]})$ is marked in $Y$.
    \end{rec}

    \begin{rec}[isomorphisms and markings]\label{rmk:equiv-markings}
A quasi-category $A$ becomes a marked simplicial set $A^\natural$ with the {\em natural marking}, under which a 1-simplex is marked if and only if it is an isomorphism. When we regard an object as being a quasi-category in the marked setting we will always assume that it carries the natural marking without comment. A functor $f\colon A\to B$ between quasi-categories automatically preserves natural markings simply because the corresponding functor $\ho(f)\colon\ho{A}\to\ho{B}$ preserves isomorphisms.
\end{rec}

\begin{ntn}
      In this context it is useful to adopt the special marking convention for horns ($n \geq 1$, $0\leq k \leq n$) under which we
      \begin{itemize}
      \item write $\Del^{n:k}$ for the marked simplicial set obtained from the standard minimally marked simplex $\Del^n$ by also marking the edge $\fbv{0,1}$ in the case $k=0$ and marking the edge $\fbv{n-1,n}$ in the case $k=n$,
      \item inherit the marking of the horn $\Horn^{n,k}$ from that of $\Del^{n:k}$, and
      \item use $\Horn^{n,k}\inc\Del^{n:k}$ to denote the marked inclusion of this horn into its corresponding specially marked simplex.
      \end{itemize}
\end{ntn}

      Using these conventions we may recast Joyal's ``special horn filler'' result \cite[1.3]{Joyal:2002:QuasiCategories} simply as follows.

\begin{prop}[Joyal]\label{prop:joyal-special-horn} A naturally marked quasi-category has the right lifting property with respect to all marked horn inclusions $\Horn^{n,k}\inc\Del^{n:k}$, for $n\geq 1$ and $0\leq k \leq n$.
\end{prop}

 An important corollary is that a Kan complex is precisely a quasi-category in which every 1-simplex is an isomorphism \cite[1.4]{Joyal:2002:QuasiCategories}.

    \begin{rec}[the model structure of naturally marked quasi-categories]\label{rec:qmc-quasi-marked}
There is a model structure on the category of marked simplicial sets whose fibrant-cofibrant objects are precisely the naturally marked quasi-categories (see Lurie~\cite[\S 3.1]{Lurie:2009fk} or Verity~\cite[\S 6.5]{Verity:2007:wcs1}). This model category is  combinatorial and cartesian and is completely characterised by the fact that it has:
      \begin{itemize}
        \item {\em weak equivalences\/} which are those maps $w\colon X\to Y$ of marked simplicial sets for which $\ho(A^w)\colon\ho(A^Y)\to\ho(A^X)$ is an equivalence of categories for all (naturally marked) quasi-categories $A$,
        \item {\em cofibrations\/} which are simply the injective maps of marked simplicial sets, and
        \item {\em fibrations between fibrant objects\/} which are the isofibrations of naturally marked quasi-categories.
      \end{itemize}
      Here, the exponential $A^X$ is the internal hom in the category of marked simplicial sets $\msSet$. The functor $\ho\colon\msSet\to\Cat$ is the left adjoint to the nerve functor $\nrv\colon\Cat\to\msSet$ which carries a category $\scat{C}$ to the marked simplicial set whose underlying simplicial set is the usual nerve and in which a $1$-simplex is marked if and only if it is an isomorphism in $\scat{C} \cong \ho{\scat{C}}$. The left adjoint $h$ sends a marked simplicial set $X$ to the localisation of its homotopy category $hX$ at the set of marked edges. Note that in the case of a naturally marked quasi-category $A^\natural$, $h(A^\natural) = hA$, the usual homotopy category of the quasi-category.

      By \cite[7.14]{Joyal:2007kk}, a cofibration is a weak equivalence if and only if it has the left lifting property with respect to the fibrations between fibrant objects.
        In particular, in this model structure all of the special marked horn inclusions $\Horn^{n,k}\inc\Del^{n:k}$ ($n \geq 1$, $0\leq k\leq n$) and the inclusion $(\Del^1)^\sharp\inc\iso$ of the marked 1-simplex into the naturally marked isomorphism category are trivial cofibrations (see \cite[B.10, B.15]{DuggerSpivak:2011ms}). This proves that an isomorphism $\Del^1 \to A$ in a quasi-category may always be extended to a functor $\iso \to A$ \cite[1.6]{Joyal:2002:QuasiCategories}.
\end{rec}

  \begin{obs}[natural markings, internal homs, and products]\label{obs:nat-mark-homs}
The product of two naturally marked quasi-categories is again a naturally marked quasi-category.
By cartesianness of the marked model structure, if $A$ is a naturally marked quasi-category and $X$ is any marked simplicial set then the exponential $A^X$ is again a naturally marked quasi-category.    In summary, the fully faithful natural marking functor $\natural\colon\qCat\to\msSet$ is a cartesian closed functor, in the sense that it preserves products and internal homs.
    \end{obs}

  The content of observation \ref{obs:nat-mark-homs} is  more profound than one might initially suspect. It might be summarised by the slogan ``a natural transformation of functors is an isomorphism if and only if it is a {\em pointwise isomorphism}''. The precise meaning of this slogan is encoded in the following result.

            \begin{lem}[pointwise isomorphisms are isomorphisms]\label{lem:pointwise-equiv}
Let $X$ be a marked simplicial set and let $A$ be a naturally marked quasi-category.  A $1$-simplex $k\colon X\times\Del^1\to A$ is marked in $A^X$ if and only if for all $0$-simplices $x$ in $X$ the $1$-simplex $k(x\cdot\degen^0,\id_{[1]})$ is marked in $A$.
       \end{lem}
Here is the intuition for this result. The component of a map $k \colon X \times \Del^1 \to A$ at a 1-simplex $f \colon a \to b$ in $X$      is a diagram $\Del^1 \times \Del^1 \to A$ \begin{equation}\label{eq:pointwise-equivalence-square}\xymatrix@=35pt{ \cdot \ar[d]_{k(f,\face^1)} \ar[r]^{k(a,\id_{[1]})} \ar[dr]|{k(f,\id_{[1]})} & \cdot \ar[d]^{k(f,\face^0)} \\ \cdot \ar[r]_{k(b,\id_{[1]})} & \cdot}\end{equation} If $f$ is marked and $k$ is a marked map, then the verticals are marked in $A$. If $A$ is a naturally marked quasi-category, then if the horizontals, the ``components'' of $k$, are marked, then so is the diagonal edge, simply because isomorphisms compose. If this is the case for all marked 1-simplices $f$, then $k$ is marked in $A^X$ by the definition of the internal hom in $\msSet$.
               
      \begin{proof}
As recalled in~\ref{rec:marked-prod-exp}, $k$ is a marked $1$-simplex in $A^X$ if and only if $k(f,\id_{[1]})$ is marked in $A$ for all marked edges $f$ of $X$. In particular, the edges $k(x\cdot\degen^0,\id_{[1]})$ are necessarily marked in $A$ if $k$ is marked in $A^X$. We show that this condition is sufficient to detect the marked edges $k \in (A^X)_1$.

 The $2$-simplex $(f\cdot\degen^0,\degen^1)$ of $X\times\Del^1$ can be drawn as follows:
        \[\left(    \vcenter{
        \xymatrix{ & \cdot \ar[dr]^{f} \ar@{}[d]|(.6){f \cdot \degen^0} \\ \cdot \ar@{=}[ur]^{(f \cdot \face^1)\cdot \degen^0} \ar[rr]_{f} & & \cdot} }\quad,\quad   \vcenter{
        \xymatrix{ & \cdot \ar@{=}[dr]^{\degen^0} \ar@{}[d]|(.6){\degen^1} \\ \cdot \ar[ur]^{{\id_{[1]}}} \ar[rr]_{\id_{[1]}} & & \cdot} } \right) \]
Applying $k$, the $2$-simplex $k(f\cdot\degen^0,\degen^1)$ of $A$ witnesses the fact that $k(f,\id_{[1]})$ is a composite of $k(f,\degen^0)$ and $k((f\cdot\face^1)\cdot\degen^0,\id_{[1]})$. 
        
        Now when $f$ is marked in $X$, the edge $(f,\degen^0)$ is marked in $X\times \Del^1$, so it follows that $k(f,\degen^0)$ is marked in $A$. By assumption the $1$-simplex $k((f\cdot\face^1)\cdot\degen^0,\id_{[1]})$ is also marked in $A$. The isomorphisms, that is to say naturally marked $1$-simplices, compose in $A$ simply because isomorphisms compose in the category $\ho{A}$, so it follows that $k(f,\id_{[1]})$ is marked in $A$. 
    \end{proof}

Recall that the marked edges in a naturally marked quasi-category are precisely the isomorphisms. Reinterpetting Lemma \ref{lem:pointwise-equiv} in the unmarked context, we have proven:

\begin{cor}\label{cor:pointwise-equiv}
For any quasi-category $A$ and simplicial set $X$, an edge $k \in (A^X)_1$ is an isomorphism if and only if each of its components $k(x) \in A_1$, defined by evaluating at each vertex $x \in X_0$, are isomorphisms.
\end{cor}

  \begin{obs}\label{obs:marked-arrow-subcat}
    If $A$ is a naturally marked quasi-category then pre-composition by the inclusion $\Del^1\inc(\Del^1)^\sharp$ gives rise to an inclusion $A^{(\Del^1)^\sharp}\inc A^{\Del^1}$ of naturally marked quasi-categories. Taking transposes, we see that Lemma \ref{lem:pointwise-equiv} may be recast as saying that a marked simplicial map $k\colon X\to A^{\Del^1}$ has a (necessarily unique) lift as the dotted arrow in
    \begin{equation*}
      \xymatrix@=2em{
        & {A^{(\Del^1)^\sharp}}\ar@{u(->}[d] \\
        {X}\ar[r]_{k}\ar@{-->}[ur] &
        {A^{\Del^1}}
      }
    \end{equation*}
    if and only if $k$ maps each 0-simplex $x\in X$ to an object $k(x)\in A^{\Del^1}$ which corresponds to a marked arrow of $A$. In other words, the map $A^{(\Del^1)^\sharp}\inc A^{\Del^1}$ is a fully faithful inclusion which identifies $A^{(\Del^1)^\sharp}$ with the full sub-quasi-category of $A^{\Del^1}$ whose objects are the isomorphisms in $A$.
  \end{obs}

\subsection{Join and slice}\label{subsec:join}

Particularly to facilitate comparisons between our development of the theory of quasi-categories, using the enriched category theories of 2-categories and simplicial categories, and the more traditional accounts following Joyal and Lurie, we review Joyal's slice and join constructions, introduced in \cite{Joyal:2002:QuasiCategories}. Unlike in the classical treatments, these technical combinatorial details will be of secondary importance for us, and for that reason, we encourage the reader to skip this section upon first reading, referring back only as necessary. A more leisurely account of the combinatorial work reviewed here can be found in an earlier version of this paper \cite[\S A]{RiehlVerity:2015tt-v3}.

  \begin{defn}[joins and d{\'e}calage]\label{defn:join-dec} The algebraists' skeletal category $\Del+$ of all finite ordinals and order preserving maps supports a canonical strict (non-symmetric) monoidal structure $(\Del+,\oplus,[-1])$ in which $\oplus$ denotes the {\em ordinal sum\/} given 
  \begin{itemize}
    \item for objects $[n],[m]\in\Del+$ by $[n]\oplus[m] \defeq [n+m+1]$,
    \item for arrows $\alpha\colon[n]\to[n'], \beta\colon[m]\to[m']$ by $\alpha\oplus\beta\colon[n+m+1]\to[n'+m'+1]$ defined by
  \begin{equation*}
    \alpha\oplus\beta(i) =
    \begin{cases}
    \alpha(i)& \text{if $i\leq n$,} \\
    \beta(i-n-1) + n' + 1& \text{otherwise.}
    \end{cases}
  \end{equation*}
  \end{itemize}
By Day convolution, this bifunctor extends to a (non-symmetric) monoidal closed structure $(\asSet, \join, \Del^{-1}, \dec_l, \dec_r)$ on the category of augmented simplicial sets. Here the monoidal operation $\join$ is known as the {\em simplicial join\/} and its closures $\dec_l$ and $\dec_r$ are known as the {\em left and right d{\'e}calage constructions}, respectively.   To fix handedness, we declare that for each augmented simplicial set $X$ the functor $\dec_l(X,{-})$ (resp.\ $\dec_r(X,{-})$) is right adjoint to $X\join{-}$ (resp.\ ${-}\join X$).

The join $X\join Y$ of augmented simplicial sets $X$ and $Y$ may be described explicitly as follows:
  \begin{itemize}
    \item it has simplices pairs $(x,y) \in (X\join Y)_{r+s+1}$ with $x\in X_r$, $y\in Y_s$,
    \item if $(x,y)$ is a simplex of $X\join Y$ with $x \in X_r$ and $y \in Y_s$ and $\alpha\colon[n]\to[r+s+1]$ is a simplicial operator in $\Del+$, then $\alpha$ may be uniquely decomposed as $\alpha=\alpha_1\join\alpha_2$ with $\alpha_1\colon[n_1]\to[r]$ and $\alpha_2\colon[n_2]\to[s]$, and we define $(x,y)\cdot\alpha\defeq (x\cdot\alpha_1,y\cdot\alpha_2)$. 
  \end{itemize}
If $f\colon X\to X'$ and $g\colon Y\to Y'$ are simplicial maps then the simplicial map $f\join g\colon X\join Y\to X'\join Y'$  carries the simplex $(x,y)\in X\join Y$ to the simplex $(f(x),g(y))\in X'\join Y'$. 
  \end{defn}

  \begin{defn}[d{\'e}calage and slices]\label{defn:slices}
The d{\'e}calage functors can be used to define Joyal's \emph{slice} construction for a map $f\colon X\to A$ of simplicial. 
Fixing a simplicial set $X$ and identifying the category $\sSet$  of simplicial sets with the full subcategory of terminally augmented simplicial sets in $\asSet$, we define a functor
    \begin{equation*}
      {-}\mathbin{\bar\join} X\colon \xy<0em,0em>*+{\sSet}\ar <6em,0em>*+{X\slice\sSet}\endxy\mkern30mu
      (\text{resp.\ } 
      X\mathbin{\bar\join}{-}\colon \xy<0em,0em>*+{\sSet}\ar <6em,0em>*+{X\slice\sSet}\endxy)
    \end{equation*}
    which carries a simplicial set $Y\in\sSet$ to the object ${*}\join X\colon X\cong\Del^{-1}\join X\to Y\join X$ (resp.\ $X\join{*}\colon X\cong X\join\Del^{-1} \to X\join Y$) induced by the map ${*}\colon\Del^{-1}\to Y$ corresponding to the unique $(-1)$-simplex of $Y$. This functor preserves all colimits, and thus admits a right adjoint that we now describe explicitly.
    
Observe that the $(-1)$-dimensional simplices of $\dec_r(X, A)$ (resp.\ $\dec_l(X,A)$) are in bijective correspondence with simplicial maps $f\colon X\to A$. So if we are given such a simplicial map we may, by recollection \ref{rec:augmentation}, extract the component of $\dec_r(X,A)$ (resp.\ $\dec_l(X,A)$) consisting of those simplices whose $(-1)$-face is $f$, which we denote by $\slicer{A}{f}$ (resp.\ $\slicel{f}{A}$) and call the {\em slice of $A$ over (resp.\ under) $f$}. Now it is a matter of an easy calculation to demonstrate directly that $\slicer{A}{f}$ (resp. $\slicel{f}{A}$) has the universal property required of the right adjoint to ${-}\bar\join X$ (resp.\ $X\bar\join {-}$) at the object $f\colon X\to A$ of $X\slice\sSet$.

    In other words, these d{\'e}calages admit the following canonical decompositions as disjoint unions of (terminally augmented) slices: 
    \begin{equation*}
      \dec_r(X,A)=\bigsqcup_{f\colon X\to A} (\slicer{A}{f})\mkern30mu
      \dec_l(X,A)=\bigsqcup_{f\colon X\to A} (\slicel{f}{A})
    \end{equation*}

    We think of the slice $\slicel{f}{A}$ as being the simplicial set of {\em cones under the diagram $f$\/} and we think of the dual slice $\slicer{A}{f}$ as being the simplicial set of {\em cones over the diagram $f$}.
  \end{defn}
  
  \begin{obs}[slices of quasi-categories]\label{obs:slice-and-qcats}
    A direct computation from the explicit description of the join construction given above demonstrates that the Leibniz join (see recollection~\ref{rec:leibniz}) of a horn and a boundary $(\Horn^{n,k}\inc\Del^n)\leib\join(\boundary\Del^m\inc\Del^m)$ is again isomorphic to a single horn $\Horn^{n+m+1,k}\inc\Del^{n+m+1}$. Dually the Leibniz join $(\boundary\Del^n\inc\Del^n)\leib\join(\Horn^{m,k}\inc\Del^m)$ is isomorphic to the single horn $\Horn^{n+m+1,n+k+1}\inc\Del^{n+m+1}$. 

    Combining these computations with the properties of the Leibniz construction developed in \cite[\S\ref*{reedy:sec:Leibniz-Reedy}]{RiehlVerity:2013kx}, we may show that an augmented simplicial set $A$ has the right lifting property with respect to all (inner) horn inclusions then so do the left and right d{\'e}calages $\dec_l(X,A)$ and $\dec_r(X,A)$ for any augmented simplicial set $X$. In particular, this tells us that if $f\colon X\to A$ is any map of simplicial sets and $A$ is a quasi-category then the slices $\slicel{f}{A}$ and $\slicer{A}{f}$ are also quasi-categories. 

    Working in the marked context, we may extend this result to Leibniz joins with specially marked outer horns. That then allows us to prove that if $p\colon A\to B$ is an isofibration of quasi-categories and $f\colon X\to A$ is any simplicial map then the induced simplicial maps $p\colon \slicer{A}{f}\to \slicer{B}{pf}$ and $p\colon \slicel{f}{A}\to \slicel{pf}{B}$ are also isofibrations of quasi-categories.
  \end{obs}

A variant of the join and slice constructions, also due to Joyal, is more closely related to the enriched categorical comma constructions that we will use here.

 \begin{defn}[fat join]\label{def:fat-join}
    We define the {\em fat join} of two simplicial sets $X$ and $Y$ to be the simplicial set $X\fatjoin Y$ constructed by means of the following pushout:
    \begin{equation}\label{eq:fat-join-def}
      \xymatrix@R=2em@C=4em{
        {({X}\times{Y})\sqcup({X}\times{Y})}\ar[r]^-{\pi_X\sqcup\pi_Y}
        \ar[d]_{\langle{X}\times\face^1\times{Y},
          {X}\times\face^0\times{Y}\rangle} &
        {{X}\sqcup{Y}}\ar[d] \\
        {{X}\times\Del^1\times{Y}} \ar[r] &
        {{X}\fatjoin{Y}}\poexcursion
      }
    \end{equation}
    We may extend this construction to simplicial maps in the obvious way to give us a bifunctor $\fatjoin\colon\sSet\times\sSet\to\sSet$, and it is clear that this preserves connected colimits in each variable. It does not preserve all colimits  because the coproduct bifunctor $\sqcup$ (as used in the top right hand corner of the defining pushout above) fails to preserve coproducts  in each variable (while it does preserve connected colimits). In particular, a fat join of a simplicial set $X$ with the empty simplicial set, rather than being empty, is isomorphic  to $X$ itself.

%    Now if we again identify the category of simplicial sets $\sSet$ with the full subcategory of terminally augmented simplicial sets then we may extend our fat join to a bifunctor on augmented simplicial sets. To do this we start by observing that every augmented simplicial set $X$ may be written canonically as a coproduct $\bigsqcup_{i\in I} X_i$ in $\asSet$ of terminally augmented simplicial sets. So if $X = \bigsqcup_{i\in I} X_i$ and $Y=\bigsqcup_{j\in J} Y_j$ are two augmented simplicial sets with terminally augmented components $X_i$ and $Y_j$, then we define $X\join Y$ to be the coproduct $\bigsqcup_{i\in I, j\in J} X_i\fatjoin Y_j$ in $\asSet$. We extend this to maps of augmented simplicial sets in the obvious way, using the fact that we may decompose such maps into families of maps between terminally augmented components. It is now a routine matter to verify that, when regarded as being a bifunctor on $\asSet$, the fat join does indeed preserve {\em all\/} colimits in each variable. Consequently, since $\asSet$ is a presheaf category, it follows that the fat join bifunctor on $\asSet$ has both left and right closures $\fatdec_l(X,A)$ and $\fatdec_r(X,A)$, called  {\em left and right fat d{\'e}calage\/} respectively, which notation we fix by declaring that if $X$ is an augmented simplicial set then $X\fatjoin{-}\dashv\fatdec_l(X,{-})$ and ${-}\fatjoin X\dashv\fatdec_r(X,{-})$.

The fat join of two non-empty simplicial sets $X$ and $Y$ may be described more concretely as the simplicial set obtained by taking the quotient of ${X}\times\Del^1\times{Y}$ under the simplicial congruence relating the pairs of $r$-simplices
    \begin{equation}\label{eq:fat-join-cong}
    (x,0,y)\sim (x,0,y')\quad \text{and}\quad (x,1,y) \sim (x',1,y) 
   \end{equation}
   where $0$ and $1$ denote the constant operators $[r]\to[1]$. We use square bracketed triples $[x,\beta,y]_\sim$ to denote equivalence classes under $\sim$. 
  \end{defn}

  \begin{defn}[fat slice]\label{defn:fat-slices}
    Replaying Joyal's slice construction of Definition~\ref{defn:slices}, if $X$ is a simplicial set, we may use the fat join to construct a functor
    \begin{equation*}
      {-}\mathbin{\bar\fatjoin} X\colon \xy<0em,0em>*+{\sSet}\ar <6em,0em>*+{X\slice\sSet}\endxy\mkern30mu
      (\text{resp.\ } 
      X\mathbin{\bar\fatjoin}{-}\colon \xy<0em,0em>*+{\sSet}\ar <6em,0em>*+{X\slice\sSet}\endxy)
    \end{equation*}
    which carries a simplicial set $Y\in\sSet$ to the object ${*}\fatjoin X\colon X\cong\Del^{-1}\fatjoin X\to Y\fatjoin X$ (resp.\ $X\fatjoin{*}\colon X\cong X\fatjoin\Del^{-1} \to X\fatjoin Y$). These functors admit right adjoints whose value at an object $f\colon X\to A$ of $X\slice\sSet$ is denoted $\fatslicer{A}{f}$ (resp.\ $\fatslicel{f}{A}$) and is called the {\em fat slice of $A$ over} (resp.\ \emph{under}) $f$.
    \end{defn}

  \begin{obs}[comparing join constructions]
    When $\beta\colon [n]\to [1]$ is a simplicial operator let $\hat{n}_\beta$ denote the largest integer in the set $\{-1\}\cup\{i\in[n]\mid \beta(i)=0\}$  and let $\check{n}_\beta=n-1-\hat{n}_\beta$. Define an associated pair $\hat\beta\colon[\hat{n}_\beta]\to[n]$ and $\check\beta\colon[\check{n}_\beta]\to[n]$ of simplicial face operators in $\Del+$ by $\hat\beta(i) = i$ for all $i\in[\hat{n}_\beta]$ and $\check\beta(j)=j+\hat{n}_\beta + 1$ for all $j\in[\check{n}_\beta]$. 
    
    Now if $X$ and $Y$ are (terminally augmented) simplicial sets we may define a map $\bar{s}^{X,Y}$ which carries an $n$-simplex $(x,\beta,y)$ of $X\times \Del^1\times Y$ to the $n$-simplex $(x\cdot\hat\beta,y\cdot\check\beta)$ of $X\join Y$. A straightforward calculation demonstrates that this map commutes with the simplicial actions on these sets and is thus a simplicial map. Furthermore, the family of simplicial maps $\bar{s}^{X,Y}\colon X\times \Del^1\times Y \to X\join Y$ is natural in $X$ and $Y$.

    Of course, since $X$ and $Y$ are terminally augmented, we also have canonical maps $l^{X,Y}\colon X\cong X\join\Del^{-1}\to X\join Y$ and $r^{X,Y}\colon Y\cong\Del^{-1}\join Y\to X\join Y$ and we may assemble all these maps together into a commutative square 
    \begin{equation}\label{eq:join-comp-def}
      \xymatrix@R=2em@C=4em{
        {({X}\times{Y})\sqcup({X}\times{Y})}\ar[r]^-{\pi_X\sqcup\pi_Y}
        \ar[d]_{\langle{X}\times\face^1\times{Y},
          {X}\times\face^0\times{Y}\rangle} &
        {{X}\sqcup{Y}}\ar[d]^{\langle l^{X,Y}, r^{X,Y}\rangle} \\
        {{X}\times\Del^1\times{Y}} \ar[r]_{\bar{s}^{X,Y}} &
        {{X}\join{Y}}
      }
    \end{equation}
    whose maps are all natural in $X$ and $Y$. Using the defining universal property of fat join, as given in~\eqref{eq:fat-join-def}, these squares induce maps $s^{X,Y}\colon X\fatjoin Y\to X\join Y$ which are again natural in $X$ and $Y$. Should we so wish, we may now take suitable coproducts of these maps to canonically extend this family of simplicial maps to a natural transformation between the extended fat join and join bifunctors on augmented simplicial sets.

    More explicitly, if $n,m\geq 0$, then $\bar{s}^{n,m}\colon\Del^n\times\Del^1\times\Del^m\to\Del^{n+m+1}$ is the unique simplicial map determined by the (order preserving) action on vertices given by:
    \begin{equation}\label{eq:tnm-def}
      \bar{s}^{n,m}(i,j,k) = 
      \begin{cases}
        i & \text{if $j=0$, and} \\
        k+n+1 & \text{if $j=1$.}
      \end{cases}
    \end{equation}
    This takes simplices related under the congruence defined in~\eqref{eq:fat-join-def} of Definition~\ref{def:fat-join} to the same simplex and thus induces a unique map $s^{n,m}\colon\Del^n\fatjoin\Del^m\to\Del^n\join\Del^m$ on the quotient simplicial set.
  \end{obs}

  \begin{prop}\label{prop:join-fatjoin-equiv}
    For all simplicial sets $X$ and $Y$, the  map $s^{X,Y}\colon X\fatjoin Y\to X\join Y$ is a weak equivalence in the Joyal model structure.
  \end{prop}
\begin{proof} For proof see \cite[4.2.1.2]{Lurie:2009fk} or \cite[A.4.11]{RiehlVerity:2015tt-v3}.

\end{proof}

  \begin{lem}\label{lem:slices-quillen}
 For any simplicial set $X$, the slice and fat slice adjunctions
    \begin{align*}
  		\adjdisplay \textstyle X\bar\join{-}-|{} :X\slice\sSet->\sSet. & 
      \mkern20mu & 
  		\adjdisplay \textstyle {-}\bar\join X-|{} :X\slice\sSet->\sSet. \\
  		\adjdisplay \textstyle X\bar\fatjoin{-}-|{}:X\slice\sSet->\sSet. & 
      \mkern20mu & 
  		\adjdisplay \textstyle {-}\bar\fatjoin X-|{}:X\slice\sSet->\sSet.
    \end{align*}
    of Definitions~\ref{defn:slices} and~\ref{defn:fat-slices} are Quillen adjunctions with respect to the Joyal model structure on $\sSet$ and the corresponding sliced model structure on $X\slice\sSet$.
  \end{lem}
  
  \begin{proof}
    By~\cite[7.15]{Joyal:2007kk} it is enough to check that in each of these adjunctions the left adjoint preserves cofibrations and the right adjoint preserves fibrations between fibrant objects. From the explicit descriptions of the join and fat join, it is not difficult to see that the left adjoints preserve monomorphisms of simplicial sets. Observations~\ref{obs:slice-and-qcats} tells us that if $p \colon A \to B$ is an isofibration of quasi-categories and $f \colon X \to A$ is any simplicial map, then the induced simplicial maps $p\colon \slicer{A}{f}\to \slicer{B}{pf}$ and $p\colon \slicel{f}{A}\to \slicel{pf}{B}$ are also isofibrations of quasi-categories. The corresponding result for fat slices is a special case of Lemma \ref{lem:comma-obj-maps} below.%SAY MORE
  \end{proof}

  Finally, we arrive at the advertised comparison result relating the slice and fat slice constructions.
  
  \begin{prop}[slices and fat slices of a quasi-category are equivalent]\label{prop:slice-fatslice-equiv}
    Suppose that $X$ is any simplicial set, that $\sSet$ carries the Joyal model structure, and that $X\slice\sSet$ carries the associated sliced model structure. Then the comparison maps $s^{X,Y}\colon X\fatjoin Y\to X\join Y$ furnish us with natural transformations $s^{X,{-}}\colon X\bar\fatjoin{-}\to X\bar\join{-}$ and $s^{{-},X}\colon {-}\bar\fatjoin X\to {-}\bar\join X$ which are pointwise weak equivalences. Furthermore, these induce natural transformations  on corresponding right adjoints, whose components $e_l^f\colon \slicel{f}{A}\to \fatslicel{f}{A} $ and $e_r^f\colon \slicer{A}{f}\to \fatslicer{A}{f}$ at an object $f\colon X\to A$ of $X\slice\sSet$ are equivalences of quasi-categories whenever $A$ is a quasi-category.
  \end{prop}
  
  \begin{proof}
    The assertions involving left adjoints were proven in Proposition~\ref{prop:join-fatjoin-equiv}. The Quillen adjunctions established in Lemma~\ref{lem:slices-quillen} allow us to apply the standard result in model category theory~\cite[1.4.4]{Hovey:1999fk} that a natural transformation between left Quillen functors has components which are weak equivalences at each cofibrant object (which fact we have already established) if and only if the  induced natural transformation between the corresponding right Quillen functors has components which are weak equivalences at each fibrant object. Now simply observe that an object $f\colon X\to A$ is fibrant in $X\slice\sSet$ if and only if $A$ is a quasi-category. 
  \end{proof}
  
  \begin{rmk}\label{rmk:map-slices}
    Suppose that $f\colon B\to A$ and $g\colon C\to A$ are two simplicial maps. We generalise our slice and fat slice notation by using $\slicer{g}{f}$, $\fatslicer{g}{f}$, $\slicel{f}{g}$ and $\fatslicel{f}{g}$ to denote the objects constructed in the following pullback diagrams
    \begin{equation}
      \xymatrix@=2em{
        {\slicer{g}{f}} \pbexcursion\ar[r]\ar[d] & {\slicer{A}{f}}\ar[d]^\pi \\
        {C}\ar[r]_g & {A}
      }
      \mkern50mu
      \xymatrix@=2em{
        {\fatslicer{g}{f}} \pbexcursion\ar[r]\ar[d] & {\fatslicer{A}{f}}\ar[d]^\pi \\
        {C}\ar[r]_g & {A}
      }
      \mkern50mu
      \xymatrix@=2em{
        {\slicel{f}{g}} \pbexcursion\ar[r]\ar[d] & {\slicel{f}{A}}\ar[d]^\pi \\
        {C}\ar[r]_g & {A}
      }
      \mkern50mu
      \xymatrix@=2em{
        {\fatslicel{f}{g}} \pbexcursion\ar[r]\ar[d] & {\fatslicel{f}{A}}\ar[d]^\pi \\
        {C}\ar[r]_g & {A}
      }
    \end{equation}
    in which the maps labelled $\pi$ denote the various canonical projection maps. We call these the slices and fat slices of $g$ over and under $f$ respectively. 
     We have isomorphisms $\slicer{g\op}{f\op} \cong (\slicel{f}{g})\op$ and $\fatslicer{g\op}{f\op} \cong (\fatslicel{f}{g})\op$. % Similarly $g\op\downarrow f\op \cong (f\downarrow g)\op$.    
    
    When $A$ is a quasi-cat\-e\-go\-ry the projection maps $\pi$ are all isofibrations that commute with the comparison equivalences $e_l^f\colon \slicel{f}{A}\to \fatslicel{f}{A}$ and $e_r^f\colon \slicer{A}{f}\to \fatslicer{A}{f}$ of Proposition \ref{prop:join-fatjoin-equiv}. These maps are equivalences  of fibrant objects in the sliced Joyal model structure on $\sSet\slice A$. Pullback along any map in a model category is always a right Quillen functor of sliced model structures, so Ken Brown's lemma tells us that the pullbacks are again equivalences.
  \end{rmk}

%!TEX root = all.tex
% ******************************************************************
% ** Title:            The 2-category theory of quasi-categories
% **                   2-Categorical Arguments
% ** Precis:        
% ** Author:           Emily Riehl and Dominic Verity
% ** Commenced:        2/3/2012
% ******************************************************************

  \section{The 2-category of quasi-categories}\label{sec:twocat}

  The full subcategory $\qCat$ of quasi-categories and functors is closed in $\sSet$ under products and internal homs. 
   It follows that $\qCat$ is cartesian closed and that it becomes a full simplicial sub-category of $\sSet$ under its usual self enrichment. We denote this self-enriched category of quasi-categories, whose simplicial hom-spaces are given by exponentiation, by $\qCat_\infty$.

  In this section, we study a corresponding (strict) 2-category of quasi-categories $\qCat_2$ first introduced by Joyal \cite{Joyal:2008tq}. This should be thought of as being a kind of quotient of $\qCat_\infty$ whose 2-cells (1-arrows in the hom-spaces) are replaced by homotopy classes of such and in which higher dimensional information in the hom-spaces is discarded.  At first blush, it might seem that such a process would destroy far too much information to be of any great use. However, much of this paper is devoted to showing, perhaps quite surprisingly, that we may develop a great deal of the elementary category theory of quasi-categories within the 2-category $\qCat_2$ alone. Our first step in this direction will be to recognise that much of this category theory may be encoded in the weak 2-universal properties of certain constructions in this 2-category.

  In this section, we introduce the 2-category $\qCat_2$ of quasi-categories and establish a few of its basic properties. In particular, we define a particular notion of weak 2-limit appropriate to this context and show that $\qCat_2$ admits certain weak 2-limit constructions. In later sections, we use the structures introduced here to transport classical categorical proofs into the quasi-categorical context. 

	\subsection{Relating 2-categories and simplicially enriched categories}
	
\begin{ntn}[simplicial categories and 2-categories]
  The category of simplicial sets $\sSet$ is complete, cocomplete, and cartesian closed, so in particular it supports a well developed enriched category theory. We refer to $\sSet$-enriched categories simply as {\em simplicial categories\/} and the enriched functors between them as {\em simplicial functors}.

In a simplicial category $\tcat{C}$, we call the $n$-simplices of one of its simplicial hom-spaces $\tcat{C}(A,B)$ its {\em $n$-arrows\/} from $A$ to $B$. The composition operation of $\tcat{C}$ restricts to make the graph of the objects and $n$-arrows of $\tcat{C}$ into a category which we shall call $\tcat{C}_n$, for which $\tcat{C}_n(A,B) = \tcat{C}(A,B)_n$. Furthermore, if $\alpha\colon[n]\to[m]$ is a simplicial operator then its action on arrows gives rise to an identity-on-objects functor $\tcat{C}_m\to\tcat{C}_n$.

The category of all (small) categories $\Cat$ is also complete, cocomplete, and cartesian closed, so it too supports an enriched category theory. We refer to $\Cat$-enriched categories as {\em 2-categories\/} and the enriched functors between then as {\em 2-functors}. In a 2-category $\tcat{C}$, we follow convention and refer to its objects as {\em 0-cells}, the objects in its hom-categories as {\em 1-cells}, and the arrows in its hom-categories as {\em 2-cells}. 

  We refer the reader to Kelly's canonical tome~\cite{Kelly:2005:ECT} for the standard exposition of the yoga of enriched category theory. We also strongly recommend Kelly and Street~\cite{kelly.street:2} and Kelly~\cite{Kelly:1989fk} as elementary introductions to 2-categories and their attendant 2-limit notions. In particular, we encourage the reader to familiarise him- or herself with the rubric of pasting composition discussed in~\cite{kelly.street:2}.
\end{ntn}

  Recollection~\ref{rec:hty-category} reminds us that $\Cat$ may be regarded as a reflective subcategory of $\sSet$, or indeed $\qCat$, via the adjunction $\ho \dashv \nrv$: the natural map $X \to hX$ is an isomorphism if and only if $X$ is (the nerve of) a category.  The fact that $h\colon \sSet\to \Cat$ preserves binary products implies, and in fact is equivalent to, the observation that if $C$ is a category and $X$ is a simplicial set then their internal hom $C^X$ in $\sSet$ is again a category. The proof in \cite[B.0.16]{Joyal:2008tq} is as follows: there is a canonical map of simplicial sets $C^{hX} \to C^X$. Fixing $X$ and varying $C$ these maps define the components of a natural transformation between two right adjoints $\Cat\to\sSet$. This map is invertible because the transposed natural transformation $h(X\times Y) \to hX \times hY$ is an isomorphism.

Recollection~\ref{rec:cart-modcat} tells us the corresponding result for quasi-categories, this being that internal homs whose target objects are quasi-categories are themselves quasi-categories. In particular, it follows that each of the categories $\Cat$, $\qCat$, and $\sSet$ is cartesian closed and that the various inclusions of one into another preserve finite products and internal homs. In particular, we may regard the self-enriched categories $\Cat$ and $\qCat$ as being full simplicial subcategories of $\sSet$ under its self enrichment.  We write $\Cat_2$ for this 2-category of categories, regarded as a full subcategory of $\qCat_\infty$.

\begin{obs}\label{obs:simp-to-2-cats}
Using the fact that $\ho$ and $\nrv$ both preserve finite products, we may construct an induced adjunction
	\begin{equation*}
		\adjdisplay \ho_*-|\nrv_*:\twoCat->\sCat. 
	\end{equation*}
	between the categories of 2-categories and simplicial categories respectively. The functors in this adjunction are obtained by applying $\nrv$ and $\ho$  to the hom-objects of an enriched category on one side of this adjunction to obtain a corresponding enriched category on the other side. Here again $\nrv_*$ is fully faithful, so it is natural to regard $\twoCat$ as being a reflective full subcategory both of $\sCat$ and of its full subcategory $\qCat$--$\Cat$ of categories enriched in quasi-categories. Indeed, for our purposes here it suffices to consider the restricted adjunction \[ \adjdisplay \ho_* -| \nrv_* : \twoCat -> \eCat\qCat.\] 

Given a quasi-categorically enriched category $\tcat{C}$, the 2-category $\ho_*\tcat{C}$ is a quotient of sorts. The underlying unenriched categories of $\tcat{C}$ and $\ho_*\tcat{C}$ coincide, but 2-cells in $\ho_*\tcat{C}$ are homotopy classes of 1-arrows in $\tcat{C}$. These homotopy classes are defined using relations witnessed by the 2-arrows. All higher dimensional cells are discarded. On regarding $\ho_*\tcat{C}$ as a simplicially enriched category we see that the unit of the adjunction $\ho_*\dashv\nrv_*$ provides us with a canonical simplicial quotient functor $\tcat{C}\to\ho_*\tcat{C}$.
\end{obs}

  Our identification of categories with their nerves also leads us to regard 2-categories as certain special kinds of simplicial categories. Under this identification, a 1-cell (resp.\ 2-cell) in a 2-category can equally well be regarded as being a 0-arrow (resp.\ 1-arrow) in the corresponding simplicial category.  

\subsection{The 2-category of quasi-categories}

\begin{defn}[the 2-category of quasi-categories]\label{def:qCat-2}
  In particular, applying the functor $\ho_*$ to the quasi-categorically enriched category $\qCat_\infty$, we obtain an associated 2-category $\qCat_2 := h_*\qCat_\infty$ whose hom-categories are given by
  \begin{equation}
    \label{eq:qCat2homdefn} \hom'(A,B) \defeq \ho(B^A).
  \end{equation} 
Using the  description of $h$ given in \ref{rec:hty-category}, we find that the objects of $\qCat_2$ are quasi-categories; the 1-cells are maps of quasi-categories, which we have agreed to call \emph{functors}; and the 2-cells, which we shall call \emph{natural transformations}, are certain homotopy classes of 1-simplices in the internal hom $B^A$. 

  More explicitly, a 2-cell $f\Rightarrow g$ between parallel functors $f,g \colon A \rightrightarrows B$  is an equivalence class represented by a simplicial map $\alpha\colon A\times \Del^1\to B$ making the following diagram \[ \vcenter{ \xymatrix{ A \times \Del^0 \cong A \ar[dr]^f \ar[d]_{A\times\face^1} \\ A \times \Del^1 \ar[r]^-\alpha & B \\ A \times \Del^0 \cong A \ar[ur]_g \ar[u]^{A\times\face^0}}}\] commute. The displayed map $\alpha$ is a 1-simplex in $B^A$ from the vertex $f$ to the vertex $g$. Two such 1-simplices represent the same 2-cell if and only if they are connected by a homotopy (in the sense of \eqref{eq:homotopy-of-1-simplices}) which fixes their common domain $f$ and  codomain $g$. 

We adopt common 2-categorical notation, writing $\alpha\colon f\Rightarrow g$ to denote a 2-cell of $\qCat_2$ which is represented by a simplicial map $\alpha\colon A\times \Del^1\to B$. So if $\alpha,\beta\colon f\Rightarrow g$ are two such represented 2-cells then when we write $\alpha=\beta$ we will not mean that any two particular representing maps $\alpha,\beta\colon A\times\Del^1\to B$ are literally equal but instead that they are appropriately homotopic.

  The 2-category $\qCat_2$ and the simplicial category $\qCat_\infty$ both have the same underlying ordinary category $\qCat$. Furthermore, we know that if $A$ and $B$ are both categories regarded as quasi-categories (via the nerve functor) then $B^A\in \qCat$ is also a category and so $B^A\cong \ho(B^A)$. This in turn implies that the full sub-2-category of $\qCat_2$ spanned by the categories is itself equivalent to $\Cat_2$; we shall identify these from here on. 
  
  The fact that the homotopy category functor $\ho$ preserves finite products allows us to canonically enrich it to a simplicial functor $\ho\colon\qCat_\infty\to\Cat_2$. Specifically we take its action on the hom-space $B^A$ to be the map obtained as the adjoint transpose of the composite $\ho(B^A)\times\ho(A)\cong\ho(B^A\times A)\stackrel{\ho(\ev)}\longrightarrow\ho(B)$.
\end{defn}

\begin{obs}[pointwise isomorphisms are isomorphisms (reprise)]\label{obs:pointwise-iso-reprise}
  We say that a 2-cell $\alpha\colon f\Rightarrow g\colon A\to B$ of $\qCat_2$ is a {\em pointwise isomorphism\/} if and only if for all functors $a\colon \Del^0\to A$ (objects of $A$) the whiskered composite 2-cell $\alpha a \colon fa\Rightarrow ga\colon\Del^0\to B$ is an isomorphism in $\hom'(\Del^0, B) = \ho{B}$. Using this notion, Corollary~\ref{cor:pointwise-equiv} may be recast to posit that $\alpha$ is a pointwise isomorphism in $\qCat_2$ if and only if it is a genuine isomorphism in $\hom'(A,B)=\ho(B^A)$.
\end{obs}

Since $\qCat_\infty$ is the self enrichment of $\qCat$ under its cartesian product,  it is cartesian closed as a quasi-categorically enriched category. We now show that the 2-category $\qCat_2$ inherits the corresponding property:

\begin{prop}\label{prop:qcat2closed} $\qCat_2$ is cartesian closed as a 2-category.
\end{prop}

\begin{proof}
We show that the terminal object, binary products, and internal hom of the quasi-categorically enriched category $\qCat_\infty$ possess the corresponding 2-categorical universal properties. Specifically, we need to demonstrate the existence of canonical isomorphisms
\begin{align*}
  \hom'(A, \Del^0) &\cong \catone \\
  \hom'(A, B\times C) &\cong \hom'(A, B)\times\hom'(A,C) \\
  \hom'(A, C^B) &\cong \hom'(A\times B, C)
\end{align*}
of categories which are natural in all variables.

To establish each of these we simply apply the homotopy category functor $\ho$ to translate the corresponding $\qCat$-enriched universal properties to $\Cat$-enriched ones, as expressed in terms of the hom-categories defined in~\eqref{eq:qCat2homdefn}. 

Because $\Delta^0$ is a terminal object in the simplicially enriched sense, i.e., because $(\Delta^0)^A \cong \Delta^0$, it is also terminal in the 2-categorical sense: applying $\ho$, the canonical isomorphism
  \[
    \hom'(A,\Del^0) = h( (\Delta^0)^A) \cong h( \Delta^0) \cong \catone
  \] 
 asserts that the hom-category from $A$ to $\Delta^0$ is the terminal category. 

  In a similar fashion, since $\qCat_\infty$ is cartesian closed we know that $B \times C$ is a simplicially enriched product, as expressed by the canonical isomorphisms $(B \times C)^A \cong B^A \times C^A$. Applying $\ho$ we get:
\begin{align*}
    \hom'(A, B\times C) = h((B \times C)^A) & \cong h(B^A \times C^A) \\
    &\cong h(B^A) \times h(C^A) = \hom'(A,B)\times\hom'(A,C).
\end{align*} 

  Finally, the cartesian closure of $\qCat_\infty$  gives rise to isomorphisms $(C^B)^A\cong C^{A \times B}$, to which we may apply the homotopy category functor $\ho$ to obtain the isomorphism 
  \[ 
    \hom'(A\times B, C) = h(C^{A \times B}) \cong h((C^B)^A) = \hom'(A,C^B)
  \] 
  which says that $C^B$ defines an internal hom for the 2-category $\qCat_2$.
\end{proof}

As for any cartesian closed 2-category, the exponential defines a 2-functor $\qCat_2\op \times \qCat_2 \to \qCat_2$.

\begin{defn}[the 2-category of all simplicial sets]\label{defn:2-cat.of.all.simpsets}
  The category $\sSet$ of all simplicial sets is cartesian closed, so we can apply the functor $\ho_*\colon\sCat\to\twoCat$ to its self-enrichment. This provides us with a 2-category $\sSet_2 \defeq \ho_*\sSet$ of all simplicial sets, which has $\qCat_2$ as a full sub-2-category. On occasion, we make slightly implicit use of this larger 2-category. However, we generally choose not to distinguish it notationally from $\qCat_2$, leaving whatever disambiguation is required to the context. 
\end{defn}

\begin{rmk}\label{rmk:exp2functor} 
Exponentiation in the cartesian closed simplicial category $\sSet$ restricts to a simplicial cotensor functor $\sSet\op\times\qCat_\infty\to\qCat_\infty$.

Proposition~\ref{prop:qcat2closed} extends immediately to show that the 2-category of all simplicial sets is again cartesian closed as a 2-category.  Applying $\ho_*$, we obtain a 2-functor $\sSet_2\op\times\qCat_2\to\qCat_2$.  In particular, it follows that exponentiation by any simplicial set $X$ defines a 2-functor $(-)^X \colon \qCat_2 \to \qCat_2$.
  \end{rmk}

\begin{defn}[equivalences in 2-categories]
A 1-cell $u\colon A\to B$ in a 2-category $\tcat{C}$ is an {\em equivalence\/} if and only if there exists a 1-cell $v\colon B\to A$, called its {\em equivalence inverse}, and a pair of 2-isomorphisms $uv\cong\id_B$ and $vu\cong\id_A$. 
\end{defn}

The equivalences of a 2-category $\tcat{C}$ are preserved by all 2-functors since they are defined by 2-equational conditions. Consequently, if $u\colon A\to B$ is an equivalence in $\tcat{C}$ then, applying the representable 2-functor $\tcat{C}(X,-)$,  the functor $\tcat{C}(X,u)\colon\tcat{C}(X,A)\to \tcat{C}(X,B)$ is an equivalence of hom-categories. A basic 2-categorical fact, whose proof is left to the reader, is that these \emph{representably-defined equivalences} are necessarily equivalences in $\tcat{C}$.

\begin{lem}\label{lem:2-cat.equivs}
A 1-cell $u\colon A\to B$ in a 2-category $\tcat{C}$ is an equivalence if and only if $\tcat{C}(X,u)\colon\tcat{C}(X,A)\to \tcat{C}(X,B)$ is an equivalence of hom-categories for all objects $X \in \tcat{C}$. 
\end{lem}

%\begin{rmk}\label{rmk:weak.equivs.2-cat}
  Our central thesis is that the category theory of quasi-categories developed by Joyal, Lurie, and others is captured by $\qCat_2$. For this, it is essential that the standard notion of equivalence of quasi-categories---weak equivalence in the Joyal model structure---is encoded in the 2-category.

  To that end, observe that the description of the weak equivalences given in \ref{rec:qmc-quasicat} may be recast in our 2-categorical framework: by definition, a simplicial map $u\colon X\to Y$ is a weak equivalence in Joyal's model structure if and only if for all quasi-categories $A$ the functor $\hom'(u,A)\colon\hom'(Y,A)\to\hom'(X,A)$ is an equivalence of hom-categories.
%\end{rmk}

  Combining this description with Proposition \ref{prop:qcat2closed} and Lemma \ref{lem:2-cat.equivs} we obtain the following straightforward results:

\begin{prop}\label{prop:equivsareequivs} A functor between quasi-categories is a weak equivalence in the Joyal model structure if and only if it is an equivalence in the 2-category $\qCat_2$.
\end{prop}
\begin{proof}
The weak equivalences between quasi-categories are the representably defined equivalences in the dual 2-category $\qCat_2\op$. Equivalence in a 2-category is a self dual notion, so these coincide with the equivalences in $\qCat_2$.
\end{proof}

\begin{prop}\label{prop:equivsareequivs2}
  A simplicial map $u\colon X\to Y$ is a weak equivalence in the Joyal model structure if and only if for all quasi-categories $A$ the pre-composition functor $A^u\colon A^Y\to A^X$ is an equivalence in the 2-category $\qCat_2$.
\end{prop}

\begin{proof}
  By Lemma~\ref{lem:2-cat.equivs}, $A^u\colon A^Y\to A^X$ is an equivalence in $\qCat_2$ if and only if for all quasi-categories $B$ the functor $\hom'(B,A^u) \colon \hom'(B,A^Y) \to \hom'(B,A^X)$ is an equivalence of hom-categories. Taking duals, $\hom'(B,A^u)$ is isomorphic to $\hom'(u,A^B) \colon \hom'(Y,A^B) \to \hom'(X,A^B)$. Hence, it suffices to show that  $u\colon X\to Y$ is a weak equivalence in Joyal's model structure if and only if $\hom'(u, A^B)$ is an equivalence of hom-categories for all quasi-categories $A$ and $B$, which is the case because $B^A$ is again a quasi-category.
\end{proof}

\subsection{Weak 2-limits}\label{subsec:weak-2-limits}

Finite products aside, the 2-category $\qCat_2$ has few 2-limits. However, we shall soon discover that it has a number of important \emph{weak 2-limits} whose universal properties will be repeatedly exploited in the remainder of this paper.

\begin{defn}[smothering functors]\label{defn:smothering}
  A functor between categories is {\em smothering\/} if and only if it is surjective on objects, full, and conservative (reflects isomorphisms). Equivalently, a functor is smothering if and only if it possesses the right lifting property with respect to the set of functors
  \[ \left\{ \vcenter{\xymatrix@C=5pt{ \emptyset \ar@{u(->}[d] \\ \bullet }},  \vcenter{\xymatrix@C=5pt{ \bullet &  \ar@{u(->}[d] & \bullet  \\ \bullet \ar[rr] & {~} &  \bullet }} , \vcenter{\xymatrix@C=5pt{ \bullet \ar[rr] & \ar@{u(->}[d] & \bullet \\ \bullet \ar[rr] & {~} & \bullet \ar@<.8ex>[ll] }}  \right\}  = \left\{ \vcenter{\xymatrix@C=5pt{ \emptyset \ar@{u(->}[d] \\ \catone }},  \vcenter{\xymatrix@C=5pt{   \catone\sqcup\catone \ar@{u(->}[d]   \\   \cattwo  }} , \vcenter{\xymatrix@C=5pt{ \cattwo \ar@{u(->}[d]   \\  \iso }}  \right\}  \]
  Consequently, the class of smothering functors contains all surjective equivalences and is closed under composition, retract, and pullback along arbitrary functors. By composing lifting problems, we see that all smothering functors are isofibrations, in the sense that they have the right lifting property with respect to either inclusion $\catone\inc \iso$. It is easily checked that if $f$ is a functor which is surjective on objects and arrows, as is true for a smothering functor, and a composite $gf$ is smothering, then so is the functor $g$.
\end{defn}

The following very simple lemma will be of significant utility later on.

\begin{lem}[fibres of smothering functors]\label{lem:smothering}
  Each fibre of a smothering functor is a non-empty connected groupoid. 
\end{lem}

\begin{proof}
  Suppose that $f\colon A\to B$ is a smothering functor. The fact that it is surjective on objects implies immediately that its fibres are non-empty. Furthermore, if $a$ and $a'$ are both objects of $A$ in the fibre of $f$ over some object $b$ in $B$,  then the fullness of $f$ implies that we may find an arrow $\tau\colon a\to a'$ in $A$ with $f(\tau)=\id_b$, thus demonstrating that the fibres are  connected. Finally, if $\tau\colon a\to a'$ is an arrow of $A$ which lies in the fibre of $f$ over $b$, in other words if $f(\tau) = \id_b$, then by conservativity of $f$ we know that $\tau$ is an isomorphism. Hence, these fibres are groupoids. 
\end{proof}

  We have chosen the term smothering here to evoke the image that these are surjective covering functors in quite a strong sense.  Of course, we have placed our tongues firmly in our cheeks while introducing this nomenclature. Smothering functors can fruitfully be thought of as being a certain variety of weak surjective equivalences.

We weaken the standard theory of {\em weighted 2-limits\/} (see e.g., \cite{Kelly:1989fk}) as follows. 

\begin{defn}[weak 2-limits in a 2-category]\label{defn:weak2limit}
  Suppose that $\stcat{A}$ is a small 2-category, that $D \colon \stcat{A} \to \tcat{C}$ is a diagram in a 2-category $\tcat{C}$, and that $W\colon\stcat{A}\to\Cat_2$ is a 2-functor, which we shall refer to as a {\em weight}. If $P$ is an object in $\tcat{C}$ then a {\em cone with summit $P$ over $D$ weighted by\/} $W$ is a 2-natural transformation $c\colon W\To \tcat{C}(P,D(-))$. 

  For each object $K$ of $\tcat{C}$, composition with such a cone induces a functor
  \begin{equation}\label{eq:cone-induced}
    c_K\colon\tcat{C}(K,P)\longrightarrow \lim(W, \tcat{C}(K,D(-)))\cong\int_{a\in\stcat{A}} \tcat{C}(K,D(a))^{W(a)}
  \end{equation}
  where the expression on the right denotes the usual category of 2-natural transformations from W to the 2-functor $\tcat{C}(K,D(-))$, the 2-limit of $\tcat{C}(K,D(-))$  weighted by $W$. The family of maps \eqref{eq:cone-induced} is 2-natural in $K$.

  We say that the cone $c$ displays $P$ as a {\em weak $2$-limit of $D$ weighted by\/} $W$ if and only if the map in \eqref{eq:cone-induced} is a smothering functor for all objects $K\in\tcat{C}$.
\end{defn}

  While we feel obliged to give the last definition in its full, slightly unsightly, generality. However, the reader need not become an expert in the technology of weighted 2-limits in order to read the rest of the paper. We shall only work with certain simple varieties of weak 2-limits in $\qCat_2$, whose weak 2-universal properties we shall describe explicitly.

The fact that the fibres of a smothering functor are connected groupoids is the key ingredient in the proof of the following lemma.

\begin{lem}\label{lem:unique-weak-2-limits} Weak 2-limits are unique up to equivalence: the summits of any two weak 2-limit over a common diagram with a fixed weight are equivalent via an equivalence that commutes with the legs of the limit cones.
\end{lem}

\begin{proof}
Given a pair of cones $c\colon W\To \tcat{C}(P,D(-))$ and $c'\colon W\To \tcat{C}(P',D(-))$ that display  $P$ and $P'$ as weak 2-limits of $D$ weighted by $W$, then for each $K \in \tcat{C}$ we have a pair of smothering functors:
  \begin{equation*}
    \tcat{C}(K,P)\stackrel{c_K}\longrightarrow
    \lim(W, \tcat{C}(K,D(-)))
    \stackrel{c'_{K}}\longleftarrow\tcat{C}(K,P')
  \end{equation*}
  Taking $K=P$, consider the identity 1-cell $\id_P$, an object in the hom-category $\tcat{C}(P,P)$. Since $c'_P$ is surjective on objects, there is a 1-cell $u\colon P\to P'$, an object in $\tcat{C}(P,P')$, such that $c'_{P}(u) = c_{P}(\id_P)$. Exchanging the role of $P$ and $P'$, we also find a 1-cell $u'\colon P'\to P$ such that $c_{P'}(u')=c'_{P'}(\id_{P'})$. These definitions ensure that $u$ and $u'$ commute with the legs of the limit cones.

  Now we can apply the 2-naturality properties of the functors $c_K$ and $c'_K$ to show that
  \begin{align*}
    c_{P}(u'u) &= \lim(W, \tcat{C}(u,D(-)))(c_{P'}(u')) && \text{naturality of family $c_K$}\\
    &= \lim(W, \tcat{C}(u,D(-)))(c'_{P'}(\id_{P'})) && \text{definition of $u'$}\\
    &= c'_{P}(u) && \text{naturality of family $c'_K$}\\
    & = c_{P}(\id_P) && \text{definition of $u$.}
  \end{align*} 
  In other words, $u'u$ and $\id_P$ are both in the same fibre of $c_P$, and so they are isomorphic in that fibre since $c_P$ is a smothering functor. Dually, $uu'$ and $\id_{P'}$ are both in the same fibre of $c'_{P'}$ from which it follows that they too are isomorphic in that fibre. It follows that $u\colon P\to P'$ and $u'\colon P'\to P$ are equivalence inverses. 
\end{proof}

The only diagrams we will consider are indexed by small 1-categories $\stcat{A}$.   Because $\qCat_2$ and $\qCat_\infty$ have the same underlying category, a diagram $D \colon \stcat{A}\to\qCat$ is equally a 2-functor $D\colon\stcat{A}\to\qCat_2$ and a simplicial functor $D\colon\stcat{A}\to\qCat_\infty$. A weight $W\colon\stcat{A}\to\Cat$ for a 2-limit can be regarded as a weight for a simplicial limit by composing with the subcategory inclusion $\Cat\inc\sSet$. Our general strategy will be to show that  the simplicial weighted limit $\lim(W,D)$ exists in $\qCat_\infty$ and that it has the weak 2-universal property expected of the weak 2-limit of $D$ in $\qCat_2$. The following lemma allows us to considerably simplify the class of functors \eqref{eq:cone-induced} that we will need to consider.

\begin{lem}\label{lem:weak-simplification} Fix a small 1-category $\stcat{A}$ and a weight $W \colon \stcat{A} \to \Cat$. Suppose $\mclass{D}$ is a class of diagrams $D\colon\stcat{A} \to \qCat$ that is closed under exponentiation by quasi-categories, in the sense that if $D$ is in the class $\mclass{D}$ then so is $D(-)^X$ for any quasi-category $X$. Then $\qCat_2$ admits weak $W$-weighted 2-limits of this class of diagrams if and only if, for all $D \in \mclass{D}$, the canonical functor
\[ \ho(\lim(W,D)) \to \lim(W,h(D(-)))\] is smothering.
\end{lem}
\begin{proof}
By Definition \ref{defn:weak2limit}, to show that the simplicial weighted limit $\lim(W,D)$ defines a weak 2-limit of a diagram $D \colon \stcat{A} \to \qCat$ in the class $\mclass{D}$, we must show that  for each quasi-category $X$ the canonical comparison map
  \begin{equation}\label{eq:qCat.wl.comp.1}
    \hom'(X,\lim(W,D)) \longrightarrow \lim(W,\hom'(X,D(-)))
  \end{equation}
  is a smothering functor. Recall that $\hom'(X,-) = \ho((-)^X)$. The right adjoint simplicial functor $(-)^X\colon\qCat_\infty\to\qCat_\infty$ preserves all simplicial weighted limits; in other words, the canonical comparison map $\lim(W,D)^X\to\lim(W,D(-)^X)$ is an isomorphism. Thus, the comparison functor~\eqref{eq:qCat.wl.comp.1} is isomorphic to the functor:
  \begin{equation*}
    \ho(\lim(W,D(-)^X)) \longrightarrow \lim(W,\ho(D(-)^X)).
  \end{equation*}
By hypothesis, the diagram $D(-)^X$ is in $\mclass{D}$. Thus, to prove that $\qCat_2$ admits weak 2-limits of the diagrams in $\mclass{D}$, it suffices to show that for all diagrams $D\in\mclass{D}$ the comparison map
  \begin{equation*}
    \ho(\lim(W,D)) \longrightarrow \lim(W,\ho{(D(-))})
  \end{equation*}
  is smothering. 
\end{proof}

\begin{obs}[cones whose summits are not quasi-categories]
The classes of diagrams $\mclass{D}$ we will consider are in fact closed under exponentiation by all simplicial sets. The proof of Lemma \ref{lem:weak-simplification} can then be used to extend the 2-universal properties of the weak 2-limits of $\qCat_2$ constructed here to cones whose summits are arbitrary simplicial sets.  Abstractly speaking, this tells us that the inclusion 2-functor $\qCat_2\inc\sSet_2$ preserves the weak 2-limits of diagrams in $\mclass{D}$. 
In order to avoid repeated remarks of this kind throughout the remainder of this paper, our notation will tacitly signal when this is so by use of the letter ``$X$'' for the object of $\qCat_2$ or $\qCat_\infty$ that could equally be replaced by any simplicial set. By contrast, the letters ``$A$'', ``$B$'', and ``$C$'' refer only to quasi-categories.
\end{obs}

As our first example of a weak 2-limit in $\qCat_2$ we examine cotensors with the generic arrow $\cattwo$. 
Recall we write $A^\cattwo$ for the quasi-category $A^{\Delta^1}$ using our convention that categories are identified with their nerves. We invite the reader to verify that the natural functor $\ho(A^\cattwo) \to (\ho{A})^\cattwo$ is not an isomorphism: it is neither injective on objects nor faithful. However, it is a smothering functor. In other words:

\begin{prop}\label{prop:weak-cotensors} 
The exponential $A^\cattwo$ is a weak cotensor of $A$ by $\cattwo$ in 
$\qCat_2$.
\end{prop}

\begin{proof}
   By Lemma~\ref{lem:weak-simplification}, it suffices to prove that for any quasi-category $A$, the canonical functor \[ \ho(A^\cattwo) \longrightarrow (\ho{A})^\cattwo\] is a smothering functor. Certainly this map is surjective on objects, simply because every arrow in $\ho{A}$ is represented by a 1-simplex in the quasi-category $A$. 

To prove fullness, suppose given a commutative square in $\ho{A}$ and choose arbitrary 1-simplices representing each morphism  and their common composite \begin{equation}\label{eq:arrowinhA}\xymatrix{ \cdot \ar[d]_f \ar[r]^a  \ar[dr]|k & \cdot \ar[d]^g \\ \cdot \ar[r]_b & \cdot}\end{equation} Because $A$ is a quasi-category, any relation between morphisms in $\ho{A}$ is witnessed by a 2-simplex with any choice of representative 1-simplices as its boundary. Hence, we may choose 2-simplices witnessing the fact that $k$ is a composite of $a$ with $g$ and of $f$ with $b$ as displayed. 
 \begin{equation}\label{eq:liftedarrowinA2} \xymatrix{ \cdot \ar[d]_f \ar[r]^{a} \ar[dr]|k^*+{\labelstyle\sim}_*+{\labelstyle \sim}& \cdot \ar[d]^g \\ \cdot \ar[r]_{b} & \cdot}\end{equation}
These two 2-simplices define a map $\Delta^1  \to A^{\Delta^1} = A^\cattwo$, which represents an arrow in the category $h(A^{\cattwo})$ whose image is the specified commutative square.

To prove conservativity, suppose given a map in $h(A^{\cattwo})$ represented by a diagram  \eqref{eq:liftedarrowinA2} whose image  \eqref{eq:arrowinhA} is an isomorphism in $(\ho{A})^\cattwo$, meaning that $a$ and $b$ are isomorphisms in $\ho{A}$, in which case $a$ and $b$ are isomorphisms in the quasi-category $A$. Lemma~\ref{lem:pointwise-equiv} tells us immediately that this diagram is an isomorphism in $A^\cattwo$; compare with \eqref{eq:pointwise-equivalence-square}.
\end{proof}

\begin{rmk}
  A generalisation of this argument shows that if $\scat{C}$ is a free category and $A$ is a quasi-category then the exponential $A^\scat{C}$ is the weak cotensor of $A$ by $\scat{C}$ in $\qCat_2$. Conservativity of the canonical comparison $\ho(A^\scat{C})\to(\ho A)^\scat{C}$ follows from Lemma~\ref{lem:pointwise-equiv}. Its surjectivity on objects makes use of the fact that the inclusion of the \emph{spine} of an $n$-simplex, the simplicial subset spanned by the edges $\fbv{i,i+1}$ in $\Del^n$, is a trivial cofibration for all $n \geq 1$. Fullness is similar.

One should note, however, that this result does not hold for exponentiation by arbitrary categories $\scat{C}$. For example,  $A^{\cattwo\times\cattwo}$ is not the weak cotensor of $A$ by the product category $\cattwo\times\cattwo$ in $\qCat_2$.
\end{rmk}

\begin{prop}
  The exponential $A^\iso$ is a weak cotensor of $A$ by the generic isomorphism $\iso$ in 
  $\qCat_2$. 
\end{prop}

\begin{proof}
 By Lemma \ref{lem:weak-simplification}, it suffices to show that
  \begin{equation*}
    \ho(A^\iso)\longrightarrow \ho(A)^\iso
  \end{equation*}
  is a smothering functor. This is easiest to do by arguing in the marked context. 

By Observation~\ref{obs:nat-mark-homs}, $A^\iso$ may equally well be regarded as an internal hom of naturally marked quasi-categories in $\msSet$. Recollection~\ref{rec:qmc-quasi-marked} tells us that the inclusion $\cattwo^\sharp\inc\iso$ is a trivial cofibration in the marked model structure. Because the marked model structure is cartesian closed, the restriction functor  $A^\iso\to A^{\cattwo^\sharp}$  is a trivial fibration.  Immediately from their defining lifting properties, trivial fibrations of quasi-categories are carried by $\ho$ to functors which are surjective on objects and fully faithful, the so-called {\em surjective equivalences}, so it follows that $\ho(A^\iso)\to \ho(A^{\cattwo^\sharp})$ is a surjective equivalence. Furthermore, in the case where $A$ is an actual category, the functor $A^\iso\to A^{\cattwo^\sharp}$ is an isomorphism. So we obtain a commutative square
  \begin{equation*}
    \xymatrix{
      {\ho(A^\iso)} \ar[r]\ar[d] &
      {\ho(A)^\iso} \ar[d]^{\cong}\\
      {\ho(A^{\cattwo^\sharp})}\ar[r] &
      {\ho(A)^{\cattwo^\sharp}}
    }
  \end{equation*}
of functors between categories in which the left hand vertical is a surjective equivalence. By the composition and cancellation results described in \ref{defn:smothering}, the upper horizontal map in this square is a smothering functor if and only if the lower horizontal map is smothering.

The smothering functors are stable under pullback, so to complete our proof, we will show that  for any naturally marked quasi-category $A$ the square
\begin{equation*}
    \xymatrix{
      {\ho(A^{\cattwo^\sharp})}\pbexcursion \ar[r]\ar[d] &
      {\ho(A)^{\cattwo^\sharp}} \ar[d]\\
      {\ho(A^{\cattwo})}\ar[r] &
      {\ho(A)^{\cattwo}}
    }
  \end{equation*}
  is a pullback;  we know from Proposition~\ref{prop:weak-cotensors} that the lower horizontal map is a smothering functor.  This follows from the definition of the natural marking: a 1-simplex in $A$ is marked if and only if it is an isomorphism, which is the case if and only if it represents an isomorphism in $\ho{A}$. \end{proof}

\begin{prop}\label{prop:weak-homotopy-pullbacks} The 2-category $\qCat_2$ admits weak 2-pullbacks along isofibrations: if the square
\begin{equation*}
  \xymatrix{
    {B\times_A C}\pbexcursion\ar[r]^-{\pi_2} \ar@{->>}[d]_-{\pi_1} 
    &  C\ar@{->>}[d]^g \\
    {B} \ar[r]_f & A
  }
\end{equation*}
is a pullback in simplicial sets for which $B$, $A$, and $C$ are quasi-categories and $g$ is an isofibration, then $B\times_A C$ is a quasi-category and it is a weak 2-pullback of $g$ along $f$ in the 2-category $\qCat_2$.
\end{prop}

\begin{proof}
  The statement only applies to pullbacks of those diagrams of shape $B\xrightarrow{f} A \xleftarrow{g} C$ for which the map $g$ is an isofibration. However, Observation~\ref{obs:isofibration-closure} tells us that any exponentiated isofibration $g^X\colon C^X\to A^X$ is again an isofibration, and so we are in a position to apply Lemma~\ref{lem:weak-simplification}.

   It remains to show that the canonical comparison functor
  \begin{equation*}
    \ho(B\times_A C)\longrightarrow \ho{B}\times_{\ho{A}} \ho{C}
  \end{equation*}
  is smothering. This functor is actually bijective on objects, since in both categories an object consists simply of a pair $(b,c)$ of 0-simplices $b\in B$ and $c\in C$ with $f(b)=g(c)$. 
  
  For fullness, suppose we are given two such pairs $(b,c)$ and $(b',c')$. An arrow between these objects in $\ho{B} \times_{\ho{A}} \ho{C}$ consists of a pair of equivalence classes represented by 1-simplices $\beta \colon b \to b'$ and $\gamma \colon c \to c'$ which both map to the same equivalence class in $\ho{A}$ under $f$ and $g$ respectively. This latter condition simply posits that $f(\beta)$ and $g(\gamma)$ are homotopic relative to their endpoints in $A$; such a homotopy is represented by a 2-simplex with $2^{\text{nd}}$ face $g(\gamma)$, $1^{\text{st}}$ face $f(\beta)$, and $0^{\text{th}}$ face degenerate. This information provides us with a lifting problem between $\Horn^{2,1} \to \Delta^2$ and $g$, which we may solve because $g$ is an isofibration. The resulting filler supplies us with a 1-simplex $\gamma' \colon c \to c'$ for which $g(\gamma')=f(\beta)$ and a homotopy of $\gamma'$ and $\gamma$ (relative to their endpoints) which shows these represent the same arrow in $\ho{C}$. In other words, $(\beta,\gamma')$ is a 1-simplex in $B\times_A C$ that represents an arrow of $\ho(B\times_A C)$ from $(b,c)$ to $(b',c')$ and this arrow maps to the originally chosen arrow in $\ho{B}\times_{\ho{A}}\ho{C}$.

The proof of conservativity is simplified by arguing in the marked model structure.  Giving our quasi-categories $A$, $B$, and $C$ the natural marking, the isofibration $g$ becomes a fibration in the marked model structure. It follows that the pullback is a fibrant object and hence naturally marked. Consequently, a 1-simplex $(\beta,\gamma)$ of $B\times_A C$ represents an isomorphism in $\ho(B\times_A C)$ if and only if it is marked, and this is the case if and only if $\beta$ is marked in $B$ and $\gamma$ is marked in $C$. Now, this latter condition holds if and only if $\beta$ is invertible in $\ho{B}$ and $\gamma$ is invertible in $\ho{C}$ and these conditions together are equivalent to the pair $(\beta,\gamma)$ being invertible as an arrow in the category $\ho{B}\times_{\ho{A}} \ho{C}$.
\end{proof} 

  \begin{defn}[comma objects]\label{def:comma-obj}
 Given a pair of functors $B\xrightarrow{f} A \xleftarrow{g} C$ between quasi-categories, we define the {\em comma object\/} $f\comma g$ to be the simplicial set constructed by forming the following pullback:
    \begin{equation*}
      \xymatrix@=2.5em{
        {f\comma g}\pbexcursion \ar[r]\ar[d]_{p} &
        {A^\cattwo} \ar[d] \\
        {C\times B} \ar[r]_-{g\times f} & {A\times A}
      }
    \end{equation*}
\end{defn}

The right-hand vertical  is defined by restricting along the boundary inclusion $\Del^0\sqcup\Del^0\cong\boundary\Del^1\inc\Del^1$ and then composing with the symmetry isomorphism $A \times A \cong A \times A$. In a subsequent paper, we will think of the comma object $f \comma g$ as a module, with $C$ acting on the left and with $B$ acting on the right, which is the reason for our convention.

\begin{lem}\label{lem:comma-obj} The simplicial set $f \comma g$ is a quasi-category and the projection functors $p_0\defeq\pi_B\circ p \colon f\comma g\tfib B$ and $p_1\defeq\pi_C\circ p\colon f\comma g\tfib C$ are isofibrations.
\end{lem}
\begin{proof}
    The right hand vertical in the pullback square above is isomorphic to the simplicial map $A^{\Del^1}\to A^{\boundary\Del^1}$ and is thus, by \ref{rec:cart-modcat}, an isofibration whenever $A$ is a quasi-category. Consequently, since the product $C\times B$ is again a quasi-category, $p\colon f\comma g\to C\times B$ is  an isofibration and $f\comma g$ is a quasi-category. The projection functors $\pi_C\colon C\times B\tfib C$ and $\pi_B\colon C\times B\tfib B$ are both isofibrations because $B$ and $C$ are fibrant, so it follows that the domain and codomain projection maps $p_0 \colon f\comma g\tfib B$ and $p_1\colon f\comma g\tfib C$ are also isofibrations.
    \end{proof}
        
\begin{lem}[maps induced between comma objects]\label{lem:comma-obj-maps}
A commutative diagram
     \begin{equation*}
       \xymatrix{
         {B}\ar[r]^{f}\ar@{->>}[d]_-{r} & 
         {A}\ar@{->>}[d]_-{q} & 
         {C}\ar[l]_{g}\ar@{->>}[d]^-{s} \\
         {\bar{B}}\ar[r]_{\bar{f}} & 
         {\bar{A}} & 
         {\bar{C}}\ar[l]^{\bar{g}} 
       }
     \end{equation*} 
     in $\qCat$ in which the vertical maps are (trivial) fibrations in the Joyal model structure, induces a (trivial) fibration $r \comma_q s \colon f \comma g \tfib \bar{f}\comma\bar{g}$ between comma quasi-categories.
\end{lem}
\begin{proof}
Consider the commutative diagram
     \begin{equation*}
       \xymatrix@C=4.5em@R=2.5em{
         {C\times B}\ar[r]^-{g\times f}\ar@{->>}[d]_-{s\times r} & 
         {A\times A}\ar@{->}[d]_-{q\times q} & 
         {A^\cattwo}\ar[l]_-{(p_1,p_0)}\ar@{->}[d]^-{q^\cattwo} 
         \save "1,3"-<2.5em,1.5em>*+{P}\ar[l]\ar[d]\ar@{<<.}[]_(0.6)l \restore \\
         {\bar{C}\times\bar{B}}\ar[r]_-{\bar{g}\times\bar{f}} & 
         {\bar{A}\times\bar{A}} & 
         {\bar{A}^\cattwo}\ar[l]^-{(p_1,p_0)} 
       }
     \end{equation*}
     in which $P$ denotes the pullback of the maps $q\times q$ and $(p_1,p_0)$ and $l$ is the unique map induced into it by the right hand square. The pullbacks of the two horizontal lines are the comma objects $f\comma g$ and $\bar{f}\comma\bar{g}$ respectively. So this diagram induces a unique map $r\comma_q s\colon f\comma g\to\bar{f}\comma\bar{g}$ of comma objects which makes the manifest cube commute.
     
     The (trivial) fibrations of any model category are closed under product, so the map $s\times r$ is a (trivial) fibration in the Joyal model structure.  The induced map $l$ is isomorphic to the Leibniz hom $\leib\hom(\boundary\Del^1\inc\Del^1, q\colon A\tfib \bar{A})$; a recalled in \ref{rec:cart-modcat}, cartesianness of the Joyal model structure implies that $l$ is a (trivial) fibration.
The induced map $r\comma_q s\colon f\comma g\tfib\bar{f}\comma\bar{g}$ is again a (trivial) fibration because it factors as a composite of pullbacks of the (trivial) fibrations $s\times r$ and $l$.
   \end{proof}

\begin{prop}\label{prop:weakcomma} For any functors $B \xrightarrow{f} A \xleftarrow{g} C$  of quasi-categories, the comma quasi-category $f\comma g$ is a weak comma object in $\qCat_2$.
\end{prop}
\begin{proof}
    Again, Lemma~\ref{lem:weak-simplification} applies, so it suffices to show that the canonical comparison
  \begin{equation}\label{eq:commacat-comp}
    \ho(f\comma g) \longrightarrow \ho(f)\comma\ho(g)
  \end{equation}
  is a smothering functor. Here the target category is just the usual comma category constructed in $\Cat$. By definition,  $f\comma g \cong (C\times B)\times_{(A\times A)} A^\cattwo$ and consequently we find that we may express the functor in~\eqref{eq:commacat-comp} as a composite:
  \begin{equation*}
    \ho((C\times B)\times_{(A\times A)} A^\cattwo)
    \longrightarrow
    \ho(C\times B)\times_{\ho(A\times A)} \ho(A^\cattwo)
    \longrightarrow
    \ho(C\times B)\times_{\ho(A\times A)} \ho(A)^\cattwo
  \end{equation*}
  The first of these maps is the canonical comparison functor studied in Proposition~\ref{prop:weak-homotopy-pullbacks}, so we know that it is smothering. The second of these maps is a pullback of the canonical comparison functor discussed in Proposition~\ref{prop:weak-cotensors}; since smothering functors are stable under pullback, it too is a smothering functor. We obtain the required result from the fact that smothering functors compose.
\end{proof}

\begin{obs}[unpacking the universal property of weak comma objects]\label{obs:unpacking-weak-comma-objects}
  The smothering functors
  \begin{equation}\label{eq:weak-comma-prop}
    \hom'(X,f\comma g)\longrightarrow\hom'(X,f)\comma\hom'(X,g)
  \end{equation}
  which express the weak 2-universal property of the quasi-category $f\comma g$ are induced by composition with a cone:
  \begin{equation}\label{eq:standard-comma-pic}
    \xymatrix@=10pt{
      & f \downarrow g \ar[dl]_{p_1} \ar[dr]^{p_0} \ar@{}[dd]|(.4){\psi}|{\Leftarrow}  \\ 
      C \ar[dr]_g & & B \ar[dl]^f \\ 
      & A}
  \end{equation}
  The data displayed in \eqref{eq:standard-comma-pic} is the image of the identity 1-cell under \eqref{eq:weak-comma-prop} in the case $X = f\comma g$. The weak universal property of this comma cone has three aspects, corresponding to the surjectivity on objects, fullness, and conservativity of the smothering functor \eqref{eq:weak-comma-prop}, which we refer to as 
 {\em 1-cell induction\/}, {\em 2-cell induction}, and {\em 2-cell conservativity}.

Surjectivity on objects of the functor~\eqref{eq:weak-comma-prop} simply says that for any comma cone
  \begin{equation}\label{eq:comma-cone}
    \xymatrix@=10pt{
       & X \ar[dl]_{c} \ar[dr]^{b} \ar@{}[dd]|(.4){\alpha}|{\Leftarrow}  \\ 
       C \ar[dr]_g & & B \ar[dl]^f \\ 
       & A}
  \end{equation}
  over our diagram there exists a map $a\colon X \to f \comma g$ which factors $b\colon X\to B$ and $c\colon X\to C$ through $p_0\colon f\comma g\to B$ and $p_1\colon f\comma g\to C$ respectively and which whiskers with the 2-cell $\psi\colon fp_0\Rightarrow gp_1$ to give the 2-cell $\alpha\colon fb\Rightarrow gc$; diagrammatically speaking, \emph{1-cell induction} produces a functor $a \colon X \to f \comma g$ from a 2-cell $\alpha \colon fb \To gc$ so that:
  \begin{equation}\label{eq:comma-ind-1cell-prop}
    \vcenter{\xymatrix@=10pt{
       & X \ar[dl]_{c} \ar[dr]^{b} \ar@{}[dd]|(.4){\alpha}|{\Leftarrow}  \\ 
       C \ar[dr]_g & & B \ar[dl]^f \\ 
       & A
    }}
    \mkern20mu = \mkern20mu
    \vcenter{\xymatrix@=10pt{
      & f \downarrow g \ar[dl]_{p_1} \ar[dr]^{p_0} \ar@{}[dd]|(.4){\psi}|{\Leftarrow}  \\ 
      C \ar[dr]_g & & B \ar[dl]^f \\ 
      & A
      \save "1,2"+<0pt,40pt>*+{X}\ar "1,2" _-a\restore
      }}
  \end{equation}

Fullness of \eqref{eq:weak-comma-prop} tells us that if we are given a pair of functors $a,a'\colon X\to f\comma g$ and a pair of 2-cells
  \begin{equation}\label{eq:comma-ind-2cell-data}
    \vcenter{\xymatrix@=10pt{
      & {X}\ar[dl]_{a'}\ar[dr]^{a}
      \ar@{}[dd]|(.4){\tau_0}|{\Leftarrow} & \\
      {f\comma g}\ar[dr]_{p_0} & & 
      {f\comma g}\ar[dl]^{p_0} \\
      & B &
    }}
    \mkern30mu\text{and}\mkern30mu
    \vcenter{\xymatrix@=10pt{
      & {X}\ar[dl]_{a'}\ar[dr]^{a}
      \ar@{}[dd]|(.4){\tau_1}|{\Leftarrow} & \\
      {f\comma g}\ar[dr]_{p_1} & & 
      {f\comma g}\ar[dl]^{p_1} \\
      & C &
    }}
  \end{equation}
  with the property that 
    \begin{equation}\label{eq:comma-ind-2cell-compat}
      \xymatrix@=10pt{ & X \ar[dl]_{a'} \ar[dr]^a \ar@{}[dd]|(.4){\tau_1}|{\Leftarrow} & &  & \ar@{}[dd]|{\displaystyle =} & &  & X \ar[dl]_{a'} \ar[dr]^a  \ar@{}[dd]|(.4){\tau_0}|{\Leftarrow} \\ f \downarrow g  \ar[dr]_{p_1} & & f \downarrow g \ar[dl]|{p_1} \ar[dr]^{p_0}   \ar@{}[dd]|(.4){\psi}|{\Leftarrow} & & & &   f \downarrow g \ar[dl]_{p_1} \ar[dr]|{p_0}  \ar@{}[dd]|(.4){\psi}|{\Leftarrow}  & & f \downarrow g \ar[dl]^{p_0} \\ & C \ar[dr]_g & & B \ar[dl]^f & & C \ar[dr]_g & & B \ar[dl]^f & &  \\ & & A & &  &  & A}
    \end{equation}
  then there exists a 2-cell $\tau \colon a \Rightarrow a'$, defined by \emph{2-cell induction}, satisfying the equalities 
  \[     \vcenter{\xymatrix@=10pt{
      & {X}\ar[dl]_{a'}\ar[dr]^{a}
      \ar@{}[dd]|(.4){\tau_0}|{\Leftarrow} & \\
      {f\comma g}\ar[dr]_{p_0} & & 
      {f\comma g}\ar[dl]^{p_0} \\
      & B &
    }}    \mkern20mu = \mkern20mu
    \vcenter{\xymatrix@=30pt{ X \ar@/^2ex/[d]^a \ar@/_2ex/[d]_{a'} \ar@{}[d]|(.4){\tau}|{\Leftarrow}  \\ f \downarrow g \ar[d]^{p_0} \\ B}}
      \mkern30mu\text{and}\mkern30mu
    \vcenter{\xymatrix@=10pt{
      & {X}\ar[dl]_{a'}\ar[dr]^{a}
      \ar@{}[dd]|(.4){\tau_1}|{\Leftarrow} & \\
      {f\comma g}\ar[dr]_{p_1} & & 
      {f\comma g}\ar[dl]^{p_1} \\
      & C &
    }}     \mkern20mu = \mkern20mu
        \vcenter{\xymatrix@=30pt{ X \ar@/^2ex/[d]^a \ar@/_2ex/[d]_{a'} \ar@{}[d]|(.4){\tau}|{\Leftarrow}  \\ f \downarrow g \ar[d]^{p_1} \\ C}}.    \]

    Finally, conservativity of \eqref{eq:weak-comma-prop} tells us that if we are given a 2-cell $\tau\colon a\Rightarrow a' \colon X \to f \comma g$ then if the whiskered composites $p_0\tau$ and $p_1\tau$, as shown in the previous diagram, are isomorphisms in $\hom'(X,B)$ and $\hom'(X,C)$ respectively, then $\tau$ is also an isomorphism in $\hom'(X,f\comma g)$; this is \emph{2-cell conservativity}.
 \end{obs}

\begin{lem}[1-cell induction is unique up to isomorphism]\label{lem:1cell-ind-uniqueness}
Any two 1-cells $a,a' \colon X \to f \comma g$ over a weak comma object \eqref{eq:standard-comma-pic} that are induced by the same comma cone $\alpha \colon fb \To gc$ are isomorphic over $C \times B$. 
\end{lem}
\begin{proof}
This follows from Lemma~\ref{lem:smothering}, which demonstrates that fibres of smothering functors are connected groupoids, or can be proven directly. From the defining property of induced 1-cells displayed in~\eqref{eq:comma-ind-1cell-prop} it follows that $p_0 a = p_0 a'$, $p_1 a = p_1 a'$, and $\psi a = \psi a'$. We can regard the first two of these equalities as being identity 2-cells of the form displayed in~\eqref{eq:comma-ind-2cell-data}. Then the third of these equalities may be re-interpreted as positing the compatibility property displayed in~\eqref{eq:comma-ind-2cell-compat} for those identity 2-cells. So we may apply the 2-cell induction property of $f\comma g$ to obtain a 2-cell $\tau\colon a\Rightarrow a'$ whose whiskered composites with $p_0$ and $p_1$ are the identity 2-cells corresponding to the equalities $p_0 a = p_0 a'$ and $p_1 a = p_1 a'$ respectively. This then allows us to apply the 2-cell conservativity property of our weak comma object to show that $\tau\colon a\Rightarrow a'$ is an isomorphism.
 \end{proof}

\subsection{Slices of the category of quasi-categories}\label{subsec:slice-2cats-of-qcats}

\begin{defn}[enriching the slices of $\qCat$]\label{defn:enriched-slice}
  For a quasi-category $A$, we will write $\qCat/A$ for the full subcategory of the usual slice category whose objects are isofibrations $E\tfib A$. Where not otherwise stated, we shall restrict our attention to these subcategories of isofibrations: these are the subcategories of fibrant objects in slices of Joyal's model structure and so are better behaved when viewed from the perspective of formal quasi-category theory than the slice categories of all maps with fixed codomain.
   
   The category $\qCat/A$ has two enrichments of interest to us here. Let $\qCat_2\slice A$ and $\qCat_\infty\slice A$ denote the 2-category and simplicial category (respectively) whose objects are the isofibrations with codomain $A$ and whose hom-category and simplicial hom-space (respectively) between $p \colon E \tfib A$ and $q \colon F \tfib A$ are defined by the pullbacks 
   \begin{equation}\label{eq:slice-hom-objects}  
    \xymatrix{
      \hom'_A(p,q) \pbexcursion \ar[d] \ar[r] & \hom'(E,F) \ar[d]^-{\hom'(E,q)} &  &   \hom_A(p,q) \pbexcursion \ar[r] \ar[d] & F^E \ar@{->>}[d]^-{q^E} \\ 
       \catone \ar[r]_-p & \hom'(E,A) & &     \Delta^0 \ar[r]_-p & A^E}
   \end{equation}
  The objects of $\hom'_A(p,q)$ and the vertices of $\hom_A(p,q)$ are exactly the morphisms from $p$ to $q$ in $\qCat/A$. The morphisms in $\hom'_A(p,q)$, 2-cells in the 2-category $\qCat_2\slice A$, are natural transformations between functors $E \to F$ in $\qCat_2$ whose whiskered composite with $q$ is the identity 2-cell on $p$. Since $q\colon F\tfib A$ is an isofibration we know that $q^E\colon F^E\tfib A^E$ is also an isofibration as is its pullback $\hom_A(p,q)\tfib\Del^0$; hence, $\hom_A(p,q)$ is a quasi-category. In other words, $\qCat_\infty\slice A$ is enriched in quasi-categories.
\end{defn}

\begin{obs}[pushforward]\label{obs:fibred-pushforward} If $f \colon B \tfib A$ is an isofibration of quasi-categories then post-composition defines a simplicial functor $f_* \colon \qCat_\infty\slice B \to \qCat_\infty\slice A$ and a 2-functor $f_* \colon \qCat_2\slice B \to \qCat_2\slice A$. 
\end{obs}

One reason for our particular interest in the simplicial categories $\qCat_\infty\slice A$ has to do with the following observation.  Simplicially enriched limits are defined up to isomorphism and thus assemble into a simplicial functor. The universal property defining weak 2-limits, however, lacks a uniqueness statement of sufficient strength to make them assemble into a (strict) 2-functor. In particular:

\begin{obs}[pullback]\label{obs:fibred-pullback}
Consider any functor $f \colon B \to A$ between quasi-categories. Pullback along $f$ defines a functor $f^* \colon \qCat/A \to \qCat/B$, but it cannot be extended to a 2-functor between slice 2-categories $\qCat_2\slice A$ and $\qCat_2\slice B$ in any canonical way. On the other hand, pullback is a genuine simplicial limit in $\qCat_\infty$ and so it does define a simplicial functor $f^* \colon \qCat_\infty\slice A \to \qCat_\infty\slice B$, which in turn gives rise to a 2-functor  $f^* \colon \ho_*(\qCat_\infty\slice A) \to \ho_*(\qCat_\infty\slice B)$ on application of $\ho_*\colon\eCat{\sSet}\to\twoCat$.  The remarks apply equally to the larger slice categories of all maps with fixed codomain.
\end{obs}

\begin{obs}[comparing the 2-categories $\qCat_2\slice A$ and $\ho_*(\qCat_\infty\slice A)$]
The 2-categories $\qCat_2\slice A$ to $\ho_*(\qCat_\infty\slice A)$ have the same 0-cells and 1-cells; however it is not the case that their 2-cells coincide. If we are given a parallel pair of 1-cells \[ \xymatrix@=1.5em{ E \ar@{->>}[dr]_p \ar@/^1ex/[rr]^f \ar@/_1ex/[rr]_g & & F \ar@{->>}[dl]^q \\ & A}\] a 2-cell from $f$ to $g$ in
  \begin{description}
    \item[$\qCat_2\slice A$] is a homotopy class of 1-simplices $f \to g$ in $F^E$ that whisker with $q$ to the homotopy class of the degenerate 1-simplex on $p$.
    \item[$\ho_*(\qCat_\infty\slice A)$] is a homotopy class represented by a 1-simplex $f \to g$ in the fibre of $q^E\colon F^E\tfib A^E$ over the vertex $p\in A^E$ under homotopies which are also constrained to that fibre.
  \end{description}
  Note here that the notion of homotopy involved in the description of 2-cells in $\ho_*(\qCat_\infty\slice A)$ is more refined (identifies fewer simplices) than that given for 2-cells in $\qCat_2\slice A$. Each homotopy class representing a 2-cell in $\qCat_2\slice A$ may actually split into a number of distinct homotopy classes representing 2-cells in $\ho_*(\qCat_\infty\slice A)$.
\end{obs}

  Consequently, it is not the case that these two enrichments of $\qCat\slice A$ to a 2-category are identical. However, they are related by a 2-functor whose properties we now enumerate.

\begin{defn}[smothering 2-functor]\label{defn:smothering-2-functor}
  A 2-functor $F\colon\tcat{C}\to\tcat{D}$ is said to be a {\em smothering 2-functor\/} if it is surjective on 0-cells and \emph{locally smothering}, i.e., if for all 0-cells $K$ and $K'$ in $\tcat{C}$ the action $F\colon \tcat{C}(K,K')\to\tcat{D}(FK,FK')$ of $F$ on the hom-category from $K$ to $K'$ is a smothering functor.

Note that smothering 2-functors are also conservative at the level of 1-cells in the sense appropriate to 2-category theory; that is to say if $k\colon K\to K'$ is a 1-cell in $\tcat{C}$ for which $Fk$ is an equivalence in $\tcat{D}$ then $k$ is an equivalence in $\tcat{C}$.
\end{defn}

\begin{prop}\label{prop:slice-smothering-2-functor} There exists a canonical 2-functor $\ho_*(\qCat_\infty\slice A)\to\qCat_2\slice A$  which acts identically on 0-cells and 1-cells and is a smothering 2-functor.
\end{prop}
\begin{proof}
  To construct the required 2-functor, apply the homotopy category functor $\ho$ to the defining pullback square for $\hom_A(p,q)$ in~\eqref{eq:slice-hom-objects} to obtain a square which then induces a functor $h(\hom_A(p,q))\to \hom'_A(p,q)$ by the pullback property of the defining square for $\hom'_A(p,q)$. It is a routine matter now to check that we may assemble these actions on hom-categories together to give a 2-functor which acts as the identity on the common underlying category $\qCat\slice A$ of these 2-categories. 

  To show that this 2-functor is smothering, we already know that it acts bijectively on 0-cells, so all that remains is to show that each $h(\hom_A(p,q))\to \hom'_A(p,q)$ is a smothering functor. This fact follows by direct application of Proposition~\ref{prop:weak-homotopy-pullbacks} to the defining pullbacks~\eqref{eq:slice-hom-objects}.
 \end{proof}

Our next aim is to develop a useful principle by which to recognise those 1-cells of $\ho_*(\qCat_\infty\slice A)$ which are equivalences in there. To achieve this, we must first explore the 2-categorical properties of the isofibrations between quasi-categories.

 \begin{defn}[representably defined isofibrations in 2-categories]\label{defn:representable-isofibrations}
  A 1-cell $p\colon B\to A$ in a 2-category $\tcat{C}$ is said to be a {\em representably defined isofibration\/} (or just an \emph{isofibration}) if and only if for each object $X\in\tcat{C}$ the functor $\tcat{C}(X,p)\colon\tcat{C}(X,B)\to\tcat{C}(X,A)$ is an isofibration of categories (has the right lifting property with respect to the inclusion $\catone\inc\iso$).  In more explicit terms, this means that for any diagram \[ \xymatrix{ \ar@{}[dr]|(.7){\alpha\cong} & B \ar[d]^p  & \ar@{}[d]|{\displaystyle\rightsquigarrow} &  \ar@{}[dr]|{\beta\cong} & B \ar[d]^p \\ X \ar[ur]^b \ar[r]_a & A &&  X \ar@/^1.5ex/[ur]^b \ar@/_1.5ex/[ur]_*!<-3pt,+3pt>{\labelstyle x} \ar[r]_a & A}\] consisting of 1-cells $a$ and $b$ and a 2-isomorphism $\alpha \colon pb \cong a$, there exists a 1-cell $x$ and 2-isomorphism $\beta \colon b \cong x$ so that $p \beta = \alpha$ and $px = a$.
\end{defn}

\begin{lem}\label{lem:representable-isofibration} If $p \colon B \tfib A$ is an isofibration between quasi-categories, then $p$ is a representably defined isofibration in $\qCat_2$.
\end{lem}
\begin{proof} 
For any simplicial set $X$,   $p^X\colon B^X\tfib A^X$ is also an isofibration  and in particular has the right lifting property with respect to $\catone\inc\iso$. Using the standard homotopy coherence result, recalled in \ref{rec:qmc-quasi-marked}, that an isomorphism in the homotopy category of a quasi-category can be extended to a functor with domain $\iso$, it follows that $\hom'(X,p)\colon\hom'(X,B)\to\hom'(X,A)$ also has the right lifting property with respect to $\catone\inc\iso$. Thus $\hom'(X,p)$ is 
  an isofibration of categories, which shows that the isofibrations of quasi-categories are representably defined  in the 2-category $\qCat_2$.
  \end{proof}

    The following lemma, stated here in the special case of $\qCat_2$, applies equally to any slice 2-category whose objects are isofibrations.

  \begin{lem}\label{lem:proj-is-1-conservative}
     The canonical projection 2-functor $\qCat_2\slice A\to \qCat_2$ is conservative on 1-cells in the appropriate 2-categorical sense: if
\begin{equation}\label{eq:equiv-to-lift}
  \xymatrix@=1.5em{
    {E}\ar[rr]^w\ar@{->>}[dr]_p && {F}\ar@{->>}[dl]^q \\
    & {A} &
  }
\end{equation}
is a 1-cell in $\qCat_2\slice A$ for which $w\colon E\to F$ admits an equivalence inverse $w' \colon F \to E$ in $\qCat_2$, then $w$  is  an equivalence in the slice 2-category $\qCat_2\slice A$.
  \end{lem}
  \begin{proof}
By a standard 2-categorical argument, we may choose  2-isomorphisms $\alpha\colon w'w\cong\id_E$ and $\beta\colon\id_F\cong ww'$ which display $w'$ as a left adjoint equivalence inverse to $w$ in $\qCat_2$. As $p$ is an isofibration in $\qCat_2$, the isomorphism $q\beta\colon q\cong qww' = pw'$ can be lifted along $p$ to give a 1-cell $\bar{w}\colon F\to E$ with $p\bar{w} = q$ and a 2-isomorphism $\gamma\colon \bar{w}\cong w'$ with $p\gamma=q\beta$. The first of these equations tells us that $\bar{w}$ is a 1-cell in $\qCat_2\slice A$.  Using the second of these equations and the triangle identities relating $\alpha$ and $\beta$, we see that the isomorphisms $\alpha\cdot\gamma w\colon \bar{w} w\cong\id_E$ and $w\gamma^{-1}\cdot\beta\colon\id_F\cong w\bar{w}$ are 2-cells in $\qCat_2\slice A$: \[ p(\alpha \cdot \gamma w) = p\alpha \cdot p\gamma w = qw\alpha \cdot q\beta w = q\id_w \qquad q(w\gamma^{-1}\cdot \beta) = qw\gamma^{-1} \cdot q\beta = p\gamma^{-1} \cdot q\beta = \id_p. \]
These isomorphisms display $\bar{w}$ as an equivalence inverse to $w$ in $\qCat_2\slice A$.
  \end{proof}

  \begin{cor}\label{cor:recog-fibred-equivs}
    The 1-cell depicted in~\eqref{eq:equiv-to-lift} is an equivalence in $\ho_*(\qCat_\infty\slice A)$ if and only if $w\colon E\to F$ is an equivalence in $\qCat_2$. 
  \end{cor}

  \begin{proof}
  By  Proposition~\ref{prop:slice-smothering-2-functor} and Lemma \ref{lem:proj-is-1-conservative}, the canonical 2-functors $\ho_*(\qCat_\infty\slice A)\to\qCat_2\slice A$ and $\qCat_2\slice A\to \qCat_2$ are both conservative on 1-cells, so their composite is also conservative on 1-cells. The result follows immediately.
  \end{proof}

  \begin{defn}[fibred equivalence]\label{defn:fibred-equivalence}
  A functor $w \colon E \to F$ between quasi-categories equipped with specified isofibrations $p\colon E \tfib A$ and $q \colon F \tfib A$ is an {\em equivalence fibred over $A$\/}, or just a {\em fibred equivalence}, if it is an equivalence in $\ho_*(\qCat_\infty\slice A)$. By Corollary \ref{cor:recog-fibred-equivs}, any equivalence in $\qCat_2$ which commutes with the maps down to $A$ is a fibred equivalence. Unpacking the definition, a fibred equivalence admits an equivalence inverse $w' \colon F \to E$ over $A$ together with isomorphisms $\alpha \colon w' w \cong \id_E \in E^E$ and $\beta \colon \id_F \cong w w' \in F^F$ represented by 1-simplices that compose with $p$ and $q$ to degenerate 1-simplices.
    \end{defn}

Corollary \ref{cor:recog-fibred-equivs} allows us to lift equivalences in $\qCat_2/A$ to fibred equivalences, which can be pulled back along a functor $f \colon B \to A$ as described in Observation \ref{obs:fibred-pullback}. The lifting arguments developed here relied upon the assumption that the simplicial categories in which we work have hom-spaces which are quasi-categories, which is why our default is to assume that the objects of our slice categories $\qCat_2\slice A$ and $\qCat_\infty\slice A$ are isofibrations.
  
\subsection{A strongly universal characterisation of weak comma objects}

We may use properties of the 2-categorical slice $\qCat_2\slice (C \times B)$ to characterise the weak comma objects of $\qCat_2$ in terms of a \emph{strict} 1-categorical universal property. We present this technical result here and then use it to good effect in section~\ref{sec:limits}, where we demonstrate how to characterise limits and colimits that exist in a quasi-category in purely 2-categorical terms.

For this subsection we shall assume, contrary to our notational convention elsewhere, that $\qCat_2\slice(C\times B)$ denotes the unrestricted slice 2-category whose objects are all functors with codomain $C\times B$.

\begin{obs}[uniqueness of 1-cell induction revisited]\label{obs:1cell-ind-uniqueness-reloaded}
Any 1-cell $a\colon X\to f\comma g$ induced by the comma cone~\eqref{eq:comma-cone} may be regarded as a 1-cell
  \begin{equation*}
    \xymatrix@=1em{
      {X}\ar[dr]_(0.3){(c,b)}\ar[rr]^{a}
      && *+!L(0.5){f\comma g}\ar[dl]^(0.3){(p_1,p_0)} \\
      & {C\times B}&
    }
  \end{equation*}
  in $\qCat_2\slice(C\times B)$. If we are given a second 1-cell $a'\colon X\to f\comma g$ which is also induced by the same comma cone then the argument of Lemma~\ref{lem:1cell-ind-uniqueness} delivers us a 2-cell
  \begin{equation}\label{eq:induced-1-cell-comparison}
    \xymatrix@=1.5em{
      {X}\ar[dr]_(0.3){(c,b)} 
      \ar@/^1.5ex/[rr]^{a}_{}="one" \ar@/_1.5ex/[rr]_{a'}^{}="two" \ar@{=>}"one";"two"^{\tau}
      && *+!L(0.5){f\comma g}\ar[dl]^(0.3){(p_1,p_0)} \\
      & {C\times B}&
    }
  \end{equation}
in $\qCat_2\slice(C\times B)$, which is moreover an isomorphism; this is what we meant by the assertion that any pair of functors defined by 1-cell induction over the same comma cone are isomorphic over $C \times B$.
 Conversely, by 2-cell conservativity of the comma quasi-category $f \downarrow g$, any 2-cell of $\qCat_2\slice(C\times B)$ of the form depicted in~\eqref{eq:induced-1-cell-comparison} is an isomorphism. Thus, the hom-category $\hom'_{C\times B}((c,b),(p_1,p_0))$ is a groupoid, whose connected components comprise those 1-cells induced by a common cone \eqref{eq:comma-cone}.
\end{obs}

\begin{obs}\label{obs:squares-set}
  For each object $(c,b)\colon X\to C\times B$ of $\qCat/(C\times B)$ we have a set $\sq_{g,f}(c,b)$ of 2-cells as depicted in~\eqref{eq:comma-cone}. This construction may be extended immediately to a contravariant functor $\sq_{g,f}\colon(\qCat/(C\times B))\op\to\Set$, which carries a morphism
  \begin{equation*}
    \xymatrix@=1em{
      {X}\ar[dr]_(0.3){(c,b)}\ar[rr]^{u}
      && *+!L(0.5){Y}\ar[dl]^(0.3){(\bar{c},\bar{b})} \\
      & {C\times B}&
    }
  \end{equation*}
  of $\qCat/(C\times B)$ to the function $\sq_{g,f}(u)$ which maps a 2-cell $\beta$ of $\sq_{g,f}(\bar{c},\bar{b})$ to the whiskered 2-cell $\beta u$ in $\sq_{g,f}(c,b)$.

  \end{obs}

\begin{obs}\label{obs:groupoid-components}
There is a product-preserving  functor $\pi^g_0\colon\Cat\to\Set$ that sends a category to  the set of connected components of its sub-groupoid of isomorphisms. We may apply $\pi^g_0$ to the hom-categories of a 2-category $\tcat{C}$ to construct a category $(\pi^g_0)_*\tcat{C}$.  Any isomorphism $K\cong L$ in the category $(\pi^g_0)_*\tcat{C}$ can be lifted to a corresponding equivalence in $\tcat{C}$ by picking representatives $w\colon K\to L$ and $w'\colon L\to K$ in $\tcat{C}$ for the isomorphism and its inverse. The 2-isomorphisms $\alpha\colon w'w\cong\id_K$ and $\beta\colon ww'\cong\id_L$ which witness these as equivalence inverses in $\tcat{C}$ arise by choosing 2-cells which witness the mutual inverse identities $w'w=\id_K$ and $ww'=\id_L$ in $(\pi^g_0)_*\tcat{C}$.
\end{obs}

\begin{lem}\label{lem:sq-as-a-functor}
The functor $\sq_{g,f}$ factorises through the quotient functor $\qCat/(C\times B)\to(\pi^g_0)_*(\qCat_2\slice(C\times B))$ to define a functor
\begin{equation}\label{eq:the-real-sq-functor}
    \sq_{g,f}\colon(\pi^g_0)_*(\qCat_2\slice(C\times B))\op\longrightarrow\Set.
    \end{equation}
\end{lem}
\begin{proof}
If we are given a 2-cell
  \begin{equation*}
    \xymatrix@=1.5em{
      {X}\ar[dr]_(0.3){(c,b)} 
      \ar@/^1.5ex/[rr]^{u}_{}="one" \ar@/_1.5ex/[rr]_{u'}^{}="two" \ar@{=>}"one";"two"^{\tau}
      && *+!L(0.5){Y}\ar[dl]^(0.3){(\bar{c},\bar{b})} \\
      & {C\times B}&
    }
  \end{equation*}
  in $\qCat_2\slice(C\times B)$ and a 2-cell $\beta \in \sq_{g,f}(\bar{c},\bar{b})$ then the middle four interchange rule for the horizontal composite of the 2-cells $\beta$ and $\tau$ provides us with a commutative square
   \begin{equation*}
    \xymatrix@=1.5em{
      {f\bar{b}u} \ar@{=>}[r]^{\beta u}\ar@{=>}[d]_{f\bar{b}\tau} & 
      {g\bar{c}u}\ar@{=>}[d]^{g\bar{c}\tau} \\
      {f\bar{b}u'} \ar@{=>}[r]_{\beta u'} & {g\bar{c}u'}
    }
  \end{equation*} 
  whose vertical arrows are the identities on $fb$ and $gc$ respectively. Hence, $\beta u = \beta u'$, and we conclude that if $u$ and $u'$ are 1-cells in the same connected component of the category $\hom'_{C \times B}((c,b),(\bar{c},\bar{b}))$ then the functions $\sq_{g,f}(u)$ and $\sq_{g,f}(u')$ are identical.
\end{proof}

  This functor allows us to expose another aspect of the weak 2-universal property of weak comma objects: namely that the comma cone formed from the cospan $B \xrightarrow{f} A \xleftarrow{g} C$ represents the functor \eqref{eq:the-real-sq-functor}.

\begin{lem}\label{lem:cpts-and-comma-2-cells}
  The weakly universal comma cone 
  \begin{equation}\label{eq:first-comma-cone}
    \xymatrix@=10pt{
      & f \downarrow g \ar[dl]_{p_1} \ar[dr]^{p_0} \ar@{}[dd]|(.4){\psi}|{\Leftarrow}  \\ 
      C \ar[dr]_g & & B \ar[dl]^f \\ 
      & A}
  \end{equation}
  provides us with an element $\psi\in\sq_{g,f}(p_1,p_0)$ which is universal, in the usual sense, for the functor $\sq_{g,f}\colon(\pi^g_0)_*(\qCat_2\slice(C\times B))\op\to\Set$. Furthermore, any comma cone
   \begin{equation}\label{eq:other-comma-cone}
    \xymatrix@=10pt{
      & Q \ar[dl]_{q_1} \ar[dr]^{q_0} \ar@{}[dd]|(.4){\phi}|{\Leftarrow}  \\ 
      C \ar[dr]_g & & B \ar[dl]^f \\ 
      & A}
  \end{equation} 
  for which the 2-cell $\phi\in\sq_{g,f}(q_1,q_0)$ is a universal element of the functor $\sq_{g,f}$ displays $Q$ as a weak comma object in $\qCat_2$.
\end{lem}

\begin{proof}
  For each object $(c,b)\colon X\to C\times B$ of $\qCat_2\slice(C\times B)$ the element $\psi\in\sq_{g,f}(p_1,p_0)$ induces a function
  \begin{equation*}
\pi^g_0(\hom'_{C\times B}((c,b),(p_1,p_0)))\longrightarrow \sq_{g,f}(c,b)
  \end{equation*}
  which carries a functor $a\colon X\to f\comma g$ representing an element of the set on the left to the whiskered composite $\psi a$ on the right. The element $\psi\in\sq_{g,f}(p_1,p_0)$ is universal for $\sq_{g,f}$ if and only if each of those functions is a bijection. Surjectivity follows directly from the 1-cell induction property of $f\comma g$, and injectivity follows from the reformulation of Lemma \ref{lem:smothering} discussed in Observation~\ref{obs:1cell-ind-uniqueness-reloaded}. 

If $\phi\in\sq_{g,f}(q_1,q_0)$ is another element which is universal for $\sq_{g,f}$, then by Yoneda's lemma the objects $(p_1,p_0)\colon f\comma g\to C\times B$ and $(q_1,q_0)\colon Q\to C\times B$ are isomorphic in the category $(\pi^g_0)_*(\qCat_2\slice(C\times B))$ via an isomorphism whose action under $\sq_{g,f}$ carries $\psi\in\sq_{g,f}(p_1,p_0)$ to $\phi\in\sq_{g,f}(q_1,q_0)$. Proceeding as in Observation \ref{obs:groupoid-components}, we may pick representatives of this isomorphism and its inverse to provide a pair of 1-cells
\begin{equation*}
  \xymatrix@=1.5em{
    {Q} \ar@/_1ex/[rr]_w 
    \ar[dr]_(0.4){(q_1,q_0)} &&
    *+!L(0.5){f\comma g} \ar@/_1ex/[ll]_{w'}
    \ar[dl]^(0.4){(p_1,p_0)}  \\
    & {C\times B}
  }
\end{equation*}
which are related by a pair of 2-isomorphisms $\alpha\colon w'w\cong \id_{Q}$ and $\beta\colon ww'\cong\id_{f\comma g}$ in the slice 2-category $\qCat_2\slice(C\times B)$. The fact that this isomorphism carries $\phi$ to $\psi$ under the action of $\sq_{g,f}$ provides the 2-cellular equations $\psi w = \phi$ and $\phi w' = \psi$. 

  To prove the 1-cell induction property for the comma cone~\eqref{eq:other-comma-cone} suppose that we are given a comma cone~\eqref{eq:comma-cone}. The 1-cell induction property of $f\comma g$ provides us with a 1-cell $a\colon X\to f\comma g$ with the defining property that $p_0 a = b$, $p_1 a = c$, and $\psi a = \alpha$. The functor $w'\colon f\comma g\to Q$ satisfies the equations $q_0 w' = p_0$, $q_1 w' = p_1$, and $\phi w' = \psi$, so we have $q_0 w' a = p_0 a = b$, $q_1 w' a = p_1 a = c$, and $\phi w' a = \psi a = \alpha$. This demonstrates that $w' a\colon X\to Q$ is a 1-cell induced by the comma cone~\eqref{eq:comma-cone} with respect to the comma cone~\eqref{eq:other-comma-cone}.

  To prove the 2-cell induction property for the comma cone~\eqref{eq:other-comma-cone} suppose that we are given a pair of 1-cells $a,a'\colon X\to Q$ and a pair of 2-cells $\tau_0 \colon q_0 a \Rightarrow q_0 a'$ and $\tau_1 \colon q_1 a \Rightarrow q_1 a'$ satisfying the condition given in~\eqref{eq:comma-ind-2cell-compat} with respect to the comma cone~\eqref{eq:other-comma-cone}. The 1-cells $w a, w a'\colon X\to f\comma g$ and the 2-cells $\tau_0 \colon p_0 w a = q_0 a \Rightarrow q_0 a' = p_0 w a'$ and $\tau_1 \colon p_1 w a = q_1 a \Rightarrow q_1 a' = p_1 w a'$ also satisfy the condition given in~\eqref{eq:comma-ind-2cell-compat} with respect to the comma cone~\eqref{eq:first-comma-cone}. Hence, the 2-cell induction property of $f\comma g$ ensures that we have a 2-cell $\mu\colon w a \Rightarrow w a'$ with the defining properties that $p_0 \mu = \tau_0$ and $p_1\mu = \tau_1$. Combining this with the invertible 2-cell $\alpha\colon w'w\cong\id_Q$, we may construct a 2-cell 
  \begin{equation*}
    \xymatrix@C=4em{
      {\tau \defeq a} \ar@{=>}[r]_-{\cong}^-{\alpha^{-1} a} &
      {w'wa} \ar@{=>}[r]^{w'\mu} &
      {w'wa'} \ar@{=>}[r]_-{\cong}^-{\alpha a'} & {a'}
    }
  \end{equation*}
Because $\alpha$ is a 2-cell in the endo-hom-category in $\qCat_2\slice(C\times B)$ on the object $(q_1,q_0)\colon Q\to C\times B$, $q_0\alpha = \id_{q_0}$ and $q_1\alpha = \id_{q_1}$. It follows that $q_0 \tau = q_0 w' \mu = p_0 \mu = \tau_0$ and $q_1 \tau = q_1 w' \mu = p_1 \mu = \tau_1$, 
  which demonstrates that $\tau\colon a\Rightarrow a'$ satisfies the defining properties required of a 2-cell induced by the pair of 2-cells $\tau_0$ and $\tau_1$.

  The proof of 2-cell conservativity is of a similar ilk and is left to the reader.
\end{proof}

%!TEX root = all.tex
% ******************************************************************
% ** Title:            The 2-category theory of quasi-categories
% **                   adjunctions
% ** Precis:        
% ** Author:           Emily Riehl and Dominic Verity
% ** Commenced:        2/3/2012
% ******************************************************************

\section{Adjunctions of quasi-categories}\label{sec:qcatadj}

We begin our 2-categorical development of quasi-category theory by introducing the appropriate notion of adjunction, following Joyal. As observed in \cite{kelly.street:2} and elsewhere, adjunctions can be defined internally to any 2-category and the proofs of many of their familiar properties can be internalised similarly.

\setcounter{thm}{0}
\begin{defn}[adjunction]\label{defn:adjunction}
An {\em adjunction\/}  \[ \adjdisplay f-| u : A ->B .\] in a 2-category consists of objects $A,B$; 1-cells $f \colon B \to A$, $u \colon A \to B$; and {\em unit\/} and {\em counit\/} 2-cells $\eta \colon \id_B \Rightarrow uf$, $\epsilon \colon fu \Rightarrow \id_A$ satisfying the triangle identities.
\[\xymatrix@=1.5em{ & B \ar[dr]^f \ar@{}[d]|(.6){\Downarrow\epsilon} \ar@{=}[rr] &  \ar@{}[d]|(.4){\Downarrow\eta} & B \ar@{}[d]^*+{=} & B &&   B \ar@{=}[rr] \ar[dr]_f & \ar@{}[d]|(.4){\Downarrow \eta} & B \ar[dr]^f \ar@{}[d]|(.6){\Downarrow\epsilon} & {\mkern40mu}\ar@{}[d]^*+{=} &  B \ar@/^2ex/[d]^f \ar@/_2ex/[d]_f \ar@{}[d]|(.4){\id_f}|(.6){=} & \\A \ar[ur]^u \ar@{=}[rr] & &  A \ar[ur]_u & {\mkern40mu} & A \ar@/^2ex/[u]^u \ar@/_2ex/[u]_u \ar@{}[u]|(.4){=}|(.6){\id_u}  && &A \ar[ur]_u \ar@{=}[rr] & & A & A }\]
\end{defn}

In particular, an \emph{adjunction between quasi-categories} is an adjunction in the 2-category $\qCat_2$.  As always we identify the unit and counit 2-cells with the simplicial maps \[ \xymatrix@=1.5em{ B \ar[d]_{i_0} \ar@{=}[dr] & & &  A \ar[r]^u \ar[d]_{i_0} & B \ar[d]^f  \\ B \times \Del^1 \ar[r]_-{\eta} & B & \mathrm{and} &  A \times \Del^1 \ar[r]^-{\epsilon} & A \\ B \ar[u]^{i_1} \ar[r]_f & A \ar[u]_u & &  A \ar[u]^{i_1} \ar@{=}[ur]}\] (1-simplices in $B^B$ and $A^A$ respectively) representing the unit and counit respectively. Because $B^A$ and $A^B$ are quasi-categories we know, from the description of the homotopy category of a quasi-category given in Recollection~\ref{rec:hty-category}, that for any choice of representatives of the unit and counit there exist maps \[ \alpha \colon A \times \Del^2 \to B \qquad \mathrm{and} \qquad \beta \colon B \times \Del^2 \to A\] (2-simplices in $B^A$ and $A^B$ respectively) which witness the triangle identities in the  sense that their boundaries have the form 
 \[ \xymatrix@=1em{ & ufu \ar@{}[d]|(.6){\alpha} \ar[dr]^{u\epsilon} & & & fuf \ar@{}[d]|(.6){\beta} \ar[dr]^{\epsilon f} \\ u \ar[ur]^{\eta u} \ar[rr]_{\id_u} & & u & f \ar[ur]^{f\eta} \ar[rr]_{\id_f} & & f}\]

\begin{ex}
On account of the fully-faithful inclusion $\Cat_2 \inc \qCat_2$, any adjunction of categories gives rise to an adjunction of quasi-categories with canonical representatives for the unit and counit. Conversely, the 2-functor $\ho \colon \qCat_2 \to \Cat_2$ carries any adjunction of quasi-categories to an adjunction between their respective homotopy categories.
\end{ex}

\begin{ex} The homotopy coherent nerve, introduced in \cite{Cordier:1982:HtyCoh} and studied in \cite{Cordier:1986:HtyCoh}, defines a 2-functor from the 2-category of topologically enriched categories, continuous functors, and enriched natural transformations to $\qCat_2$. This 2-functor factors through the 2-category of locally Kan simplicial categories, simplicial functors, and simplicial natural transformations; the locally Kan simplicial categories are the cofibrant objects in Berger's model structure \cite{Bergner:2007fk}.  Hence, any enriched adjunction between topological or fibrant simplicial categories gives rise to an adjunction of quasi-categories by passing to homotopy coherent nerves. As in the unenriched case, there exist canonical representatives for the unit and counit defined by applying the homotopy coherent nerve to the corresponding enriched natural transformations.
\end{ex}

\begin{ex}\label{ex:simp.quillen.adj}
Any simplicially enriched Quillen adjunction between simplicial model categories descends to an adjunction between the associated quasi-categories, constructed by restricting to the fibrant-cofibrant objects and then applying the homotopy coherent nerve. This restriction is necessary to define the quasi-category associated to a simplicial model category; the homotopy coherent nerve of a simplicial category might not be a quasi-category if the simplicial category is not locally Kan. The subcategory of fibrant-cofibrant objects of a simplicial model category is locally Kan, and furthermore the hom-space bifunctor preserves weak equivalences in both variables; it is common to say that only between fibrant-cofibrant objects are the simplicial hom-spaces guaranteed to have the ``correct'' homotopy type. 

In contrast with the topological case, some care is required to define the functors constituting the adjunction; the point-set level functors will not do because neither adjoint need land directly in the fibrant-cofibrant objects. We prove that  a simplicial Quillen adjunction descends to an adjunction of quasi-categories in Theorem~\ref{thm:simplicial-Quillen-adjunction}.
\end{ex}

Adjunctions can also be constructed internally to $\qCat_2$ using its weak 2-limits, as we shall see in the next section. Later, we will also meet adjunctions constructions using limits or colimits defined internally to a quasi-category.

\subsection{Right adjoint right inverse adjunctions}\label{subsec:RARI} 

We begin by studying an important class of adjunctions whose counit 2-cells are isomorphisms.

\begin{defn}\label{defn:RARI}
A 1-cell $f \colon B \to A$ in a 2-category admits a \emph{right adjoint right inverse} (abbreviated \emph{RARI}) if it admits a right adjoint $u \colon A \to B$ so that the counit of the adjunction $f \dashv u$ is an isomorphism.
\end{defn}

In the situation of Definition \ref{defn:RARI}, $f$ defines a \emph{left adjoint left inverse} (abbreviated \emph{LALI}) to $u$. When the counit of $f \dashv u$ is an isomorphism, the whiskered composites $f\eta$ and $\eta u$ of the unit must also be isomorphisms. Indeed, to construct an adjunction of this form it suffices to give 2-cells with these properties, as demonstrated by the following 2-categorical lemma. 

\begin{lem}\label{lem:adjunction.from.isos} Suppose we are given a pair of 1-cells $u \colon A \to B$ and $f\colon B\to A$ and a 2-isomorphism $fu \cong \id_A$ in a 2-category. If there exists a 2-cell $\eta' \colon \id_B \Rightarrow uf$ with the property that $f\eta'$ and $\eta' u$ are 2-isomorphisms, then $f$ is left adjoint to $u$. Furthermore, in the special case where $u$ is a section of $f$, then $f$ is left adjoint to $u$ with the counit of the adjunction an identity.
\end{lem}
\begin{proof} 
Let $\epsilon \colon fu \To \id_A$ be the isomorphism, taken to be the identity in the case where $u$ is a section of $f$. We will define an adjunction $f \dashv u$ with counit $\epsilon$ by modifying $\eta' \colon \id_B \To uf$.  The ``triangle identity composite'' $\theta\defeq u\epsilon \cdot \eta' u \colon u \To u$ defines an automorphism of $u$.  Define 
\[ \eta \defeq \xymatrix{ \id_B \ar@{=>}[r]^-{\eta'} & uf \ar@{=>}[r]^-{\theta^{-1}f} & uf.}\] Immediately, $u\epsilon \cdot \eta u = \id_u$, as is verified by the calculation: 
\begin{equation}\label{eq:triangle-calculation-1} \xymatrix@R=1.2em@C=2.5em{ u \ar@{=>}[dr]_\theta \ar@{=>}[r]^-{\eta' u} & ufu \ar@{=>}[r]^{\theta^{-1} fu} \ar@{=>}[d]^{u \epsilon} & ufu \ar@{=>}[d]^{u \epsilon} \\ & u \ar@{=>}[r]_{\theta^{-1}} & u}\end{equation} 

The other triangle identity composite $\phi\defeq \epsilon f \cdot f \eta $ is an isomorphism, as a composite of isomorphisms, and also an idempotent:
\begin{equation}\label{eq:triangle-calculation-2} \xymatrix@R=1.2em@C=2.5em{ f \ar@{=>}[d]_{f \eta} \ar@{=>}[r]^-{f \eta} & fuf \ar@{=}[dr] \ar@{=>}[d]_{f\eta uf} \\ fuf \ar@{=>}[d]_{\epsilon f} \ar@{=>}[r]_-{fuf\eta} & fufuf \ar@{=>}[d]^{\epsilon f} \ar@{=>}[r]_-{fu\epsilon f} & fuf \ar@{=>}[d]^{\epsilon f} \\ f \ar@{=>}[r]_-{f\eta} & fuf \ar@{=>}[r]_-{\epsilon f} & f}\end{equation} But any idempotent isomorphism is an identity: the isomorphism $\phi$ can be cancelled from both sides of the idempotent equation $\phi \cdot \phi = \phi$. Hence, $\epsilon f \cdot f \eta = \id_f$, proving the second triangle identity.
\end{proof}

\begin{rmk}[idempotent isomorphisms]\label{rmk:idempotent-isomorphisms} Because $\qCat_2$ has many weak but few strict 2-limits, it is frequently easier to show that a 2-cell is an isomorphism than to show that it is an identity. When we desire an identity and not merely an isomorphism,  we will make frequent use of the trick that any idempotent isomorphism is an identity.
\end{rmk}

We now show that for any functor $\ell \colon C \to B$, the codomain projection functor $\pi_1 \colon B \comma \ell \to C$  admits a right adjoint right inverse, the ``identity functor'' $i \colon C \to B \comma \ell$ defined below. Here the right adjoint $i$ defines a section to the left adjoint $p_i$. Taking the counit of $i \dashv \pi_1$ to be an identity, as permitted by Lemma \ref{lem:adjunction.from.isos}, the adjunction lifts to the slice 2-category $\qCat_2/C$.

\begin{lem}\label{lem:technicalsliceadjunction}
  Suppose that $\ell\colon C\to B$ is a functor of quasi-categories and let $i\colon C\to B\comma\ell$ be any functor induced by the identity comma cone:
  \begin{equation}\label{eq:technicalsliceadjunction}
    \vcenter{\xymatrix@=1em{
      & {C}\ar@{=}[dl]\ar[dr]^{\ell} & \\
      {C}\ar[rr]_{\ell} && {B}
      \ar@{} "1,2";"2,2" |(0.6){\textstyle =}
    }}
    \mkern20mu = \mkern20mu
    \vcenter{\xymatrix@=1em{
      & {C}\ar[d]^-i \ar@/^/[ddr]^\ell \ar@/_/@{=}[ddl] & \\
      & {B\comma\ell}\ar[dl]|{p_1}\ar[dr]|{p_0} & \\
      {C}\ar[rr]_{\ell} && {B}
      \ar@{} "2,2";"3,2" |(0.6){\Leftarrow\phi}
    }}
  \end{equation}
  Then $i\colon C\to B\comma\ell$ is right adjoint to the codomain projection functor $p_1\colon B\comma\ell\to C$ in the slice 2-category $\qCat_2\slice C$ 
  \[    \xymatrix@=1.2em{
      {C}\ar@{=}[dr]\ar@/_1.5ex/[rr]_-{i}^-{}="one"
      & & *+!L(0.5){B\comma\ell}\ar@{->>}[dl]^-{p_1}
      \ar@/_1.5ex/[ll]_-{p_1}^-{}="two" \\
      & {C} &
      \ar@{}"one";"two"|{\bot}
    }
    \]
 and the counit may be chosen to be an identity 2-cell.
\end{lem}

\begin{proof}
By construction, $i$ is a section to the isofibration $p_1$ and, accordingly, we may take the counit of the postulated adjunction to be the identity $p_1 i = \id_C$. Now a 2-cell $\nu \colon \id_{B\comma \ell} \Rightarrow i p_1$ provides us with a 2-cell in $\qCat_2\slice C$ which satisfies the triangle identities with respect to that counit if and only if $p_1\nu$ and $\nu i$ are identity 2-cells. 

  We construct a suitable 2-cell $\nu\colon \id_{B\comma \ell} \Rightarrow i p_1$ by applying the 2-cell induction property of $B\comma\ell$ to the pair of 2-cells $\phi\colon p_0 \Rightarrow \ell p_1 = p_0 i p_1$ and $\id_{p_1}\colon p_1 = p_1 i p_1$; here, the compatibility condition of~\eqref{eq:comma-ind-2cell-compat} reduces to the trivial pasting identity \[ \vcenter{\xymatrix@=0.7em{ & B\comma \ell \ar@/^2ex/[ddr]^{p_0}   \ar@/^2ex/[ddl]|*+<3pt>{\scriptstyle p_1} \ar@/_2ex/[ddl]_{p_1}  \ar@{}[ddl]|{=} \\ &  \ar@{}[dr]|(.3){\Leftarrow\phi} \\ C \ar[rr]_\ell & & B}} \mkern20mu = \mkern20mu\vcenter{ \xymatrix@=0.7em{ & B\comma \ell \ar@/^2ex/[ddr]^{p_0}  \ar@/_2ex/[ddr]|*+<3pt>{\scriptstyle\ell p_1} \ar@{}[ddr]|{\Leftarrow\phi}  \ar@/_2ex/[ddl]_{p_1}  \\ & \ar@{}[dl]|(0.3){=} & \\ C \ar[rr]_\ell & & B}}  \] By construction, $\nu\colon \id_{B\comma \ell} \Rightarrow i p_1$ is a 2-cell satisfying $p_0\nu = \phi$ and $p_1\nu = \id_{p_1}$.

To show that $\nu i$ is an isomorphism,  observe that $p_0\nu i = \phi i = \id_\ell$ and $p_1\nu i = \id_{p_1} i = \id_{p_1 i} = \id_{\id_C}$, so  using the 2-cell conservativity property of $B\comma\ell$ we conclude that $\nu i$ is an isomorphism.  By Lemma \ref{lem:adjunction.from.isos} this suffices; indeed, applying middle-four interchange to $\nu i \cdot \nu i$ and the equation $p_1\nu=\id_{p_1}$,  $\nu i$ can be seen to be an idempotent isomorphism and thus an identity.
\end{proof}

In general, if a (representable) isofibration $f \colon B \tfib A$ admits a right adjoint right inverse $u$, then the counit of the RARI adjunction may be chosen to be an identity. Lemma \ref{lem:representable-isofibration}, which shows that an isofibration between quasi-categories defines a representable isofibration in $\qCat_2$, will allow us to make frequent use of this ``strictification'' result.

\begin{lem}\label{lem:isofibration-RARI} If $f \colon B \tfib A$ is a representable isofibration in a 2-category $\tcat{C}$ admitting a right adjoint right inverse $u' \colon A \to B$, then there exists a 1-cell $u \colon A \to B$ that is right adjoint right inverse to $f$ with identity counit.
\end{lem}
\begin{proof}
We construct the functor $u\colon A\to B$ and an isomorphism $\beta\colon u' \cong u$ by applying the universal property of the isofibration $f\colon B\tfib A$ to the counit $\epsilon'\colon fu'\cong\id_A$.
\[\xymatrix{ \ar@{}[dr]|(.7){\epsilon'\cong} & B \ar@{->>}[d]^f  & \ar@{}[d]|{\displaystyle\rightsquigarrow} &  \ar@{}[dr]|{\beta\cong} & B \ar@{->>}[d]^f \\ A \ar[ur]^{u'} \ar@{=}[r] & A &&  A \ar@/^1.5ex/[ur]^{u'} \ar@/_1.5ex/[ur]_{u} \ar@{=}[r] & A}\] By construction $fu = \id_A$. The composite
 $\eta \defeq \xymatrix@1{\id_B \ar@{=>}[r]^{\eta'} & u'f \ar@{=>}[r]^{\beta f} & uf}$ of the original unit $\eta'$ with the lifted isomorphism $\beta$ defines a 2-cell that whiskers with $f$ and $u$ to isomorphisms, permitting the application of Lemma \ref{lem:adjunction.from.isos} to conclude.
\end{proof}

\subsection{Terminal objects as adjoint functors}\label{subsec:terminal}

A quasi-category $A$ has a terminal object if and only if the projection functor $! \colon A \to \Del^0$ admits a right adjoint right inverse:

\begin{defn}[terminal objects]\label{defn:terminal}
An object $t$ in a quasi-category $A$ is {\em terminal\/} if there is an adjunction \[ \adjdisplay !-| t :\Del^0 -> A . \]  
 \end{defn}

Dually, of course, an object in $A$ is initial just when it defines a left adjoint left inverse to $! \colon A \to \Delta^0$.

 \begin{ex}[slices have terminal objects]\label{ex:slice-terminal}
For any object $a$  of a quasi-category $A$, there is an adjunction \[\adjdisplay !-| i : \Del^0 -> A\comma a .\] whose right adjoint, defining the terminal object of $A \comma a$, is any vertex of $A \comma a$ that is isomorphic to the degenerate 1-simplex $a\cdot\degen^0\colon a\to a$. This functor whiskers with the comma cone to an identity 2-cell:
\[     \vcenter{\xymatrix@=1em{
      & {\Del^0}\ar@{=}[dl]\ar[dr]^{a} & \\
      {\Del^0}\ar[rr]_{a} && {A}
      \ar@{} "1,2";"2,2" |(0.6){\textstyle =}
    }}
    \mkern20mu = \mkern20mu
    \vcenter{\xymatrix@=1em{
      & {\Del^0}\ar[d]^-i \ar@/^/[ddr]^a \ar@/_/@{=}[ddl] & \\
      & {A\comma a}\ar[dl]|{p_1}\ar[dr]|{p_0} & \\
      {\Del^0}\ar[rr]_{a} && {A}
      \ar@{} "2,2";"3,2" |(0.6){\Leftarrow\phi}
    }}\]
Thus, the adjunction $!\dashv i$ is a special case of Lemma \ref{lem:technicalsliceadjunction}.
\end{ex}

Lemma \ref{lem:adjunction.from.isos} allows us to describe the minimal information required to display a terminal object.

\begin{lem}[minimal information required to display a terminal object]\label{lem:min-term-pres}
  To demonstrate that an object $t$ is terminal in $A$ it is enough to provide a unit 2-cell $\eta\colon\id_A\Rightarrow t!$ for which the whiskered composite $\eta t$ is an isomorphism.
 \end{lem}

When $A$ is a category this presentation is neither more nor less than the well known observation that an object $t$ is terminal in $A$ if and only if there exists a cocone on the identity diagram with vertex $t$ whose component at $t$ is an isomorphism. The proof of this lemma applies in any 2-category which possesses a 2-terminal object.

\begin{proof}
The categories $\hom'(\Del^0,\Del^0)$ and $\hom'(A,\Del^0)$ are both isomorphic to the terminal category $\catone$, so the counit is necessarily taken to be the identity and one of the triangle identities arises trivially. By Lemma \ref{lem:adjunction.from.isos} it remains only to provide a unit $\eta \colon \id_A \To t!$ for which the whiskered composition $\eta t$ is an isomorphism.  Specialising the proof of Lemma \ref{lem:adjunction.from.isos}, it follows formally that $\eta t \colon t \To t$ is an idempotent isomorphism and hence an identity, as required.
\end{proof}

The following straightforward 2-categorical lemma provides us with a useful ``external'' characterisation of terminal objects in quasi-categories.

\begin{lem}\label{lem:adj-ext-univ}
  Suppose we are given a pair of 1-cells $u\colon A\to B$ and $f\colon B\to A$ and a 2-cell $\epsilon\colon fu\Rightarrow\id_A$ in a 2-category $\tcat{C}$. Then $f$ is left adjoint to $u$ with counit $\epsilon$ in $\tcat{C}$ if and only if for all 0-cells $X\in\tcat{C}$ the functor $\tcat{C}(X,f)\colon\tcat{C}(X,B)\to\tcat{C}(X,A)$ is left adjoint to $\tcat{C}(X,u)\colon\tcat{C}(X,A)\to\tcat{C}(X,B)$, in the usual sense, with counit $\tcat{C}(X,\epsilon)$. 
\end{lem}

\begin{proof}
  The only if direction is immediate on observing that $\tcat{C}(X,-)$ is a 2-functor and thus preserves adjunctions. For the converse, we observe that the family of units of the adjunctions $\tcat{C}(X,f)\dashv\tcat{C}(X,u)$ is 2-natural in $X$ and so the 2-categorical Yoneda lemma provides us with a 2-cell $\eta\colon\id_B\Rightarrow uf$ with the property that $\tcat{C}(X,\eta)$ and $\tcat{C}(X,\epsilon)$ are unit and counit of the adjunction  $\tcat{C}(X,f)\dashv\tcat{C}(X,u)$. A further application of the 2-categorical Yoneda lemma demonstrates that the triangle identities for $\eta$ and $\epsilon$ follow immediately from those for $\tcat{C}(X,\eta)$ and $\tcat{C}(X,\epsilon)$.
\end{proof}

\begin{prop}\label{prop:terminal-ext-univ}
A vertex $t$ in a quasi-category $A$  is terminal if and only if for all $X$ the constant functor $\xymatrix@1{{X}\ar[r]^{!} & {\Del^0}\ar[r]^t & {A}}$ is terminal, in the usual sense, in the hom-category $\hom'(X,A)$.
\end{prop}

\begin{proof}
Apply Lemma~\ref{lem:adj-ext-univ} to the functors $t\colon\Del^0\to A$ and $!\colon A\to \Del^0$ and the identity natural transformation $!t = \id_{\Del^0}$.
\end{proof}

We conclude by comparing our definition of terminal object with its antecedent.

\begin{ex}\label{ex:terminaldefn} Joyal defines a vertex $t$ in a quasi-category $A$ to be terminal if and only if any sphere $\partial\Delta^n \to A$ whose final vertex is $t$ has a filler \cite[4.1]{Joyal:2002:QuasiCategories}. In Proposition~\ref{prop:terminalconverse}, we will show that Joyal's definition is equivalent to ours. For the moment, however, we shall at least take some satisfaction in convincing ourselves directly that his notion implies ours.

Supposing that $t \in A$ is terminal in Joyal's sense, then to define an adjunction $\adjinline !-| t :\Del^0 -> A .$ we wish to define a unit $\eta\colon\id_A\Rightarrow t!$ for which $\eta t$ is an identity. This unit is represented by a map \[ \xymatrix@=1.5em{ A \ar[d]_-{i_0} \ar@{=}[dr] \\ A \times \Delta^1 \ar[r]_-\eta & A \\  A \ar[u]^{i_1} \ar[r]_{!} & \Delta^0 \ar[u]_t} \] which we define as follows. For each $a \in A_0$, use the universal property of $t$ to choose a 1-simplex $\eta a \colon \Delta^1 \to A$ from $a$ to $t$. We take care to pick $\eta t$ to be the degenerate 1-simplex at $t$, thus ensuring that the 2-cell $\eta t$ will be the identity at $t$ as required by Lemma \ref{lem:min-term-pres}.

To define $\eta \colon A \to A^{\Delta^1}$ it suffices to inductively specify maps $\Delta^n \xrightarrow{\sigma} A \xrightarrow{\eta} A^{\Delta^1}$ for each non-degenerate $\sigma \in A_n$ compatibly with taking faces of $\sigma$. The map $\eta (\sigma \times \id_{\Delta^1}) \colon \Delta^n \times \Delta^1 \to A$ should be thought of as the component of $\eta$ at $\sigma$. The chosen 1-simplices $\eta a$ define the components at the vertices $a \in A_0$.

For each non-degerate $\alpha \colon a \to a' \in A_1$, define a cylinder $\Delta^1 \times \Delta^1 \to A$ as follows. The 1-skeleton consists of the displayed 1-simplices.
\[ \xymatrix{ a \ar[d]_\alpha \ar[r]^{\eta a} \ar[dr]|{\eta a} & t \ar@{=}[d]^{t\cdot\sigma^0} \\ a' \ar[r]_{\eta a'} & t} \] 
One shuffle is defined by degenerating $\eta a$. The other is chosen by applying the universal property of $t$ to the sphere formed by $\alpha$, $\eta a$, and $\eta a'$.

Continuing inductively, suppose we have chosen, for each $\sigma \in A_n$, a cylinder $\Delta^n \times \Delta^1 \to A$ from $\sigma$ to the degenerate $n$-simplex at $t$ in such a way that these choices are compatible with the face and degeneracy maps from the $n$-truncation $\sk_n\Del$ of $\Del$. Given a non-degenerate simplex $\tau \in A_{n+1}$, this simplex together with the $(n+1)$-simplices chosen for each of its $n$-dimensional faces $\tau\delta^i$ form an $(n+2)$-sphere with final vertex $t$, and we may choose a filler $\hat{\tau} \in A_{n+2}$. Define the requisite cylinder, the component of $\eta$ at $\tau$, to be the composite \[ \Delta^{n+1} \times \Delta^1 \xrightarrow{q} \Delta^{n+2} \xrightarrow{\hat{\tau}} A\] of $\hat{\tau}$ with the map induced by the functor $q \colon [n+1] \times [1] \to [n+2]$ defined by $q(i,0) = i$ and $q(i,1) = n+2$. By construction, $\hat{\tau} \face^i = \hat{\tau \face^i}$ for each $0 \leq i \leq n+1$, that is, the $i^{\th}$ face of the sphere whose filler defines $\hat{\tau}$ is the $(n+1)$-simplex chosen to fill the corresponding sphere for $\tau\delta^i$; thus the cylinder for $\tau$ is chosen compatibly with its faces. 
\end{ex}

This example will be generalised in Proposition \ref{prop:limitsasadjunctions} to limits of arbitrary shape.

\subsection{Basic theory}

A key advantage to our 2-categorical definition of adjunctions is that formal category theory supplies easy proofs of a number of desired results.

\begin{prop}\label{prop:adj-comp} A pair of adjunctions $\adjinline f -| u : A -> B.$ and $\adjinline f' -| u' : B -> C.$  in a 2-category compose to give an adjunction $\adjinline ff' -| u'u : A -> C.$. In particular, we may compose adjunctions of quasi-categories.
\end{prop}

\begin{proof}
The unit and counit of the composite adjunction are \[ \xymatrix@=10pt{ C \ar[dr]_{f'} \ar@{=}[rrrr] & & \ar@{}[d]|(.4){\Downarrow\eta'}& & C & & & C \ar[dr]^{f'} \ar@{}[d]|(.6){\Downarrow\epsilon'} \\ & B \ar[dr]_f \ar@{=}[rr] & \ar@{}[d]|(.4){\Downarrow\eta} & B \ar[ur]_{u'}  & & & B \ar[ur]^{u'} \ar@{=}[rr] & \ar@{}[d]|(.6){\Downarrow\epsilon} & B \ar[dr]^f \\  & & A \ar[ur]_u & &  & A \ar[ur]^u \ar@{=}[rrrr] & & & & A} \qedhere\] 
\end{proof}

Recall Proposition \ref{prop:equivsareequivs}, which demonstrates that equivalences in $\qCat_2$ are exactly the weak equivalences between quasi-categories in the Joyal model structure. The following classical 2-categorical result allows us to promote any equivalence to an adjoint equivalence (cf.~\cite[IV.4.1]{Maclane:1971:CWM}):

\begin{prop}\label{prop:equivtoadjoint} Any equivalence $w\colon A\to B$ in a 2-category may be promoted to an adjoint equivalence in which $w$ may be taken to be either the left or right adjoint. In particular, we may promote equivalences of quasi-categories to adjoint equivalences.
\end{prop}

\begin{proof}
  This is an immediate corollary of Lemma~\ref{lem:adjunction.from.isos}. 
\end{proof}

 \begin{prop}\label{prop:expadj} Suppose $\adjinline f -| u : A -> B.$ is an adjunction of quasi-categories. For any simplicial set $X$ and any quasi-category $C$, \[ \adjdisplay f^X-| u^X : A^X-> B^X .\qquad \text{and} \qquad \adjdisplay C^u -| C^f : C^A -> C^B .\] are adjunctions of quasi-categories. 
 \end{prop}
 \begin{proof}
By \ref{prop:qcat2closed} and \ref{rmk:exp2functor}, exponentiation defines 2-functors $(-)^X \colon \qCat_2 \to \qCat_2$ and $C^{(-)} \colon \qCat_2\op \to \qCat_2$, which preserve adjunctions.
 \end{proof}

As an easy corollary of the last few results, terminal objects are preserved by right adjoints, initial objects are preserved by left adjoints, and they are both preserved by equivalences.

\begin{prop}\label{prop:terminaldefn} If $u \colon A \to B$ is a right adjoint or an equivalence of quasi-categories and $t$ is a terminal object of $A$, then $ut$ is a terminal object in $B$.
\end{prop}
\begin{proof} By Proposition~\ref{prop:equivtoadjoint}, if $u$ is an equivalence then it may be promoted to a right adjoint, which reduces preservation by equivalences to preservation by right adjoints. Now Proposition~\ref{prop:adj-comp} tells us that we may compose the adjunction in which $u$ features with that which displays $t$ as a terminal object in $A$ to obtain an adjunction which displays $ut$ as a terminal object in $B$.
\end{proof}

\subsection{The universal property of adjunctions}

An essential point in the proof of the main existence theorem of \cite{RiehlVerity:2012hc} is that adjunctions between quasi-categories, while defined equationally, satisfy a universal property. In the terminology introduced there, any adjunction between quasi-categories extends to a \emph{homotopy coherent adjunction}. By contrast, a monad in $\qCat_2$ need not underlie a homotopy coherent monad. In this subsection, we provide several forms of the universal property held by an adjunction. 

Given an adjunction, we form the comma quasi-categories 
\begin{equation}\label{eq:commaobjdefn} \xymatrix@=1.5em{ f \comma A \ar[d]_{(p_1,p_0)} \ar[r] \pbexcursion & A^\cattwo \ar[d]  & & B \comma u \ar[d]_{(q_1,q_0)} \ar[r] \pbexcursion & B^\cattwo \ar[d] \\ A \times B \ar[r]_{\id_A \times f} & A \times A & & A \times B \ar[r]_{u \times \id_B} & B \times B}\end{equation} as in Definition~\ref{def:comma-obj}.  These quasi-categories are equipped with 2-cells
\[
\xymatrix@=15pt{ & f \comma A \ar[dl]_{p_1} \ar[dr]^{p_0} \ar@{}[d]|(.6){\Leftarrow\alpha} & && & B \comma u \ar[dl]_{q_1} \ar[dr]^{q_0} \ar@{}[d]|(.6){\Leftarrow\beta} \\ A & & B \ar[ll]^f && B \ar[rr]_u & & A}
\]
satisfying the weak 2-universal properties detailed in Observation~\ref{obs:unpacking-weak-comma-objects}. Mimicking the standard argument, we derive a fibred equivalence $f \comma A \simeq B \comma u$ from the unit and counit of our adjunction.

\begin{prop}\label{prop:adjointequiv} If $\adjinline f -| u : A -> B.$ is an adjunction of quasi-categories, then there is a fibred equivalence between the objects $(p_1,p_0)\colon f \comma A\tfib A\times B$ and $(q_1,q_0)\colon B \comma u\tfib A\times B$. 
\end{prop}
\begin{proof}
The composite 2-cells displayed on the left of the equalities below give rise to functors $w\colon B \comma u \to f \comma A$ and $w'\colon f \comma A \to B \comma u$ by 1-cell induction: 
\begin{equation*}\xymatrix@C=10pt{ & B \comma u \ar[dl]_{q_1} \ar[dr]^{q_0} \ar@{}[d]|(.6){\Leftarrow\beta} &  &  & B \comma u \ar[d]^{w}  & && & f \comma A \ar[dl]_{p_1} \ar[dr]^{p_0} \ar@{}[d]|(.6){\Leftarrow\alpha} & & & f \comma A \ar[d]^{w'}  \\ A \ar@{=}[dr] \ar[rr]^u & \ar@{}[d]|{\Leftarrow\epsilon} & B  \ar[dl]^f & = & f \comma A \ar[dl]_{p_1} \ar[dr]^{p_0} \ar@{}[d]|(.6){\Leftarrow\alpha} & & & A  \ar[dr]_u & \ar@{}[d]|{\Leftarrow\eta} & B \ar@{=}[dl] \ar[ll]_f & = & B \comma u \ar[dl]_{q_1} \ar[dr]^{q_0} \ar@{}[d]|(.6){\Leftarrow\beta}  \\ & A & & A  & & B \ar[ll]^f & & & B & & A \ar[rr]_u & & B}\end{equation*} 
By these defining pasting identities, the induced functors provide us with 1-cells 
\begin{equation*}
    \xymatrix{ f\comma A \ar@{->>}[dr]_{(p_1,p_0)} \ar@/^1ex/[rr]^{w'} & & B \comma u \ar@{->>}[dl]^{(q_1,q_0)} \ar@/^1ex/[ll]^w \\ & A \times B} 
\end{equation*}
in the slice 2-category $\qCat_2\slice(A\times B)$ commuting with the canonical isofibrations to $A \times B$. These identities give rise to the following sequence of pasting identities 
\begin{equation*}
\xymatrix@C=10pt{ & f \comma A \ar[d]^{w'} & & & f \comma A \ar[d]^{w'} & & & f \comma A \ar[dl]_{p_1} \ar[dr]^{p_0}\ar@{}[d]|(.6){\Leftarrow\alpha} & \ar@{}[dr]|*+{=} & & f \comma A  \ar[dl]_{p_1} \ar[dr]^{p_0} \ar@{}[d]|(.6){\Leftarrow\alpha} \\ & B \comma u \ar[d]^{w} & \ar@{}[d]|*+{=} & & B \comma u \ar[dl]_{q_1} \ar[dr]^{q_0} \ar@{}[d]|(.6){\Leftarrow\beta}  & {=}  & A  \ar[dr]_u \ar@{=}[dd] &\ar@{}[d]|{\Leftarrow\eta} & B \ar[ll]_f \ar@{=}[dl]  & A & & B \ar[ll]^f \\ & f \comma A \ar[dl]_{p_1} \ar[dr]^{p_0} \ar@{}[d]|(.6){\Leftarrow\alpha} & & A \ar@{=}[dr] \ar[rr]^u & \ar@{}[d]|(0.4){\Leftarrow\epsilon}  & B  \ar[dl]^f  &\ar@{}[r]|(0.4){\Leftarrow\epsilon} &  B \ar[dl]^f  &  \\ A & & B \ar[ll]^f & & A & & A& &  }
\end{equation*} 
in which the last step is an application of one of the triangle identities of the adjunction $f\dashv u$. This tells us that the endo-1-cells $ww'$ and $\id_{f\comma A}$ on the object $(p_1,p_0)\colon f\comma A\tfib A\times B$ in $\qCat_2\slice(A\times B)$ both map to the same 2-cell $\alpha$ under the whiskering operation. Applying Lemma \ref{lem:1cell-ind-uniqueness} (or Observation~\ref{obs:1cell-ind-uniqueness-reloaded}), 
we find that  $ww'$ and $\id_{f\comma A}$ are connected by a 2-isomorphism in $\qCat_2\slice(A\times B)$. A dual argument provides us with a 2-isomorphism between the 1-cells $w' w$ and $\id_{B\comma u}$ in the groupoid of endo-cells on $(q_1,q_0)\colon B\comma u\tfib A\times B$. This data provides us with an equivalence in the slice 2-category $\qCat_2\slice(A\times B)$, which we may lift along the smothering 2-functor of Proposition~\ref{prop:slice-smothering-2-functor} to give a fibred equivalence over $A\times B$.
\end{proof}

Just as in ordinary category theory, the Proposition~\ref{prop:adjointequiv} has a converse:

\begin{prop}\label{prop:adjointequivconverse} Suppose we are given functors $u \colon A \to B$ and $f\colon B\to A$ between quasi-categories. If there is a fibred equivalence between $(p_1,p_0)\colon f \comma A\tfib A\times B$ and $(q_1,q_0)\colon B \comma u\tfib A\times B$, then $f$ is left adjoint to $u$.
\end{prop}

    Schematically the proof of this result proceeds by observing that the image of the identity morphism at $f$ under the equivalence $f \comma A \simeq B \comma u$ defines a candidate unit for the desired adjunction. This can then be shown to have the appropriate universal property; the proof, however is slightly subtle. We delay it to the next section, where it will appear as a special case of a more general result needed there.

\begin{obs}[the hom-spaces of a quasi-category]\label{obs:pointwise-adjoint-correspondence} One model for the hom-space between a pair of objects $a$ and $a'$ in a quasi-category $A$ is the comma quasi-category $a \comma a'$, denoted by $\mathrm{Hom}_A(a,a')$ in~ \cite{Lurie:2009fk}. Proposition~\ref{prop:weakcomma} tells us that the canonical comparison $\ho(a\comma a')\to \ho(a)\comma \ho(a')$ from the homotopy category of this hom-space is a smothering functor. Its codomain $\ho(a)\comma\ho(a')$ is a comma category of arrows between a fixed pair of objects in the category $\ho A$, so it is simply the discrete category whose objects are the arrows from $a$ to $a'$ in $\ho A$. It follows from conservativity of the smothering functor that all arrows in $\ho(a\comma a')$ and thus also $a\comma a'$ are isomorphisms; hence,  $a\comma a'$ is a Kan complex by Joyal's result \cite[1.4]{Joyal:2002:QuasiCategories}.

By Observation~\ref{obs:fibred-pullback}, the fibred equivalence of Proposition~\ref{prop:adjointequiv} may be pulled back along the functor $(a,b)\colon\Del^0\to A\times B$ associated with any pair of vertices $a\in A$ and $b\in B$ to give an equivalence $fb \comma a \simeq b \comma ua$ of hom-spaces. This should be thought of as a quasi-categorical analog of the usual adjoint correspondence defined for arrows between a fixed pair of objects $b \in B$ and $a \in A$. 
\end{obs}

\begin{rmk}\label{rmk:vs-lurie-adjunction}
Observation \ref{obs:pointwise-adjoint-correspondence} demonstrates that the 2-categorical definition of an adjunction implies the definition of adjunction given by Lurie in \cite[5.2.2.8]{Lurie:2009fk}. As his definition has a more complicated form, we prefer not to recall it here. It is in fact precisely equivalent to Joyal's 2-categorical definition \ref{defn:adjunction}. Our preferred proof that Lurie's definition implies Joyal's makes use of the fact that the domain and codomain projections from comma quasi-categories are, respectively, cartesian and cocartesian fibrations. A proof will appear in \cite{RiehlVerity:2015fy}, which gives new 2-categorical definitions of these notions, which, when interpreted in $\qCat_2$, recapture precisely the (co)cartesian fibrations of \cite{Lurie:2009fk}.
\end{rmk}

We may apply Proposition~\ref{prop:adjointequiv}  to give a converse to Example~\ref{ex:terminaldefn}, proving that our notion of terminal objects is equivalent to Joyal's. The proof requires one combinatorial lemma, which relates certain comma quasi-categories with Joyal's slices, which are recalled in \ref{defn:slices} and \ref{rmk:map-slices}.

\begin{lem}\label{lem:slice-equiv-comma} For any vertex $a$ in a quasi-category $A$, there is an equivalence
\[\xymatrix@=1em{ \slicer{A}{a} \ar@{->>}[dr] \ar[rr]^-\sim & & A \comma a \ar@{->>}[dl] \\ & A} \] over $A$, which pulls back along any $f \colon B \to A$ to define an equivalence $\slicer{f}{a} \simeq f \comma a$ over $B$.
\end{lem}
\begin{proof}
The result follows from an isomorphism $A \comma a \cong \fatslicer{A}{a}$ between the comma and the fat slice construction reviewed in Definition \ref{defn:fat-slices}. The map $\slicer{A}{a} \to A \comma a$ and the equivalence over $A$ are then special cases of Proposition \ref{prop:slice-fatslice-equiv}. To establish the isomorphism, it suffices to show that $A \comma a$ has the universal property that defines $\fatslicer{A}{a}$.  By adjunction, a map $X \to \fatslicer{A}{a}$ corresponds to a commutative square, as displayed on the left:
\[ \vcenter{\xymatrix{ X \coprod X \ar[d] \ar[r]^-{\pi_X \coprod !} & X \coprod \Del^0 \ar[d]^{ (f, a)} \\ X \times \Del^1 \ar[r]_-k & A}} \qquad \leftrightsquigarrow \qquad \vcenter{\xymatrix{ X \ar[r]^-k \ar[d]_{(!,f)} & A^{\Del^1} \ar[d] \\  \Del^0 \times A \ar[r]_-{ a \times \id_A} & A \times A}}\]
which transposes to the commutative square displayed on the right. The data of the right-hand square is precisely that of a map $X \to A \comma a$ by the universal property of the pullback \ref{def:comma-obj} defining the comma quasi-category.

The isomorphism $A \comma a \cong \fatslicer{A}{a}$ pulls back to define an isomorphism $f \comma a \cong \fatslicer{f}{a}$. The map $\slicer{f}{a} \to f \comma a$ is then an equivalence over $B$ by Remark \ref{rmk:map-slices}.
\end{proof}

\begin{prop}\label{prop:terminalconverse} A vertex $t \in A$ is terminal in the sense of  Joyal's \cite[4.1]{Joyal:2002:QuasiCategories} if and only if \[\adjdisplay !-| t :\Del^0 -> A . \] is an adjunction of quasi-categories. 
\end{prop}
\begin{proof}
The ``if'' direction is Example~\ref{ex:terminaldefn}. For the converse implication, an adjunction $! \dashv t$ gives rise to an equivalence between $!\comma \Delta^0 \cong A$ and $A \comma t$ over $A$ by Proposition~\ref{prop:adjointequiv}. Hence, by the 2-of-3 property of equivalences, the isofibration $A \comma t \tfib A$ is a trivial fibration. Lemma \ref{lem:slice-equiv-comma} supplies an equivalence \[\xymatrix@=1em{ \slicer{A}{t} \ar@{->>}[dr]_\sim \ar[rr]^-\sim & & A \comma t \ar@{->>}[dl] \\ & A} \] between our comma quasi-category and Joyal's slice quasi-category; see \ref{defn:slices} for a definition. Applying the 2-of-3 property again, it follows that the isofibration $\slicer{A}{t} \tfib A$ is a trivial fibration; the right lifting property against the boundary inclusions $\boundary\Delta^n \to \Delta^n$ says precisely that $t \in A$ is terminal in Joyal's sense.
\end{proof}

One reason for our particular interest in terminal objects is to show that the units and counits of adjunctions have universal properties which may be expressed ``pointwise'' in terms of certain outer horn filler conditions.

\begin{prop}[the pointwise universal property of an adjunction]\label{prop:pointwise-univ-adj}
    Suppose that we are given an adjunction 
    \begin{equation*}
        \adjdisplay f -| u : A -> B.
    \end{equation*}
    of quasi-categories with unit $\eta\colon\id_B\Rightarrow uf$ and counit $\epsilon\colon fu\Rightarrow \id_A$. Then for each $a \in A$ the (fat) slice quasi-category $f\comma a\simeq\slicer{f}{a}$ has terminal object $\epsilon a\colon fua\to a$, namely the component of the counit $\epsilon$ at $a$.
\end{prop}

\begin{proof}
    From Proposition~\ref{prop:adjointequiv}, the adjunction $f\dashv u$ gives rise to the equivalence $f \comma A \simeq B \comma u$ fibred over $A \times B$. By Observation~\ref{obs:fibred-pullback}, for each  $a\in A$, the fibred equivalence pulls back along the functor $(a,\id_B)\colon B\to A\times B$ to give a fibred equivalence
    \begin{equation}
      \xymatrix{ f\comma a \ar@{->>}[dr]_{p_0} \ar@/^1ex/[rr]^{w'} & & B \comma ua \ar@{->>}[dl]^{q_0} \ar@/^1ex/[ll]^w \\ & B}
    \end{equation}
  over $B$. 

By Example~\ref{ex:slice-terminal}, we know that $B\comma ua$ has the identity map $ua\cdot\degen^0\colon ua\to ua$ as its terminal object, and by Proposition~\ref{prop:terminaldefn} we know that terminal objects transport along equivalences, so it follows that $f\comma a$ also has terminal object $w'(ua\cdot\degen^0)$. It is now easily checked, from the definition of $w'$ given in Proposition~\ref{prop:adjointequiv}, that $w'(ua\cdot\degen^0)$ is isomorphic to $\epsilon a\colon fua\to a$. The desired result follows on transporting this terminal object along the equivalence between $f\comma a$ and $\slicer{f}{a}$ provided by the geometry result of Lemma \ref{lem:slice-equiv-comma}.
\end{proof}

Of course, the unit of an adjunction of quasi-categories satisfies a dual universal property.

\begin{obs}[unpacking this pointwise universal property of an adjunction] \label{obs:universal-property-of-epsilon}
    Unpacking the definitions in Remark~\ref{rmk:map-slices} and Definition~\ref{defn:slices} we see that a map $X\to \slicer{f}{a}$ corresponds to a pair of maps $b\colon X\to B$ and $\alpha\colon X\join\Del^0\to A$ which make the diagram 
    \[
        \xymatrix@=1.5em{ 
            X \ar[d] \ar[r]^f & B \ar[d]^b \\ X \join \Del^0 \ar[r]^-{\alpha} & A \\ \Del^0 \ar[u] \ar[ur]_{a}} \] commute.

By Proposition \ref{prop:terminalconverse}, we know that $\epsilon a\colon fua\to a$ is terminal in $\slicer{f}{a}$ is terminal if and only if every sphere $\boundary\Del^{n-1}\to \slicer{f}{a}$ whose last vertex is $\epsilon a$ may be filled to a simplex. Applying our description of maps into $\slicer{f}{a}$ and observing that $\Del^{n-1}\join\Del^0\cong\Del^n$ and $\boundary\Del^{n-1}\join\Del^0\cong\Horn^{n,n}$, we see that $\epsilon a$ being terminal means that if we are given
    \begin{itemize} 
        \item a horn $\Horn^{n,n} \to A$, with $n \geq 2$ together with
        \item a sphere $\partial\Delta^{n-1} \to B$ whose composite with $f$ is the boundary of the missing face of the horn, with the property that
        \item  the final edge of the horn  is $\epsilon a$ 
    \end{itemize} 
    then there is 
    \begin{itemize}
        \item a simplex $\Delta^n \to A$ filling the given horn and
        \item a simplex $\Delta^{n-1} \to B$ filling the given sphere, with the property that
        \item the $n\th$ face of the filling $n$-simplex in $B$ is the simplex obtained by applying $f$ to the filling $(n-1)$-simplex in $A$.
    \end{itemize} 

    For $n=2$, this situation is summarised by the following schematic:
 \[ \vcenter{ \xymatrix@=1.2em{ & fua \ar[dr]^{\epsilon} & \\ fb \ar[rr]_\alpha & & a}} b \in B_0\quad \rightsquigarrow \vcenter{\xymatrix@=1.2em{ & fua \ar[dr]^{\epsilon} \ar@{}[d]|(.6){\sigma} &\\ fb \ar[ur]^{f\beta} \ar[rr]_\alpha & & a}}  \mkern10mu \sigma \in A_2,\ \beta \colon b \to ua \in B_1\]
  \end{obs}

\begin{obs}[the relative universal property of an adjunction]
For any quasi-category $X$ the 2-functor $\hom'(X,{-})\colon \qCat_2\to\Cat$ carries an adjunction $f\dashv u\colon A\to B$ of quasi-categories to an adjunction $\hom'(X,f)\dashv \hom'(X,u)\colon\hom'(X,A)\to\hom'(X,B)$ of categories. Extending Lemma~\ref{lem:adj-ext-univ}, a standard and easily established fact of 2-category theory is that $f\colon B\to A$ has a right adjoint in $\qCat_2$ if and only if for each quasi-category $X$ the functor $\hom'(X,f)\colon\hom'(X,B)\to\hom'(X,A)$  has a right adjoint. We might call this observation the {\em external\/} universal property of an adjunction.

    There is a closely related {\em internal\/} or {\em relative\/} universal property of adjunctions in $\qCat_2$, which arises instead from Remark~\ref{rmk:exp2functor} that the cotensor $(-)^X\colon\qCat_2\to\qCat_2$ is also a 2-functor. Applying this cotensor 2-functor to the adjunction $f\dashv u$ we obtain its relative universal property simply as the pointwise universal property of the adjunction $f^X\dashv u^X\colon A^X\to B^X$ as derived in Proposition~\ref{prop:pointwise-univ-adj} and expressed explicitly in Observation~\ref{obs:universal-property-of-epsilon}. The relative universal property of adjunctions will become a key tool in the proof that any adjoint functor between quasi-categories extends to a homotopy coherent adjunction; see  \cite{RiehlVerity:2012hc}.  
\end{obs}

Another application of Proposition \ref{prop:adjointequiv} allows us to show that an isofibration between quasi-categories admits a right adjoint right inverse if and only if the following lifting property holds.

\begin{lem}[right adjoint right inverse as a lifting property]\label{lem:RARI-lifting}
  An isofibration $f \colon B \tfib A$ of quasi-categories admits a right adjoint right inverse if and only if for all $a \in A_0$ there exists $ua \in B_0$ with $fua = a$ and so that any lifting problem with $n \geq 1$
\begin{equation}\label{eq:RARI-lifting} \xymatrix{ \Delta^0 \ar[r]_{\fbv{n}} \ar@/^2ex/[rr]^{ua} & \boundary\Delta^n \ar@{u(->}[d] \ar[r] & B \ar@{->>}[d]^f \\ & \Delta^n \ar@{-->}[ur] \ar[r] & A} \end{equation}
has a solution.
\end{lem}
\begin{proof}
If $u$ is the right adjoint right inverse, then $fu = \id_A$ and there is a trivial fibration $B \comma u \trvfib f \comma fu \cong f \comma A$ over $A \times B$ defined by applying $f$ (Lemma \ref{lem:comma-obj-maps} proves that this map is an isofibration and Proposition \ref{prop:adjointequiv} shows that it is an equivalence). This trivial fibration pulls back over any vertex $a \in A_0$ to define a trivial fibration $B \comma ua \trvfib f \comma a$. The domain and codomain are equivalent to Joyal's slices by Lemma \ref{lem:slice-equiv-comma}, so the isofibration $\slicer{B}{ua} \tfib \slicer{f}{a}$ is also a trivial fibration:
\[ \xymatrix{ \boundary\Delta^{n-1} \ar[r] \ar[d] & \slicer{B}{ua} \ar[d] \\ \Delta^n \ar[r] \ar@{-->}[ur] & \slicer{f}{a} \cong B \times_A \slicer{A}{a}}\] In adjoint form, this is the lifting property of \eqref{eq:RARI-lifting}.

Conversely, the lifting property \eqref{eq:RARI-lifting} can be used to inductively define a section $u \colon A \to B$ of $f$ extending the choices $ua \in B_0$ for $a \in A_0$. The inclusion $\sk_0A\hookrightarrow A$ can be expressed as a countable composite of pushouts of coproducts of maps $\boundary\Del^n\hookrightarrow\Del^n$ with $n \geq 1$, and each intermediate lifting problem required to define a lift
\[ \xymatrix{ \Delta^0 \ar[r]_-{a} \ar@/^2ex/[rr]^{ua} & \sk_0 A \ar@{u(->}[d] \ar[r] & B \ar@{->>}[d]^f \\ & A \ar@{-->}[ur]^u \ar@{=}[r] & A}\]
will have the form of \eqref{eq:RARI-lifting}. To show that $u$ is a right adjoint right inverse to $f$, it suffices, by Lemma \ref{lem:adjunction.from.isos} to define a 2-cell $\eta \colon \id_B \To uf$ that whiskers with $u$ and with $f$ to isomorphisms. We construct a representative for $\eta$ by solving the lifting problem
\[\xymatrix{ B \coprod B \ar[d] \ar[rr]^{\id_B \coprod uf} & & A \ar[d]^{f} \\ B \times \Delta^1 \ar@{-->}[urr]^\eta \ar[r]_-{\pi_B} & B \ar[r]_f & A}\]
By construction $f\eta=\id_f$. 

To show that $\eta u$ is an isomorphism it suffices, by Corollary \ref{cor:pointwise-equiv}, to check that its components $\eta u(a) \colon ua \to ufua=ua$ are isomorphisms in $A$. Inverse isomorphisms can be found by elementary applications of the lifting property \eqref{eq:RARI-lifting}, whose details we leave to the reader.
\end{proof}

\subsection{Fibred adjunctions}\label{subsec:fibred.adjunction}

Fibred equivalences over $A$, i.e., equivalences in $\ho_*(\qCat_\infty\slice  A)$, are preferable to equivalences in the slice 2-category $\qCat_2\slice  A$ because the former can be pulled back along arbitrary maps $f \colon B \to A$; see Observation~\ref{obs:fibred-pullback}. Precisely the same kind of reasoning applies to adjunctions in $\qCat_2\slice  A$. 

\begin{defn}[fibred adjunctions]\label{defn:fibred.adj}
  We refer to adjunctions in $\ho_*(\qCat_\infty\slice A)$ as \emph{adjunctions fibred over $A$} or simply \emph{fibred adjunctions}.
\end{defn}

    Our aim in this section is to show that any adjunction in $\qCat_2\slice A$ can  be lifted to an adjunction fibred over $A$, i.e., to an adjunction in $\ho_*(\qCat_\infty\slice A)$. In particular, such a result will allow us to prove that any adjunction in $\qCat_2\slice  A$ may be pulled back along any functor $f\colon B\to A$. We shall use this result to define a loops--suspension adjunction on any quasi-category with appropriate finite limits and colimits (cf.\ Proposition~\ref{prop:loops-suspension}).

    Recall from Proposition \ref{prop:slice-smothering-2-functor} that the canonical 2-functor $\ho_*(\qCat_\infty\slice A) \to \qCat_2\slice A$ is a smothering 2-functor. Consequently, the following 2-categorical lemma is key:

\begin{lem}\label{lem:missed-lemma} Suppose $F \colon \tcat{C} \to \tcat{D}$ is a smothering 2-functor. Then any adjunction in $\tcat{D}$ can be lifted to an adjunction in $\tcat{C}$. Furthermore, if we have previously specified a lift of the objects, 1-cells, and either the unit or counit of the adjunction in $\tcat{D}$, then there is a lift of the remaining 2-cell that combines with the previously specified data to define an adjunction in $\tcat{C}$.
\end{lem}
\begin{proof}
We use surjectivity on objects and local surjectivity on arrows to define $u \colon A \to B$ and $f\colon B\to A$ in $\tcat{C}$ lifting the objects and 1-cells of the downstairs adjunction. Then we use local fullness to define lifts $\epsilon \colon fu \Rightarrow \id_A$ and $\eta' \colon \id_B \Rightarrow uf$ of the downstairs counit and unit. If desired, we can regard $A$, $B$, $f$, $u$ and $\epsilon$ as ``previously specified''. We will show that $f \dashv u$ by modifying the 2-cell $\eta'$. The details are similar to the proof of Lemma \ref{lem:adjunction.from.isos}.

We define a 2-cell $\theta\colon u\Rightarrow u$ as the ``triangle identity composite'' $\theta\defeq u\epsilon \cdot \eta' u$ and observe that $F\theta = \id_{Fu}$. Applying the local conservativity of the action of $F$ on 2-cells, we conclude that $\theta$ is an isomorphism. Define the 2-cell $\eta \colon \id_B \Rightarrow uf$ to be the composite $\eta \defeq \theta^{-1} f \cdot \eta'$. Because $F\theta$ is an identity,  $F\eta$ and $F\eta'$ lift the same downstairs 2-cell. We claim that this data forms an adjunction in $\tcat{C}$.

The diagram \eqref{eq:triangle-calculation-1} demonstrates that $u\epsilon \cdot \eta u = \id_u$. The diagram \eqref{eq:triangle-calculation-2} demonstrates that the other triangle identity composite $\phi\defeq \epsilon f \cdot f \eta $ is an idempotent. Finally observe that the component parts we've composed to make $\phi$ all map by $F$ to the corresponding components of the original adjunction in $\lcat{L}$. It follows that $F\phi$ is equal to the corresponding triangle identity composite in $\lcat{L}$ and so is an identity. Consequently, applying the local conservativity of $F$ on 2-cells we find that $\phi$ is an isomorphism. Because all idempotent isomorphisms are identities,  it follows that $\epsilon f \cdot f \eta = \id_f$ as required.
\end{proof} 

\begin{cor}\label{cor:missed-lemma}
  Every adjunction in $\qCat_2\slice A$ lifts to an adjunction fibred over $A$.
\end{cor}

\begin{proof}
  Combine Proposition~\ref{prop:slice-smothering-2-functor} and Lemma~\ref{lem:missed-lemma}.
\end{proof}

\begin{ex}\label{ex:fibred-technical-slice-adjunction}
Corollary \ref{cor:missed-lemma} allows us to lift the adjunction $\adjinline p_1 -| i : C -> B\comma \ell.$ of Lemma \ref{lem:technicalsliceadjunction} to a fibred adjunction over $C$ whose counit is an identity.
\end{ex}

\begin{ex}[fibred isofibration RARIs]\label{ex:isofib-section.fibred.adjunction} 
Lemma \ref{lem:isofibration-RARI} demonstrates that any right adjoint right inverse to an isofibration $f \colon B \tfib A$ can be modified to produce a RARI $f \dashv u$ with an identity counit. This latter adjunction provides us with an adjunction in $\qCat_2\slice A$ which we may lift into $\ho_*(\qCat_\infty\slice A)$ to give an adjunction
\begin{equation}\label{eq:fibred.terminal}
  \xymatrix@=1.5em{
    {A}\ar@/_1.2ex/[rr]_u\ar@{=}[dr] & {\bot} & 
    {B}\ar@/_1.2ex/[ll]_f\ar@{->>}[dl]^{f} \\
    & A &
  }
\end{equation}
which is fibred over $A$. In essence, this latter fibred adjunction expresses the fact that each of the fibres of the isofibration $f\colon B\tfib A$ has a terminal object.
\end{ex}

\begin{obs}\label{obs:isofib-section.fibred.adjunction}
   Applying the 2-functor $\hom'_A(p,{-})$ represented by an isofibration $p\colon E\tfib A$ to the fibred adjunction in~\eqref{eq:fibred.terminal} we obtain an adjunction
  \begin{equation*}
    \adjdisplay f\circ{-} -| u\circ{-} : 
    \hom'_A(p,\id_A) -> \hom'_A(p, f).
  \end{equation*}
  of hom-categories. Now the identity functor $\id_A$ is the 2-terminal object of the 2-category $\qCat_2\slice A$, so it follows that $\hom'_A(p,\id_A)\cong\catone$. Hence, the displayed adjunction amounts simply to the assertion that $up$ is a terminal object of the category $\hom'_A(p, f)$. Consequently, applying Lemma~\ref{lem:adj-ext-univ}, we discover that there exists a fibred adjunction of the form displayed in~\eqref{eq:fibred.terminal} if and only if for all isofibrations $p\colon E\tfib A$ the composite map $up\colon E\to B$ is a terminal object of the hom-category $\hom'_A(p,f)$.
\end{obs}

A final example of a fibred adjunction describes the ``composition'' functor $A^{\Horn^{2,1}} \to A^\cattwo$ that fills a (2,1)-horn and then restricts to the missing face as the right and left adjoint, respectively, to the pair of functors that extend a 1-simplex into a composable pair by using the identities at its domain and codomain.

\begin{ex}\label{ex:comp.ident.adj}
  There exists a pair of adjunctions
  \begin{equation*}
    \xymatrix@C=10em@R=1ex{
      {{\Del^1}}\ar[r]|*+{\scriptstyle \face^1} & {{\Del^2}}
      \ar@/^2.5ex/[l]^{\degen^0}_{}="l" \ar@/_2.5ex/[l]_{\degen_1}^{}="u"
      \ar@{} "u";"l" |(0.2){\bot} |(0.8){\bot}
    }
  \end{equation*}
  of ordered sets, whose units and counits arise as the equalities $\degen^0\face^1 = \degen^1\face^1=\id_{\Del^1}$ and the inequalities $\face^1\degen^0 < \id_{[2]} < \face^1\degen^1$. Now if $A$ is a quasi-category, we may apply Proposition~\ref{prop:expadj} to construct the associated pair of adjunctions
  \begin{equation*}
    \xymatrix@C=10em{
      {A^{\Del^2}}
      \ar[r]|*+{\scriptstyle A^{\face^1}} &
      {A^{\Del^1}}
      \ar@/^2.5ex/[l]^{A^{\degen^1}}_{}="l" \ar@/_2.5ex/[l]_{A^{\degen_0}}^{}="u"
      \ar@{} "u";"l" |(0.2){\bot} |(0.8){\bot} 
    }
  \end{equation*}
  Here the upper adjunction has identity unit and the lower adjunction has identity counit. So it follows from Example~\ref{ex:isofib-section.fibred.adjunction} that this is a pair of adjunctions fibred over $A^{\Del^1}$ with respect to the projections $A^{\face^1}\colon A^{\Del^2}\tfib A^{\Del^1}$ and $\id_{A^{\Del^1}}\colon A^{\Del^1} \tfib A^{\Del^1}$. 

Because the horn inclusion $\Horn^{2,1}\inc\Del^2$ is a trivial cofibration in Joyal's model structure, the associated restriction isofibration $p\colon A^{\Del^2}\tfib A^{\Horn^{2,1}}$ is an equivalence of quasi-categories fibred over $A^{\Horn^{2,1}}$. By Proposition~\ref{prop:equivtoadjoint} (applied to $\qCat_2/A^{\Horn^{2,1}}$) and Corollary~\ref{cor:missed-lemma}, the fibred equivalence formed by $p$ and a chosen inverse $p'$ can be promoted to a pair of adjoint equivalences $p \dashv p' \dashv p$ fibred over $A^{\Horn^{2,1}}$. On account of the pushout diagram defining the (2,1)-horn,  $A^{\Horn^{2,1}}$ is isomorphic to the pullback:
\begin{equation*}
  \xymatrix@=2em{ \Horn^{2,1} \pbexcursion & \Del^1 \ar[l]_-{\face^2} & & 
    {A^{\Horn^{2,1}}}\pbexcursion
    \ar[r]^-{\pi_0}\ar[d]_-{\pi_1} & {A^\cattwo}\ar@{->>}[d]^-{p_1} \\ \Del^1 \ar[u]^{\face^0} & \Del^0 \ar[u]_{\face^0}\ar[l]^-{\face^1} & &
    {A^\cattwo}\ar@{->>}[r]_-{p_0} & A
  }
\end{equation*}

Now we may take the pushforward of the fibred adjunctions of the last two paragraphs along the isofibrations $(A^{\fbv{1}},A^{\fbv{0}})\colon A^{\Del^1}\tfib A\times A$ and $(A^{\fbv{2}},A^{\fbv{0}})\colon A^{\Horn^{2,1}}\tfib A\times A$ respectively to obtain adjunctions fibred over $A\times A$. Composing these we obtain a pair of adjunctions 
  \begin{equation}\label{eq:comp.ident.adj}
    \xymatrix@C=10em{
      *+[l]{A^{\Horn^{2,1}}\cong A^\cattwo\times_AA^\cattwo}
      \ar[r]|*+{\scriptstyle m} &
      {A^\cattwo}
      \ar@/^2.5ex/[l]^{i_1}_{}="l" \ar@/_2.5ex/[l]_{i_0}^{}="u"
      \ar@{} "u";"l" |(0.2){\bot} |(0.8){\bot} 
    }
  \end{equation}
  which are fibred over $A\times A$ with respect to the projections $(p_1,p_0)\colon A^\cattwo\tfib A\times A$ and $(p_1\pi_1,p_0\pi_0)\colon A^{\Horn^{2,1}}\tfib A\times A$. Here the upper adjunction has isomorphic unit and the lower adjunction has isomorphic counit. The functors $i_0$ and $i_1$ degenerate the domain and codomain respectively of a given 1-simplex to form a (2,1)-horn. The map $m$ is a ``composition'' functor.
\end{ex}

%!TEX root = all.tex
% ******************************************************************
% ** Title:           The 2-category theory of quasi-categories
% **                   limits and colimits
% ** Precis:        
% ** Author:           Emily Riehl and Dominic Verity
% ** Commenced:        2/3/2012
% ******************************************************************

\section{Limits and colimits}\label{sec:limits}

In this section, we demonstrate that limits and colimits of individual diagrams in a quasi-category can be encoded as {\em absolute right and left liftings\/} in the 2-category $\qCat_2$. The proof that this definition is equivalent to the standard one makes use of the fact that absolute lifting diagrams in $\qCat_2$ can be detected by an equivalence of suitably defined comma quasi-categories. This observation, combined with Example~\ref{ex:adjasabslifting}, also supplies the proof of Proposition~\ref{prop:adjointequivconverse}, completing the unfinished business from the previous section. 

We begin with a general definition:

\setcounter{thm}{0}
\begin{defn}\label{defn:abs-right-lift} In a 2-category, an \emph{absolute right lifting diagram}   consists of the data \begin{equation}\label{eq:absRlifting}\xymatrix{ \ar@{}[dr]|(.7){\Downarrow\lambda} & B \ar[d]^f \\ C \ar[r]_g \ar[ur]^\ell & A}\end{equation} with the universal property that if we are given any 2-cell $\chi$ of the form depicted to the left of the following equality
\begin{equation}\label{eq:abs-lifting-property}
    \vcenter{\xymatrix{ X \ar[d]_c \ar[r]^b \ar@{}[dr]|{\Downarrow\chi} & B \ar[d]^f \\ C \ar[r]_g & A}} \mkern20mu = \mkern20mu \vcenter{\xymatrix{ X \ar[d]_c \ar[r]^b \ar@{}[dr]|(.3){\exists !\Downarrow}|(.7){\Downarrow\lambda} & B \ar[d]^f \\ C \ar[ur]|(.4)*+<2pt>{\scriptstyle\ell} \ar[r]_g & A}}
 \end{equation} 
 then it admits a unique factorisation of the form displayed to the right of that equality. When this condition holds for the diagram in~\eqref{eq:absRlifting} we say that it {\em displays $\ell$ as an absolute right lifting of $g$ through $f$}.
\end{defn}

\begin{ex}\label{ex:adjasabslifting} The counit of an adjunction $\adjinline f-|u:A->B.$ defines an absolute right lifting diagram 
  \begin{equation}\label{eq:adjasabslifting}
    \xymatrix{ \ar@{}[dr]|(.7){\Downarrow\epsilon} & B \ar[d]^f \\ A \ar[ur]^u \ar[r]_{\id_A} & A}  
  \end{equation}
  and, conversely, if this diagram displays $u$ as an absolute right lifting of the identity on its domain through $f$ then $f$ is left adjoint to $u$ with counit 2-cell $\epsilon$. 
\end{ex}

\begin{proof} 
This is a standard 2-categorical result. The 2-functor represented by $X$ carries an adjunction $f \dashv u$ to an adjunction whose counit has the universal property described in~\eqref{eq:abs-lifting-property} for the 2-cell \eqref{eq:adjasabslifting}.

Conversely, given an  absolute right lifting diagram~\eqref{eq:adjasabslifting}, we take this 2-cell to be the counit and define the unit by applying the universal property of this absolute right lifting to the identity 2-cell:
\begin{equation}\label{eq:unitdefn} 
  \vcenter{\xymatrix{ B \ar[r]^{\id_B} \ar[d]_f \ar@{}[dr]|{\Downarrow \id_f} & B \ar[d]^f \\ A \ar[r]_{\id_{A}} & A}} = \vcenter{\xymatrix{ B \ar[r]^{\id_B} \ar[d]_f \ar@{}[dr]|(.3){\Downarrow \eta}|(.7){\Downarrow\epsilon} & B \ar[d]^f \\ A \ar[r]_{\id_{A}} \ar[ur]|*+{\scriptstyle u} & A}} 
  \end{equation}
  This defining equation establishes one of the triangle identities. The other is obtained by pasting $\epsilon$ on the left of both of the 2-cells of \eqref{eq:unitdefn} and applying the uniqueness statement in the universal property of the absolute right lifting:
 \[  \vcenter{\xymatrix{ \ar@{}[dr]|(.7){\Downarrow\epsilon} & B \ar[r]^{\id_B} \ar[d]_f \ar@{}[dr]|{\Downarrow \id_f} & B \ar[d]^f \\ A \ar[r]_{\id_{A}} \ar[ur]^{u} & A \ar[r]_{\id_{A}} & A}} = \vcenter{\xymatrix{  \ar@{}[dr]|(.7){\Downarrow\epsilon}& B \ar[r]^{\id_B} \ar[d]_f \ar@{}[dr]|(.3){\Downarrow \eta}|(.7){\Downarrow\epsilon} & B \ar[d]^f \\ A \ar[r]_{\id_{A}} \ar[ur]^{u} &  A \ar[r]_{\id_{A}} \ar[ur]|*+{\scriptstyle u} & A}}\rightsquigarrow \vcenter{\xymatrix{ & B  \ar@{}[dl]|{\Downarrow\id_{u}}  \\ A \ar@/^2.25ex/[ur]^{u} \ar@/_2.25ex/[ur]_{u} &  }} = \vcenter{\xymatrix{  \ar@{}[dr]|(.7){\Downarrow\epsilon}& B \ar[r]^{\id_B} \ar[d]_f \ar@{}[dr]|(.3){\Downarrow \eta}& B\\ A \ar[r]_{\id_{A}} \ar[ur]^{u} &  A  \ar[ur]|*+{\scriptstyle u} & }}\qedhere \]
\end{proof}

\subsection{Absolute liftings and comma objects}

We now specialise to the 2-category $\qCat_2$. Our aim is to use its weak comma objects to re-express the universal property of absolute lifting diagrams and describe various procedures through which they may be detected.

Given any diagram in $\qCat_2$ of the form displayed in~\eqref{eq:absRlifting} in $\qCat_2$ we may form comma objects $B \comma \ell$ and $f \comma g$ with canonical comma cones:
\begin{equation}\label{eq:comma-cones}
    \vcenter{ \xymatrix{ B \comma \ell \ar[d]_{p_1} \ar[r]^-{p_0} \ar@{}[dr]|(.3){\Downarrow\phi} & B  & & f \comma g \ar[d]_{q_1} \ar[r]^-{q_0}  \ar@{}[dr]|{\Downarrow\psi} & B \ar[d]^f \\ C \ar[ur]_\ell & & &C \ar[r]_g & A }}
\end{equation}
Pasting the canonical cone associated with $B \comma \ell$ onto the triangle~\eqref{eq:absRlifting} we obtain a comma cone which induces a functor $w\colon B \comma \ell \to f \comma g$ by the 1-cell induction property of $f\comma g$. Recall this means that $w$ makes the following pasting equality hold
\begin{equation}\label{eq:w-def-prop}
  \vcenter{\xymatrix@=1.2em{
    & {B\comma\ell}\ar[dl]_{p_1}\ar[dr]^{p_0} & \\
    {C} \ar[dr]_{g}\ar[rr]|*+{\scriptstyle\ell} && {B}\ar[dl]^{f} \\
    & A &  
    \ar@{} "1,2";"3,2" |(0.3){\Leftarrow\phi} |(0.7){\Leftarrow\lambda}
  }}
  \mkern 20mu = \mkern20mu
  \vcenter{\xymatrix@=1.2em{
    & {B\comma\ell}\ar[d]^{w}\ar@/_1.5ex/[ddl]_{p_1}\ar@/^1.5ex/[ddr]^{p_0} & \\
    & {f\comma g}\ar[dl]^{q_1}\ar[dr]_{q_0} & \\
    {C}\ar[dr]_{g} & & {B}\ar[dl]^{f} \\
    & {A} & 
    \ar@{} "2,2";"4,2" |{\Leftarrow\psi}
  }}
\end{equation}
and in particular may be regarded as being a 1-cell in the slice 2-category $\qCat_2\slice(C\times B)$ from $(p_1,p_0)\colon f\comma g \tfib C\times B$ to $(q_1,q_0)\colon B\comma\ell\tfib C\times B$.

\begin{prop}\label{prop:absliftingtranslation} The data of \eqref{eq:absRlifting} defines an absolute right lifting in $\qCat_2$ if and only if the induced map $w\colon B \comma \ell \to f \comma g$ of~\eqref{eq:w-def-prop} is an equivalence.
\end{prop} 
\begin{proof}
    For each pair of functors $b\colon X\to B$ and $c\colon X\to C$ as in~\eqref{eq:abs-lifting-property} observe that $\sq_{g,f}(c,b)$ (cf.\ Observation~\ref{obs:squares-set}) is simply the set of those 2-cells of the form depicted in the square on the left of the equality in~\eqref{eq:abs-lifting-property} and that $\sq_{\ell,B}(c,b)$ is the set of those 2-cells which inhabit the upper left triangle of the diagram to the right of that same equality. Define
  \begin{equation*}
    \xymatrix@C=8em{
      {\sq_{\ell,B}(c,b)}\ar[r]^{k^{\lambda}_{(c,b)}} &
      {\sq_{g,f}(c,b)}
    }
  \end{equation*}
  to be the function which takes each triangle in its domain and pastes it onto our candidate lifting diagram~\eqref{eq:absRlifting} to obtain a corresponding square as depicted in~\eqref{eq:abs-lifting-property}. This family of functions is natural in $(c,b)\colon X\to C\times B$ in the sense that they are the components of a natural transformation $k^\lambda$ between the functors
\[ \xymatrix{ (\pi^g_0)_*(\qCat_2\slice(C\times B))\op \ar@<1.5ex>[r]^-{\sq_{\ell,B}} \ar@<-1.5ex>[r]_-{\sq_{g,f}} \ar@{}[r]|-{\Downarrow k^\gamma} & \Set}\] 
of Lemma~\ref{lem:sq-as-a-functor}. By construction,  the triangle in~\eqref{eq:absRlifting} is an absolute right lifting if and only if $k^\lambda\colon \sq_{\ell,B}\Rightarrow \sq_{g,f}$ is a natural isomorphism.

Now   consider a commutative square of natural transformations
  \begin{equation*}
    \xymatrix@C=5em{
      {\pi^g_0(\hom'_{C\times B}(-,(p_1,p_0)))} 
      \ar[r]^{u\circ -}\ar[d]_{\cong} &
      {\pi^g_0(\hom'_{C\times B}(-,(q_1,q_0)))}
      \ar[d]^{\cong} \\
      {\sq_{\ell,B}}\ar[r]_{k} &
      {\sq_{g,f}}
    }
  \end{equation*}
  between presheaves on $(\pi^g_0)_*(\qCat_2\slice(C\times B))$, in which the vertical isomorphisms are those induced by the weakly universal comma cones of~\eqref{eq:comma-cones} as discussed in Lemma~\ref{lem:cpts-and-comma-2-cells}. Applying Yoneda's lemma and the definition of $(\pi^g_0)_*(\qCat_2\slice(C\times B))$, this square provides us with a canonical bijection between the set of natural transformations $k\colon\sq_{\ell,B}\Rightarrow\sq_{g,f}$ and the set of isomorphism classes of 1-cells
  \begin{equation}\label{eq:induced-u-from-nattrans-k}
    \xymatrix@=1em{
      {B\comma\ell}\ar@{->>}[dr]_(0.3){(p_1,p_0)}\ar[rr]^{u}
      && *+!L(0.5){f\comma g}\ar@{->>}[dl]^(0.3){(q_1,q_0)} \\
      & {C\times B}&
    }
  \end{equation}
  in $\qCat_2\slice(C\times B)$.  By the Yoneda lemma, $k\colon\sq_{\ell,B}\Rightarrow\sq_{g,f}$ is a natural isomorphism if and only if the corresponding $u\colon B\comma\ell\to f\comma g$ is an isomorphism in $(\pi^g_0)_*(\qCat_2\slice(C\times B))$. By Observation~\ref{obs:groupoid-components}, this holds if and only if $u$ is an equivalence in $\qCat_2\slice(C\times B)$. By Lemma~\ref{lem:proj-is-1-conservative},  this is the case if and only if $u$ is an equivalence in $\qCat_2$. 

In particular, the natural transformation $k^\lambda\colon\sq_{\ell,B}\Rightarrow\sq_{g,f}$ constructed from the 2-cell~\eqref{eq:absRlifting} corresponds  to the isomorphism class of those induced 1-cells $w\colon B\comma\ell\to f\comma g$ over $C\times B$ which satisfy the pasting identity displayed in~\eqref{eq:w-def-prop}. We have just shown that  the triangle in~\eqref{eq:absRlifting} is an absolute lifting diagram if and only if $k^\lambda\colon\sq_{\ell,B}\Rightarrow\sq_{g,f}$ is a natural isomorphism, which is the case  if and only if $w\colon B \comma \ell \to f \comma g$ is an equivalence. 
\end{proof}

\begin{rmk}
There is nothing in the proof of the Proposition \ref{prop:absliftingtranslation}, or in those of the results upon which it relies, which depends upon the vertex $X$ in~\eqref{eq:abs-lifting-property} being a quasi-category. The essential point here is that the space of maps out of any simplicial set $X$ taking values in a quasi-category is still a quasi-category. Consequently, we find that absolute lifting diagrams in $\qCat_2$ possess the factorisation property displayed in~\eqref{eq:abs-lifting-property} for 2-cells whose 0-cellular domains $X$ are general simplicial sets.  
\end{rmk}

For certain applications, it will be important to have a strengthened version of Proposition~\ref{prop:absliftingtranslation} which says that from {\em any\/} equivalence $B\comma \ell \simeq f \comma g$ fibred over $C\times B$ we may construct a 2-cell which displays $\ell$ as an absolute right lifting of $g$ through $f$. This result, Proposition~\ref{prop:absliftingtranslation2} below, proceeds directly from the following technical lemma:

\begin{lem}\label{lem:represented-nat-trans}
  For all natural transformations $k\colon\sq_{\ell,B}\Rightarrow\sq_{g,f}$ there exists a unique 2-cell $\lambda$ of the form depicted in~\eqref{eq:absRlifting} such that $k$ is equal to the natural transformation $k^\lambda$ defined by pasting a 2-cell in a triangle over $\ell$ with $\lambda$ to form a 2-cell in a square over $f$ and $g$.
\end{lem}

\begin{proof}
  A 2-cell in the triangle~\eqref{eq:absRlifting} is simply an element of $\sq_{g,f}(C,\ell)$, so we may construct our candidate 2-cell $\lambda$ from the natural transformation $k\colon\sq_{\ell,B}\Rightarrow\sq_{g,f}$ by applying it to the identity 2-cell  in $\sq_{\ell,B}(C,\ell)$; that is, we take $\lambda\defeq k_{(C,\ell)}(\id_\ell)$. 

  Lemma~\ref{lem:cpts-and-comma-2-cells} reveals that $\sq_{\ell,B}$ is a representable functor whose universal element is the 2-cell $\phi\in\sq_{\ell,B}(p_1,p_0)$ of the weakly universal cone~\eqref{eq:comma-cones} displaying $B\comma\ell$. So Yoneda's lemma tells us that in order to show that our original natural transformation $k$ is equal to $k^\lambda$ it is enough to check that they both map $\phi$ to the same element of $\sq_{g,f}(p_1,p_0)$.

To do this, first observe that the functor $i \colon C \to B\comma\ell$ defined in Lemma~\ref{lem:technicalsliceadjunction} can be regarded as a morphism in $(\pi^g_0)_*(\qCat_2\slice(C\times B))$. Its defining property, that $\phi i = \id_\ell$, may then be re-expressed as the equality $\sq_{\ell,B}(i)(\phi) = \id_\ell$ relating $\id_\ell\in\sq_{\ell,B}(C,\ell)$ and $\phi\in\sq_{\ell,B}(p_1,p_0)$. By naturality of $k$, this then allows us to obtain a similar relationship between the 2-cell $\lambda$  and the image $\mu\defeq k_{(p_1,p_0)}(\phi)$ of $\phi$ under $k$, as given by the following computation: $\sq_{g,f}(i)(k_{(p_1,p_0)}(\phi)) = k_{(C,\ell)}(\sq_{\ell,B}(i)(\phi)) = k_{(C,\ell)}(\id_\ell) = \lambda$. By the definition of the map $\sq_{g,f}(i)$, this relationship may be expressed as a pasting equality:
  \begin{equation}\label{eq:rel-mu-lambda}
    \vcenter{\xymatrix@=0.8em{
      & {C} \ar@{=}[dl]\ar[dr]^\ell & \\
      {C}\ar[dr]_g & {\scriptstyle\Leftarrow\lambda} &
      {B}\ar[dl]^f \\
      & {A} &
    }}
    \mkern20mu = \mkern20mu
    \vcenter{\xymatrix@=0.8em{
      & \save []+<0pt,1em>*+{C}\ar[d]^i\ar@{=}@/_1.5ex/[ddl]\ar@/^1.5ex/[ddr]^\ell \restore & \\
      & {B\comma\ell}\ar[dl]^{p_1}\ar[dr]_{p_0} & \\
      {C}\ar[dr]_g & {\scriptstyle\Leftarrow\mu} & 
      {B}\ar[dl]^f \\
      & {A} &
    }}
  \end{equation}

By definition, $k^\lambda$ acts on $\phi$ by pasting it to the 2-cell $\lambda$ as depicted in the diagram on the left hand side of the following computation:
  \begin{equation*}
    \vcenter{\xymatrix@=1em{
      & & {B\comma\ell}\ar[dl]_{p_1}
      \ar[dd]^{p_0}_{}="one"\\
      & {C} \ar@{=}[dl]\ar[dr]_(0.4)\ell 
      \ar@{} "one" |(0.6){\Leftarrow\phi} & \\
      {C}\ar[dr]_g & {\scriptstyle\Leftarrow\lambda} &
      {B}\ar[dl]^f \\
      & {A} &
    }}
    \mkern5mu = \mkern5mu
    \vcenter{\xymatrix@=0.8em{
      & & {B\comma\ell}\ar[dl]_{p_1}\ar[dddd]^{p_0} \\
      & {C}\ar[dd]^i\ar@/_1.5ex/@{=}[dddl]
      \ar@/^1.5ex/[dddr]^(0.3)\ell="one" 
      & \\ & & \\
      & {B\comma\ell}\ar[dl]^{p_1}\ar[dr]_{p_0} & \\
      {C}\ar[dr]_g & {\scriptstyle\Leftarrow\mu} & 
      {B}\ar[dl]^f \\
      & {A} &
      \ar@{}"1,3";"one"|(0.7){\Leftarrow\phi} 
    }}
    \mkern5mu = \mkern5mu
    \vcenter{\xymatrix@=0.8em{
      & & {B\comma\ell}\ar[dl]_{p_1}\ar[dddd]^{p_0} 
      \ar@{=}@/^1.5ex/[dddl]_{}="one"\\
      & {C}\ar[dd]^i\ar@/_1.5ex/@{=}[dddl]
      & \\ & & \\
      & {B\comma\ell}\ar[dl]^{p_1}\ar[dr]_{p_0} & \\
      {C}\ar[dr]_g & {\scriptstyle\Leftarrow\mu} & 
      {B}\ar[dl]^f \\
      & {A} &
      \ar@{}"2,2";"one"|(0.6){\Leftarrow\nu} 
    }}
    \mkern5mu = \mkern5mu
    \vcenter{\xymatrix@=0.8em{
      & {B\comma\ell}\ar[dl]_{p_1}\ar[dr]^{p_0} & \\
      {C}\ar[dr]_g & {\scriptstyle\Leftarrow\mu} & 
      {B}\ar[dl]^f \\
      & {A} &
    }}
  \end{equation*}
  To elaborate, the first step in this calculation is simply an application of the equality given in~\eqref{eq:rel-mu-lambda}. Its second step follows from the first of the defining properties of the unit $\nu\colon\id_{B\comma\ell}\Rightarrow i p_1$ of the adjunction $p_1\dashv i$ of Lemma~\ref{lem:technicalsliceadjunction}, those being that $p_0\nu=\phi$ and $p_1\nu=\id_{p_1}$. The third of these equalities follows on observing that the pasting depicted on its left is simply the horizontal composite of the 2-cells $\mu$ and $\nu$, which may be expressed as the vertical composite $qp_1\nu\cdot \mu$ in which the second factor is an identity by the second defining property of $\nu$. 
  
  In other words, this calculation demonstrates that $k^\lambda_{(p_1,p_0)}(\phi)=\mu$ which is in turn equal to $k_{(p_1,p_0)}(\phi)$, by definition. Consequently, Yoneda's lemma tells us that $k = k^\lambda$ as required. Finally, the fact that $\lambda$ is the unique 2-cell with the property that $k=k^\lambda$ follows immediately from the patent fact that $\lambda = k^\lambda_{(C,\ell)}(\id_\ell)$.
\end{proof}

As an immediate corollary, we have the following important result:

\begin{prop}\label{prop:absliftingtranslation2} Suppose we are given functors $f\colon B\to A$, $g\colon C\to A$, and $\ell\colon C\to B$ of quasi-categories. Then the construction depicted in~\eqref{eq:w-def-prop} provides us with a bijection between 2-cells of the form 
\begin{equation}\label{eq:absRlifting.2}
  \xymatrix{ \ar@{}[dr]|(.7){\Downarrow\lambda} & B \ar[d]^f \\ C \ar[r]_g \ar[ur]^\ell & A}
\end{equation}
and isomorphism classes of 1-cells
\begin{equation}\label{eq:induced-from-lambda}
  \xymatrix@=1em{
    {B\comma\ell}\ar@{->>}[dr]_(0.3){(p_1,p_0)}\ar[rr]^{w}
    && *+!L(0.5){f\comma g}\ar@{->>}[dl]^(0.3){(q_1,q_0)} \\
    & {C\times B}&
  }
\end{equation}
in $\qCat_2\slice(C\times B)$. Furthermore, this 2-cell $\lambda$ displays $\ell$ as an absolute right lifting of $g$ through $f$ if and only if any representative $w$ of the corresponding isomorphism class of functors is an equivalence.
\end{prop}

\begin{proof}
Lemma \ref{lem:represented-nat-trans} provides a canonical bijection between 2-cells \eqref{eq:absRlifting.2} and natural transformations $\sq_{\ell,B} \To \sq_{g,f}$. 
The proof of Proposition \ref{prop:absliftingtranslation} establishes a canonical bijection between natural transformations $\sq_{\ell,B} \To \sq_{g,f}$ and isomorphism classes of 1-cells \eqref{eq:induced-from-lambda}. Proposition \ref{prop:absliftingtranslation} then concludes that $\lambda$ displays $\ell$ as an absolute right lifting of $g$ through $f$ if and only if any representative $w$ of the corresponding isomorphism class of functors is an equivalence.
\end{proof}

As a special case,  if $f\comma A$ and $B \comma u$ are equivalent over $A \times B$, then $f$ is left adjoint to $u$.

\begin{proof}[Proof of Proposition~\ref{prop:adjointequivconverse}]
If $f\comma A$ and $B \comma u$ are equivalent over $A \times B$, then Proposition \ref{prop:absliftingtranslation2} provides us with a corresponding 2-cell $\epsilon\colon fu\Rightarrow\id_A$, which displays $u$ as an absolute right lifting of $\id_A$ through $f$. By Example~\ref{ex:adjasabslifting}, this provides us with enough information to conclude that $f$ is left adjoint to $u$ with counit $\epsilon$. 
\end{proof}

  A second characterisation of absolute right liftings in $\qCat_2$ relates them to the possession of terminal objects by the fibres of $q_1\colon f\comma g\tfib C$. To explain this relationship, start by applying Observation~\ref{obs:1cell-ind-uniqueness-reloaded} to show that arbitrary pairs $(\ell,\lambda)$ as depicted in~\eqref{eq:absRlifting.2} correspond to isomorphism classes of functors
  \begin{equation*}
    \xymatrix@=0.8em{
      {C}\ar[dr]_-{(C,\ell)}\ar[rr]^{t}
      && *+!L(0.5){f\comma g}\ar@{->>}[dl]^-{(q_1,q_0)} \\
      & {C\times B} &
    }
  \end{equation*}
  over $C\times B$ defined by 1-cell induction
  \begin{equation}\label{eq:t-for-lim-def-prop}
    \vcenter{\xymatrix@=0.8em{
      & \save []+<0pt,1em>*+{C}\ar[d]^-{t}\ar@{=}@/_1.5ex/[ddl]
      \ar@/^1.5ex/[ddr]^{\ell}\restore & \\
      & {f\comma g}\ar[dr]_{q_0}\ar[dl]^{q_1} & \\
      {C}\ar[dr]_{g} & {\scriptstyle\Leftarrow\psi} &
      {B}\ar[dl]^{f} \\
      & {A} &
    }}
    \mkern20mu = \mkern20mu
    \vcenter{\xymatrix@=0.7em{
      & {C}\ar@{=}[dl]\ar[dr]^{\ell} & \\
      {C}\ar[dr]_{g} & {\scriptstyle\Leftarrow\lambda} & {B}\ar[dl]^{f} \\
      & {A} &
    }}
  \end{equation}
The following proposition relates the universal properties of pairs $(\ell,\lambda)$ and corresponding functors $t$.

\begin{prop}\label{prop:right.liftings.as.fibred.terminal.objects} The 2-cell $\lambda$ shown in~\eqref{eq:absRlifting.2} displays $\ell$ as an absolute right lifting of the functor $g$ through $f$ if and only if the induced functor $t\colon C\to f\comma g$  of \eqref{eq:t-for-lim-def-prop} features in a fibred adjunction:
  \begin{equation}\label{eq:fibred.terminal.2}
    \xymatrix@=1.2em{
      {C}\ar@{=}[dr]\ar@/_1.5ex/[rr]_-{t}^-{}="one"
      & & *+!L(0.5){f\comma g}\ar@{->>}[dl]^-{q_1}
      \ar@/_1.5ex/[ll]_-{q_1}^-{}="two" \\
      & {C} &
      \ar@{}"one";"two"|{\bot}
    }
  \end{equation}
that is, if and only if $t$ defines a right adjoint right inverse to the isofibration $q_1$.
\end{prop}
\begin{proof}
First assume that the triangle in~\eqref{eq:absRlifting.2} is an absolute right lifting diagram and apply Proposition~\ref{prop:absliftingtranslation} to show that the associated functor $w\colon B\comma\ell\to f\comma g$ is a fibred equivalence with equivalence inverse $w'$. Applying Proposition~\ref{prop:equivtoadjoint} in $\qCat_2/(C\times B)$ and Corollary~\ref{cor:missed-lemma}, this may be promoted to a fibred adjoint equivalence $w'\dashv w$ over $C\times B$. Its pushforward along the projection $C\times B\tfib C$ is an adjoint equivalence fibred over $C$. 

Example~\ref{ex:fibred-technical-slice-adjunction} provides us with an adjunction $p_1\dashv i\colon C\to B\comma\ell$  also fibred over $C$. Composing these, we obtain an adjunction $p_1w'\dashv wi\colon C\to f\comma g$  again fibred over $C$. From the defining properties of $w$ and $i$, as described in~\eqref{eq:w-def-prop} and~\eqref{eq:technicalsliceadjunction}, it is clear that $wi$ is a 1-cell induced over $\psi$ by the comma cone $\lambda$, and so we may infer, by Observation~\ref{obs:1cell-ind-uniqueness-reloaded}, that it is isomorphic to $t$ over $C$. Furthermore, $w'$ is fibred over $C\times B$ so $p_1w' = q_1$, and the fibred adjunction $p_1w'\dashv wi$ reduces to a fibred adjunction $q_1\dashv t$ as required.

   For the converse, assume that we have a fibred adjunction of the form given in~\eqref{eq:fibred.terminal.2}. We must show that for any 2-cell $\mu$ 
  \begin{equation}\label{eq:limit-lifting-prop}
    \vcenter{\xymatrix@=1.5em{
      {Y}\ar[r]^{b}\ar[d]_-{c} &
      {B}\ar[d]^-{f} \\
      {C}\ar@{}[ur]|{\Downarrow\mu}\ar[r]_-{g} &
      {A}
    }}
    \mkern20mu = \mkern20mu
    \vcenter{\xymatrix@=1.5em{
      {Y}\ar[r]^{b}\ar[d]_-{c} &
      {B}\ar[d]^-{f} \\
      {C}\ar[ur]|*+<2pt>{\scriptstyle\ell}="one"\ar[r]_-{g} &
      {A}
      \ar@{} "1,1";"one"|(0.6){\Downarrow \exists!\tau}
      \ar@{} "one";"2,2"|(0.4){\Downarrow\lambda}
    }}
  \end{equation}
there exits a unique 2-cell $\tau$ which makes this pasting equation hold. 

  To do this, start by applying the 1-cell induction property of $f\comma g$ to the comma cone $\mu$ to give a functor $m\colon Y\to f\comma g$ so that
  \begin{equation}\label{eq:m-def-prop}
    \vcenter{\xymatrix@=0.7em{
      & {Y}\ar[dd]^{m}\ar@/_1.5ex/[dddl]_{c}
      \ar@/^1.5ex/[dddr]^{b} & \\
      &&\\
      & {f\comma g}\ar[dr]_{q_0}\ar[dl]^{q_1} & \\
      {C}\ar[dr]_{g} & {\scriptstyle\Leftarrow\psi} &
      {B}\ar[dl]^{f} \\
      & {A} &
    }}
    \mkern20mu = \mkern20mu
    \vcenter{\xymatrix@=0.7em{
      & {Y}\ar[dl]_{c}\ar[dr]^{b} & \\
      {C}\ar[dr]_{g} & {\scriptstyle\Leftarrow\mu} & {B}\ar[dl]^{f} \\
      & {A} &
    }}
  \end{equation}
A 2-cell $\tau\colon b\Rightarrow \ell c$ satisfying~\eqref{eq:limit-lifting-prop} gives rise to a 
  2-cell $\nu$ from $m\colon Y\to f\comma g$ to the composite functor $tc\colon Y\to f\comma g$ over $C$ by 2-cell induction: notice that the fact that we require $\nu$ to be a 2-cell over $C$ means that the equation $q_1\nu = \id_{p_1}$ must hold, which tells us that the second 2-cell of its inducing pair must  be $\id_{p_1}$.  The compatibility condition expressed in~\eqref{eq:comma-ind-2cell-compat} for the pair $(\tau,\id_{p_1})$ reduces to the pasting equality~\eqref{eq:limit-lifting-prop} by direct application of the defining properties for $m$ and $t$ given in~\eqref{eq:m-def-prop} and~\eqref{eq:t-for-lim-def-prop}. Conversely, if $\nu\colon m\Rightarrow tc$ is any 2-cell over $C$ then the whiskered 2-cell $\tau\defeq q_0\nu\colon b \Rightarrow \ell c$ satisfies~\eqref{eq:limit-lifting-prop}.

Extending Definition~\ref{defn:enriched-slice}, the map $c$ defines a 2-functor $\hom'_C(c,-)\colon\qCat_2\slice C \to \Cat_2$. As in Observation~\ref{obs:isofib-section.fibred.adjunction}, this 2-functor carries the postulated fibred adjunction $q_1\dashv t$ to a terminal object  $tc\colon Y\to f\comma g$ in the hom-category $\hom'_C(c,q_1)$. It follows that there exists a unique 2-cell $\nu\colon m\Rightarrow tc$ over $C$; hence, the 2-cell $q_0\nu\colon b \Rightarrow \ell c$ provides us with a solution to~\eqref{eq:limit-lifting-prop}. Furthermore if $\tau\colon b\Rightarrow \ell c$ is any other 2-cell which solves that pasting equality then the 2-cell it induces must necessarily be the unique such $\nu\colon m\Rightarrow tc$, and consequently we have the equality $\tau = q_0\nu$. This demonstrates that the solution to~\eqref{eq:limit-lifting-prop} is unique.
\end{proof}

\begin{obs}\label{obs:right.liftings.as.fibred.terminal.objects}
  The upshot of Proposition \ref{prop:right.liftings.as.fibred.terminal.objects} is that if the projection $q_1\colon f\comma g\tfib C$ has a fibred right adjoint~\eqref{eq:fibred.terminal.2}, then we may compose it with the weakly universal cone associated with $f\comma g$ to obtain an absolute right lifting of $g$ through $f$.
\end{obs}

This characterisation of absolute right liftings leads to the following generalisation of a classical result:

\begin{prop}\label{prop:translated.lifting}
  There exists an absolute right lifting
  \begin{equation}\label{eq:orig.lifting}
    \xymatrix{ \ar@{}[dr]|(.7){\Downarrow\lambda} & B \ar[d]^f \\ C \ar[r]_g \ar[ur]^\ell & A}
  \end{equation}
  if and only if there exists an absolute right lifting
  \begin{equation}\label{eq:translated.lifting}
    \xymatrix{ \ar@{}[dr]|(.7){\Downarrow\hat\lambda} & {f\comma A} \ar@{->>}[d]^{p_1} \\ C \ar[r]_g \ar[ur]^{\hat\ell} & A}
  \end{equation}
  Furthermore, the 2-cell $\hat\lambda$ is necessarily an isomorphism and $\hat\ell$ may be chosen so as to make it an identity.
\end{prop}

\begin{proof} Write $(r_1,r_0)\colon p_1\comma g \tfib C \times f\comma A$ for the projection defined by the comma quasi-category construction~\ref{def:comma-obj}. Directly from this definition, there exists a canonical isomorphism $p_1\comma g \cong A\comma g\times_A f\comma A$ commuting with the projections to $C \times f \comma A$. Applying Proposition~\ref{prop:right.liftings.as.fibred.terminal.objects}, our aim is to use a fibred right adjoint to $q_1$ to construct a fibred right adjoint to $r_1$ and vice versa. 
  \begin{equation}\label{eq:ran.as.fibred.adj}
 \xymatrix@=1.2em{
      {C}\ar@{=}[dr]\ar@/_1.5ex/[rr]_-{t}^-{}="one"
      & & *+!L(0.5){f\comma g}\ar@{->>}[dl]^-{q_1}
      \ar@/_1.5ex/[ll]_-{q_1}^-{}="two" \\
      & {C} &
      \ar@{}"one";"two"|{\bot}
    }    \qquad  \quad  
     \xymatrix@=1.2em{
      {C}\ar@{=}[dr]\ar@/_1.5ex/[rr]^-{}="one"
      & & *+!L(0.5){p_1\comma g}\ar@{->>}[dl]^-{r_1}
      \ar@/_1.5ex/[ll]_-{r_1}^-{}="two" \\
      & {C} &
      \ar@{}"one";"two"|{\bot}
    }
  \end{equation}

To that end, pull back  the ``composition--identity'' fibred adjunctions~\eqref{eq:comp.ident.adj} along the functor $g\times p_1\colon C\times f\comma A\to A\times A$ to obtain a pair of adjunctions 
  \begin{equation}\label{eq:adj.for.trans}
    \xymatrix@C=14em{
      {p_1\comma g\cong A\comma g\times_A f\comma A}
      \ar[]!R(0.72);[r]|*+{\scriptstyle m} &
      {f\comma g}
      \ar@/^2.5ex/[l]!R(0.72)^{i_1}_{}="l" \ar@/_2.5ex/[l]!R(0.72)_{i_0}^{}="u"
      \ar@{} "u";"l" |(0.2){\bot} |(0.8){\bot} 
    }
  \end{equation}
  fibred over $C\times f\comma A$. Pushing forward along the projection $C\times f\comma A\tfib C$, we may regard the adjunctions \eqref{eq:adj.for.trans} as fibred over $C$ with respect to the isofibrations $r_1\colon p_1\comma g\tfib C$ and $q_1\colon f\comma g\tfib C$. 

  With this adjunction in our armoury our result is essentially immediate. If we are given the left-hand fibred adjunction~\eqref{eq:ran.as.fibred.adj} witnessing the existence of the absolute right lifting of $g$ through $f$ then we may compose it with the lower fibred adjunction of~\eqref{eq:adj.for.trans} to obtain the right-hand fibred adjunction~\eqref{eq:ran.as.fibred.adj}, providing us with an absolute right lifting of $g$ through $p_1$. Conversely, we may go back in the other direction by composing the right-hand fibred adjunction with the upper fibred adjunction of~\eqref{eq:adj.for.trans} to obtain an adjunction of the type on the left of~\eqref{eq:ran.as.fibred.adj}.

  All that remains is to check the final clause of the proposition. To that end, Observation~\ref{obs:right.liftings.as.fibred.terminal.objects} tells us that we may construct an absolute right lifting of $g$ through $p_1$ by composing the right adjoint functor 
\begin{equation*}
  \xymatrix{
    {C}\ar[r]^-{t} & {f\comma g}\ar[r]^-{i_1} & {p_1\comma g}
  }
\end{equation*}
where $t$ is the fibred right adjoint of~\eqref{eq:ran.as.fibred.adj}, with the comma cone that displays $p_1\comma g$ as a weak comma object. By construction, the 2-cell of that cone is the restriction
\begin{equation*}
  \xymatrix{
    {p_1\comma g} \ar[r] & {A\comma g}\ar[r] & {A^\cattwo}\ar@{}[]!R(0.5);[rr]!L(0.5)|{\Downarrow}\ar@/^1.5ex/[]!R(0.5);[rr]!L(0.5)^{p_0}\ar@/_1.5ex/[]!R(0.5);[rr]!L(0.5)_{p_1} && {A}
  }
\end{equation*}
of the 2-cell which displays $A^\cattwo$ as a weak cotensor. Hence, the 2-cell $\hat\lambda$ constructed by Proposition~\ref{prop:right.liftings.as.fibred.terminal.objects}  is equal to 
\begin{equation*}
  \xymatrix{
    {C}\ar[r]^-{t} & {f\comma g}\ar[r] & {A^\cattwo}\ar[r]^-{i_1} & {A^\cattwo\times_A A^\cattwo}\ar[r]^-{\pi_1} & {A^\cattwo}\ar@{}[]!R(0.5);[rr]!L(0.5)|{\Downarrow}\ar@/^1.5ex/[]!R(0.5);[rr]!L(0.5)^{p_0}\ar@/_1.5ex/[]!R(0.5);[rr]!L(0.5)_{p_1} && {A}
    }
\end{equation*}
  and, consulting the definition of $i_1$ given in Example~\ref{ex:comp.ident.adj}, it is straightforward to verify that the composite of the last three cells above is equal to the identity 2-cell on $p_1\colon A^\cattwo\tfib A$. Consequently, the 2-cell in our absolute right lifting is also an identity as required.
\end{proof}

\subsection{Limits and colimits as absolute lifting diagrams}

A diagram in a quasi-category $A$ is just a map $d \colon X \to A$ of simplicial sets. In particular, when $X$ is the nerve of a small category and $A$ is the homotopy coherent nerve of a locally Kan simplicial category, a diagram is precisely a \emph{homotopy coherent diagram} in the sense of Cordier, Porter, Vogt, and others \cite{Cordier:1986:HtyCoh}.

\begin{ntn}
  From here on  we use $c\colon A\to A^X$ to denote the {\em constant diagram map}: the adjoint transpose of the projection map $\pi_A\colon A\times X\to A$. Furthermore, we shall notationally identify functors $f\colon X\to A$ and natural transformations $\alpha\colon f\Rightarrow g\colon X\to A$ with their adjoint transposes $f\colon\Del^0\to A^X$ and $\alpha\colon f\Rightarrow g\colon \Del^1\to A^X$ respectively.
\end{ntn}

\begin{defn}\label{defn:limit} We say that an absolute right lifting diagram 
    \begin{equation}\label{eq:genericlimit}
      \xymatrix{ \ar@{}[dr]|(.7){\Downarrow\lambda} & A \ar[d]^c \\ \Delta^0 \ar[ur]^\ell \ar[r]_d& A^X}
    \end{equation}
    {\em displays the vertex $\ell\in A$ as the limit of the diagram $d\colon X\to A$}. The 2-cell $\lambda$, which we may equally regard as going from the constant diagram  $X \xrightarrow{!} \Del^0 \xrightarrow{\ell} A$ to $d$, is called the \emph{limiting cone}. Dually, we say that an absolute left lifting diagram 
  \begin{equation}\label{eq:genericcolimit} 
    \xymatrix{ \ar@{}[dr]|(.7){\Uparrow\lambda} & A \ar[d]^c \\ \Delta^0 \ar[ur]^\ell \ar[r]_d & A^X}
  \end{equation} 
  {\em displays the vertex $\ell\in A$ as the colimit of the diagram $d\colon X\to A$}.  Here again the 2-cell $\lambda$, from $d$  to the constant diagram $X \xrightarrow{!} \Del^0 \xrightarrow{\ell} A$, is called the \emph{colimiting cone}.
\end{defn}

\begin{rmk}
For the most part in what follows, we shall present our results in terms of limits and absolute right liftings only. Of course, these arguments all admit the obvious duals which apply to colimits and absolute left liftings. Indeed the results of this section and the last are almost exclusively matters of formal 2-category theory. Their duals follow by re-interpreting these arguments in the dual 2-category $\qCat_2\co$ obtained by reversing the direction of all 2-cells.
\end{rmk}

A special case of Proposition~\ref{prop:right.liftings.as.fibred.terminal.objects} gives an alternative definition of limits and colimits in a quasi-category.

\begin{prop}\label{prop:limits.as.terminal.objects} A limit of $d \colon X \to A$ is a terminal object in the quasi-category $c\comma d$, and conversely a terminal object defines a limit.
\end{prop}
\begin{proof}
A limiting cone defines a vertex in the comma quasi-category $c\comma d$ by 1-cell induction; Lemma~\ref{lem:1cell-ind-uniqueness} and Proposition~\ref{prop:right.liftings.as.fibred.terminal.objects} tell us this vertex is unique up to isomorphism and terminal. Conversely, Proposition~\ref{prop:right.liftings.as.fibred.terminal.objects} implies that the data of a terminal object in $c\comma d$ defines a limit object $\ell\in A$ and a limiting cone $\lambda$ in the sense of Definition~\ref{defn:limit}.
\end{proof}

An important corollary of Proposition~\ref{prop:limits.as.terminal.objects} is that our definition of limit agrees with the existing ones in the literature. As discussed in section \ref{subsec:join} and seen already in the proof of Proposition \ref{prop:terminalconverse}, this proof makes use of an equivalence between Joyal's slice construction and our comma construction. In this case we show that the quasi-category of cones $c \comma d$ over a diagram $d \colon X \to A$ is equivalent to Joyal's quasi-category of cones $\slicer{A}{d}$, recalled in \ref{defn:slices}. 

\begin{lem}\label{lem:cone-equiv-fatcone} For any diagram $d \colon X \to A$ in a quasi-category $A$, there is an equivalence
\[ \xymatrix{ \slicer{A}{d} \ar[rr]^\simeq \ar@{->>}[dr]_{\pi} & & c \comma d \ar@{->>}[dl]^{q_0} \\ & A}\] of quasi-categories over $A$.
\end{lem}
\begin{proof}
As in the proof of Lemma \ref{lem:slice-equiv-comma}, we will demonstrate an isomorphism $c \comma d \cong \fatslicer{A}{d}$ over $A$ between the quasi-category of cones and the fat slice construction on $d \colon X \to A$ defined in \ref{defn:fat-slices}. Via this isomorphism, the equivalence $\slicer{A}{d}\simeq c \comma d$ is a special case of the equivalence of Proposition~\ref{prop:slice-fatslice-equiv}.

To establish the isomorphism, it suffices to show that $c \comma d$ has the universal property that defines $\fatslicer{A}{d}$.  By adjunction, a map $Y \to \fatslicer{A}{d}$ corresponds to a commutative square, as displayed on the left:
\[ \vcenter{   
      \xymatrix@R=2em@C=4em{
        {(Y\times X)\sqcup(Y\times X)}\ar[r]^-{\pi_Y\sqcup\pi_X}
        \ar[d] &
        {Y\sqcup X}\ar[d]^{\langle f, d\rangle} \\
        {Y\times\Del^1\times X} \ar[r]_-{k} &
        {A}
      }
} \qquad \leftrightsquigarrow \qquad \vcenter{      \xymatrix@R=2em@C=4em{
        {Y}\ar[r]^{k}\ar[d]_{(!,f)} & {(A^X)^{\Del^1}} 
        \ar[d] \\
        { \Del^0\times A}\ar[r]_-{d\times c} & {A^X\times A^X}
      }}\]
which transposes to the commutative square displayed on the right. The data of the right-hand square is precisely that of a map $Y \to c \comma d$ by the universal property of the pullback \ref{def:comma-obj} defining the comma quasi-category.
\end{proof}

Joyal defines a limit of a diagram $d \colon X \to A$ to be a terminal vertex $t$ in the slice quasi-category $\slicer{A}{d}$, thought of as the ``quasi-category of cones'' over $d$.  If $\pi\colon \slicer{A}{d}\tfib A$ denotes the canonical projection then such a limiting cone displays $\ell\defeq \pi t$ as a limit of $d$. 

\begin{prop}\label{prop:limits.are.limits}
The notion of limit and limit cone introduced in Definition \ref{defn:limit} is equivalent to the notion of limit and limit cone introduced by Joyal in \cite[4.5]{Joyal:2002:QuasiCategories}.
\end{prop}
\begin{proof}
By Proposition~\ref{prop:limits.as.terminal.objects} tells us that our definition can be recast in a corresponding form: as a terminal vertex $t$ in  the comma quasi-category $c\comma d$. Our ``quasi-category of cones'' is  equipped with a projection $q_0 \colon c \comma d \to A$, and by Proposition~\ref{prop:right.liftings.as.fibred.terminal.objects} such a limiting cone displays $\ell\defeq q_0t$ as a limit of $d$.

Lemma \ref{lem:cone-equiv-fatcone} supplies an equivalence over $A$ between the quasi-category of cones and Joyal's slice quasi-category $\slicer{A}{d}$. Applying Proposition~\ref{prop:terminaldefn}, our preservation result for terminal objects, we see that this equivalence maps a limit cone in Joyal's sense to a limit cone in our sense and vice versa. Furthermore, since this is an equivalence over $A$, it follows that these corresponding cones display the same vertex $\ell$ as the limit of $d$.
\end{proof}

\begin{defn}\label{defn:families.of.diagrams}
  A {\em family $k$ of diagrams of shape $X$\/} in a quasi-category $A$ is simply a functor $k\colon K\to A^X$. In many cases, $K$ will  be the full sub-quasi-category of $A^X$ determined by some set of diagrams and $k$ will be the inclusion $K\inc A^X$.

  We say that $A$ {\em admits limits of the family of diagrams $k\colon K\to A^X$\/} if there exists an absolute right lifting diagram:
\begin{equation}\label{eq:limits.of.a.family}
      \xymatrix{ \ar@{}[dr]|(.7){\Downarrow\lambda} & A \ar[d]^c \\ K \ar[ur]^\lim \ar[r]_k & A^X}
\end{equation}
  Furthermore, we shall simply say that $A$ admits {\em all limits of shape\/} $X$ if it admits limits of the family of all diagrams $A^X$. 

  A diagram $d\colon X\to A$ is said to be a member of the family $k$ if it is a vertex in the image of $k$, that is to say if there is a vertex $\bar{d}\in K$ such that $d=k\bar{d}$. It is trivially verified, directly from the universal property of absolute right liftings, that if $A$ admits limits of the family of diagrams $k$ and $d$ is a member of the family $k$ then the restricted triangle
  \begin{equation*}
      \xymatrix{ \ar@{}[dr]|(.7){\Downarrow\lambda\bar{d}} & A \ar[d]^c \\ \Del^0 \ar[ur]^{\lim \bar{d}} \ar[r]_d & A^X}
  \end{equation*}
  is again an absolute right lifting, thus providing us with a limit of individual diagram $d$. Our use of the adjective ``absolute'' here coincides with its usual meaning: absolute lifting diagrams are preserved by pre-composition by all functors.
\end{defn}

This result has the following converse, whose proof we delay to section~\ref{sec:pointwise}:

\begin{prop}\label{prop:families.of.diagrams}
  If $A$ admits the limit of each individual diagram $d\colon X\to A$ in the family $k\colon K\to A^X$ then it admits limits of the family of diagrams $k$.
\end{prop}

As a special case of Example \ref{ex:adjasabslifting}:

\begin{prop}\label{prop:limitsasadjunctions} A quasi-category $A$ has all limits of shape $X$ if and only if there exists an adjunction \[ \adjdisplay c -| \lim : A^X -> A.\]
\end{prop}

A key advantage of our 2-categorical definition of (co)limits in any quasi-category is that it permits us to use standard 2-categorical arguments to give easy proofs of the expected categorical theorems.

\begin{prop}\label{prop:RAPL} Right adjoints preserve limits.
\end{prop}

Our proof will closely follow the classical one. Given a diagram $d\colon X \to A$ and a right adjoint $u \colon A \to B$ to some functor $f$, a cone with summit $b$ over $ud$ transposes to a cone with summit $fb$ over $d$, which factorises uniquely through the limit cone. This factorisation transposes back across the adjunction to show that the image of the limit cone under $u$ defines a limit over $ud$.

\begin{proof}
Suppose that $A$ admits limits of a family of diagrams $k\colon K\to A^X$ as witnessed by an absolute right lifting diagram~\eqref{eq:limits.of.a.family}. Given an adjunction $f \dashv u$, and hence by Proposition \ref{prop:expadj} an adjunction $f^X \dashv u^X$, we must show that \[\xymatrix{ \ar@{}[dr]|(.7){\Downarrow\lambda} & A \ar[d]^-{c} \ar[r]^u & B \ar[d]^-{c} \\ K \ar[ur]^\lim \ar[r]_k& A^X \ar[r]_{u^X} & B^X}\] is an absolute right lifting diagram. Given a square
\[\xymatrix{ Y \ar[d]_{a} \ar[rr]^b \ar@{}[drr]|{\Downarrow\chi} & & B \ar[d]^-{c} \\ K \ar[r]_-{k} & A^X \ar[r]_{u^X} & B^X} \] we first transpose across the adjunction, by composing with $f$ and the counit. 
\[\vcenter{\xymatrix{ Y \ar[d]_-{a} \ar[rr]^b \ar@{}[drr]|{\Downarrow\chi} & & B \ar[d]^-{c} \ar[r]^f & A \ar[d]^-{c}  \\ K \ar[r]_-{k} & A^X \ar@{=}@/_3.5ex/[rr]^{\Downarrow\epsilon^X} \ar[r]^{u^X} & B^X \ar[r]^{f^X} & A^X}} = \vcenter{\xymatrix{ Y \ar@{}[drr]|(.3){\exists !\Downarrow\zeta}|(.7){\Downarrow\lambda} \ar[d]_-{a} \ar[r]^b & B \ar[r]^f & A \ar[d]^-{c} \\ K \ar[urr]_(0.4){\lim} \ar[rr]_{k} & & A^X}} \] Applying the universal property of the absolute right lifting diagram~\eqref{eq:limits.of.a.family} produces a factorisation $\zeta$, which may then be transposed back across the adjunction by composing with $u$ and the unit.
\[  \vcenter{\xymatrix{ Y \ar@{}[drr]|(.3){\exists !\Downarrow\zeta}|(.7){\Downarrow\lambda} \ar[d]_-{a} \ar[r]^b & B \ar@{=}@/^3.5ex/[rr]_{\Downarrow\eta} \ar[r]|f & A \ar[d]^-{c} \ar[r]_u & B \ar[d]^-{c} \\ K \ar[urr]_(0.4){\lim} \ar[rr]_-{k} & & A^X \ar[r]_{u^X} & B^X}}= \vcenter{\xymatrix{ Y \ar[d]_-{a} \ar[rr]^b \ar@{}[drr]|{\Downarrow\chi} & & B \ar[d]^-{c} \ar@{=}@/^3.5ex/[rr]_{\Downarrow\eta} \ar[r]_f & A \ar[d]^-{c} \ar[r]_u & B \ar[d]^-{c}  \\ K \ar[r]_-{k} & A^X \ar@{=}@/_3.5ex/[rr]^{\Downarrow\epsilon^X} \ar[r]^{u^X} & B^X \ar[r]^{f^X} & A^X \ar[r]_{u^X} & B^X}}  \] \[ = \vcenter{\xymatrix{ Y \ar[d]_-{a} \ar[rr]^b \ar@{}[drr]|{\Downarrow\chi} & & B \ar[d]^-{c} \ar@{=}@/^3.5ex/[rr]  &  & B \ar[d]^-{c}  \\ K \ar[r]_-{k} & A^X \ar@{=}@/_3.5ex/[rr]^{\Downarrow\epsilon^X} \ar[r]^{u^X} & B^X \ar[r]|{f^X}  \ar@{=}@/^3.5ex/[rr]_{\Downarrow\eta^X}& A^X \ar[r]_{u^X} & B^X}}  = \vcenter{\xymatrix{ Y \ar[d]_-{a} \ar[rr]^b \ar@{}[drr]|{\Downarrow\chi} & & B \ar[d]^-{c} \\ K \ar[r]_-{k} & A^X \ar[r]_{u^X} & B^X}}\] Here the second equality is immediate from the definition of $\eta^X$ and the third is by the triangle identity defining the adjunction $f^X \dashv u^X$. The pasted composite of $\zeta$ and $\eta$ is the desired factorisation of $\chi$ through $\lambda$. 

The proof that this factorisation is unique, which again parallels the classical argument, is left to the reader: the essential point is that the transposes are unique.
\end{proof}

\begin{cor}\label{cor:equivprescolim} Equivalences preserve limits and colimits.
\end{cor}
\begin{proof} This follows immediately from Propositions \ref{prop:RAPL} and \ref{prop:equivtoadjoint}.
\end{proof}

\begin{obs}\label{obs:transpose-abs-lifting}
  Under the 2-adjunction $-\times Y\dashv (-)^Y$ triangles of the form 
  \begin{equation}\label{eq:untransposed}
    \xymatrix{
      & {B}\ar[d]^-{f} \\
      {K\times Y} \ar[ur]^{\ell} \ar[r]_-{k} 
      & {A} \ar@{}[ul]|(0.35){\Downarrow \lambda}
    }
  \end{equation}
  correspond to transposed diagrams:
  \begin{equation}\label{eq:transposed}
    \xymatrix{
      & {B^Y}\ar[d]^-{f^Y} \\
      {K} \ar[ur]^{\hat\ell} \ar[r]_-{\hat{k}} 
      & {A^Y} \ar@{}[ul]|(0.35){\Downarrow \hat\lambda}
    }
  \end{equation}
  Furthermore, if the first of these triangles is an absolute right lifting then so is the second one. To prove this, we must show that we can uniquely factorise the 2-cell in a square \[ \xymatrix{ Z \ar[d]_-{u} \ar[r]^-{v} \ar@{}[dr]|{\Downarrow\alpha} & B^Y \ar[d]^-{f^Y} \\ K \ar[r]_-{\hat{k}} & A^Y} \] through the 2-cell $\hat\lambda$ in~\eqref{eq:transposed}. Transposing that square under the 2-adjunction, we obtain the square on the left of the following diagram: \[ \vcenter{\xymatrix{ Z \times Y \ar[d]_-{\tilde{u}} \ar[r]^-{\tilde{v}} \ar@{}[dr]|{\Downarrow \tilde{\alpha}} & B \ar[d]^-{f} \\ K\times Y \ar[r]_-{k} & A}} = \vcenter{\xymatrix{ Z \times Y \ar[d]_-{\tilde{u}} \ar[r]^-{\tilde{v}} \ar@{}[dr]|(.3){\exists !\Downarrow}|(.7){\Downarrow\lambda} & B \ar[d]^-{f} \\ K\times Y \ar[r]_-{k} \ar[ur]_(0.4){\ell} & A}}\] The unique factorisation on the right arises from the universal property of the absolute lifting diagram~\eqref{eq:untransposed}, and its transpose provides the desired unique factorisation of $\alpha$.
\end{obs}

\begin{prop}[pointwise limits in functor quasi-categories]\label{prop:pointwise-limits-in-functor-quasi-categories}
  If a quasi-category $A$ admits limits of the family of diagrams $k\colon K\to A^X$ of shape $X$ then the functor quasi-category $A^Y$ admits limits of the corresponding family of diagrams $k^Y\colon K^Y\to (A^X)^Y\cong(A^Y)^X$ of shape $X$.
\end{prop}

\begin{proof}
  On precomposing the absolute right lifting that displays the limits of the family $k\colon K\to A^X$ \eqref{eq:limits.of.a.family} by the evaluation map $\ev\colon K^Y\times Y\to K$, we obtain an absolute right lifting diagram whose adjoint transpose under the 2-adjunction ${-}\times Y\dashv (-)^Y$ is the triangle
  \begin{equation*}
      \xymatrix{ \ar@{}[dr]|(.7){\Downarrow\lambda^Y} & A^Y \ar[d]^-{c^Y} \\ K^Y \ar[ur]^{\lim^Y} \ar[r]_-{k^Y} & (A^X)^Y}
\end{equation*}
  By the last observation, this is again an absolute right lifting diagram which, on composition with the canonical isomorphism $(A^X)^Y\cong(A^Y)^X$, displays $\lim^Y$ as the family of limits required in the statement.
\end{proof}

Proposition \ref{prop:limitsasadjunctions} tells us that if $A$ has all limits of shape $X$, then there is a functor $\lim \colon A^X \to A$ that is right adjoint to the constant functor $c \colon A \to A^X$. In ordinary category theory we often deploy another adjunction related to the existence of limits of shape $X$, this being the restriction--right Kan extension adjunction between diagrams of shape $X$ and diagrams whose shape is that of a cone over $X$.

The shape of a cone over a diagram of shape $X$ is given by the simplicial set $\Del^0\join X$, defined using Joyal's join construction of Definition~\ref{defn:join-dec}.

\begin{prop}\label{prop:ran.adj.limits} A quasi-category $A$ admits limits of the family of diagrams $k\colon K\to A^X$ of shape $X$ if and only if there exists an absolute right lifting diagram
\begin{equation*}
  \xymatrix{
    & *+[r]{A^{\Del^0\join X}}\ar@{->>}[d]^-{\res} \\
    {K} \ar[ur]^{\ran} \ar[r]_-{k} 
    & *+[r]{A^X} \ar@{}[ul]|(0.3){\Downarrow\lambda}
  }
\end{equation*}
in which  $\res$  is the restriction isofibration given by pre-composition with the inclusion $X\inc \Del^0\join X$. Furthermore, when these equivalent conditions hold $\lambda$ is necessarily an isomorphism and, indeed, we may choose $\ran$ so that $\lambda$ is an identity.
\end{prop}

\begin{proof}
By Proposition~\ref{prop:join-fatjoin-equiv}, the canonical comparison $\Del^0\fatjoin X\to\Del^0\join X$ is a weak equivalence in Joyal's model structure. So if $A$ is a quasi-category, it follows, by Proposition~\ref{prop:equivsareequivs2}, that the associated pre-composition functor $A^{\Del^0\join X}\to A^{\Del^0\fatjoin X}$ is an equivalence of quasi-categories. Now the contravariant exponential functor $A^{({-})}\colon \sSet\op\to \qCat$ carries colimits to limits so it is immediate, from Definition~\ref{eq:fat-join-def}, that we have a pullback \[ \xymatrix{ A^{\Delta^0\fatjoin X} \pbexcursion \ar[r] \ar[d] & A^{X \times \Delta^1} \ar[d] \\ A \times A^X\cong A^{\Delta^0 \sqcup X}  \ar[r] & A^{X \sqcup X} \cong A^X \times A^X}\] from which we see that $A^{\Delta^0\fatjoin X}$ is isomorphic to the comma quasi-category $c\comma A^X$. It is now easily checked that a triangle of the form given in the statement is an absolute right lifting if and only if the following rearranged triangle 
\begin{equation*}
  \vcenter{\xymatrix{
      & {c\comma A^X} \ar@{->>}[d]^-{p_1} \\
      {K} \ar[ur]^{\ran} \ar[r]_-{k} 
      & {A^X} \ar@{}[ul]|(0.3){\Downarrow\lambda}
  }} \mkern20mu \defeq \mkern20mu
  \vcenter{\xymatrix{
    & *+[r]{A^{\Del^0\join X}}\ar@{->>}[d]^-{\res}\ar[r]^-{\sim} &  {c\comma A^X} \ar@{->>}[dl]^{p_1}\\
    {K} \ar[ur]^{\ran} \ar[r]_-{k} 
    & *+[r]{A^X} \ar@{}[ul]|(0.3){\Downarrow\lambda} &
  }}
\end{equation*}
has that property; now the current result is merely a special case of Proposition~\ref{prop:translated.lifting}.
\end{proof}

\begin{cor}\label{cor:ran.adj.limits} A quasi-category $A$ admits all limits of shape $X$ if and only if the restriction functor associated with the inclusion $X\inc \Del^0\join X$ has a fibred right adjoint \[ 
  \xymatrix@=1.5em{
    {A^X}\ar@/_1.2ex/[rr]_-\ran\ar@{=}[dr] \ar@{}[rr]|*{\bot} && 
    *+!L(0.5){A^{\Del^0\join X}}\ar@/_1.2ex/[ll]_-\res \ar@{->>}[dl]^{\res} \\
    & A^X &
  }
\] 
\end{cor}
\begin{proof}
Since the restriction functor $A^{\Del^0\join X} \tfib A^X$ is an isofibration, we may follow Example~\ref{ex:isofib-section.fibred.adjunction} and pick its right adjoint so that the counit of the adjunction $\res \dashv\ran$ is an identity. By Corollary~\ref{cor:missed-lemma}, this adjunction lifts to an adjunction fibred over $A^X$.
\end{proof}

As an application of some significant classical interest, we may use Proposition~\ref{prop:ran.adj.limits} to construct a loops--suspension adjunction in any pointed quasi-category admitting certain pullbacks and pushouts.

\begin{defn}[pointed quasi-categories]
    A \emph{zero object} in a quasi-category is an object in there that is both initial and terminal. We say that a quasi-category $A$ is {\em pointed\/} if it has a zero object and write $*\in A$ for that object. We call the counit $\rho\colon *!\Rightarrow\id_A$ of the adjunction $\adjinline * -| ! : A -> \Del^0.$ the {\em family of points\/} of the objects of $A$ and call the unit $\xi\colon \id_A\Rightarrow *!$ of the adjunction $\adjinline ! -| * : A -> \Del^0.$ the {\em family of co-points\/} of the objects of $A$.
\end{defn}

\begin{ntn}[pushout and pullback diagrams]\label{ntn:pb.po.joins}
  We shall adopt the following notation for certain important diagram shapes which arise naturally as simplicial subsets of the square $\Del^1\times\Del^1$:
  \begin{itemize}
    \item $\pbshape$ will denote the simplicial subset $(\Del^1\times\Del^{\fbv{1}})\cup(\Del^{\fbv{1}}\times\Del^1)$, and
    \item $\poshape$  will denote the simplicial subset $(\Del^1\times\Del^{\fbv{0}})\cup(\Del^{\fbv{0}}\times\Del^1)$.
  \end{itemize}
  Of course, $\pbshape$ and $\poshape$ are the shapes of pullback and pushout diagrams,   isomorphic to the horns $\Horn^{2,2}$ and $\Horn^{2,0}$ respectively.    The joins $\Del^0\join\pbshape$ and $\poshape\join\Del^0$ are each isomorphic to the square $\Del^1\times\Del^1$. These isomorphisms identify the canonical inclusions of those joins with the corresponding subset inclusions $\pbshape\inc\Del^1\times\Del^1$ and $\poshape\inc\Del^1\times\Del^1$ respectively.
\end{ntn}

\begin{defn}[pushouts and pullbacks in quasi-categories]
  A {\em pullback\/} in a quasi-category is a limit of a diagram of shape $\pbshape$. Dually a {\em pushout\/} in a quasi-category is a colimit of a diagram of shape $\poshape$. 
\end{defn}

\begin{obs}\label{obs:loops.diag.fam}
The family of points of a pointed quasi-category $A$ may be represented by a simplicial map $\rho\colon A\to A^\cattwo$. Now the pullback diagram shape $\pbshape$ may be represented as a glueing of two copies of $\cattwo$ identified at their initial vertex, so it follows that $A^\pbshape$ may be constructed as a pullback of two copies of $A^\cattwo$ along the projection $p_1\colon A^\cattwo\tfib A$. Consequently, two copies of $\rho$ give rise to a functor $\bar\rho\colon A\to A^\pbshape$. This functor maps each object $a$ of $A$ to a pushout diagram with outer vertices $*$, inner vertex $a$, and maps two copies of the component of $\rho$ at $a$. Dually we may define a corresponding functor $\bar\xi\colon A\to A^\poshape$ using two copies of the family of co-points.
\end{obs}

\begin{defn}[loop spaces and suspensions]\label{defn:loop.susp}
  We say that a pointed quasi-category $A$ admits the construction of {\em loop spaces\/} if it admits limits of the family of diagrams $\bar\rho\colon A\to A^\pbshape$.   Dually, we say that $A$ admits the construction of {\em suspensions\/} if  it admits colimits of the family of diagrams $\bar\xi\colon A\to A^\poshape$. These constructions, when they exist, are displayed by absolute right and left liftings
\begin{equation*}
  \xymatrix{ 
    \ar@{}[dr]|(.7){\Downarrow} & A \ar[d]^c \\ 
    A \ar[ur]^\Omega \ar[r]_{\bar\rho} & A^\pbshape
  }
  \mkern80mu
  \xymatrix{ 
    \ar@{}[dr]|(.7){\Uparrow} & A \ar[d]^c \\ 
    A \ar[ur]^\Sigma \ar[r]_{\bar\xi} & A^\poshape
  }
\end{equation*}
in which $\Omega$ is called the {\em loop space functor\/} and $\Sigma$ is called the {\em suspension functor}. Of course, if $A$ admits all pullbacks (resp.\ pushouts) then, as a special case, it admits the construction of loop spaces (resp. suspensions).
\end{defn}

\begin{ex}
  In the quasi-category of spaces, which we construct by applying the homotopy coherent nerve to the simplicially enriched category of Kan complexes, pushouts and pullbacks are constructed by taking classical homotopy pushouts and pullbacks. The quasi-category of pointed spaces is simply the slice under $\Del^0$ and its pushouts and pullbacks may be computed as in the quasi-category of spaces. It follows, therefore, that the loop space and suspension constructions in this quasi-category coincide with the usual notions in classical homotopy theory.
\end{ex}

The following proposition promotes our classical intuition about the relationship between loop and suspension constructions to a genuine adjunction of quasi-categories. To keep our proof as simple and transparent as possible, we choose to assume that the quasi-category here admits all pushouts and pullbacks, leaving it to the reader to generalise this result to one in which we only assume the existence of loop spaces and suspensions.

\begin{prop}\label{prop:loops-suspension} Suppose that $A$ is a pointed quasi-category which admits all pushouts and pullbacks. Then $A$ has a loops--suspension adjunction \[ \adjdisplay \Sigma -| \Omega: A -> A.\]
\end{prop}

\begin{proof}
  By Corollary~\ref{cor:ran.adj.limits} and the ruminations of~\ref{ntn:pb.po.joins}, the hypothesis that $A$ has pullbacks and pushouts implies that there are adjunctions
\begin{equation}\label{eq:pullback.pushout.adj}    
  \xymatrix@R=0em@!C=8em{
    {A^\pbshape}
    \ar@/_0.55pc/[r]!L(0.5)_-{\ran} 
    \ar@{}[r]!L(0.5)|-{\displaystyle\bot} 
    \ar@{<-}@/^0.55pc/[r]!L(0.5)^-{\res} & 
    {A^{\Delta^1\times\Delta^1}} & 
    {A^\poshape}
    \ar@/_0.55pc/[l]!R(0.5)_-{\lan}
    \ar@{<-}@/^0.55pc/[l]!R(0.5)^-{\res}  
    \ar@{}[l]!R(0.5)|-{\displaystyle\bot}
  }
\end{equation}
which are fibred over $A^\pbshape$ and $A^\poshape$, respectively. Now the inclusion of $\Del^0\sqcup\Del^0$ into $\Del^1\times\Del^1$ which picks out the vertices $(1,0)$ and $(0,1)$ factorises through each of the subsets $\pbshape$ and $\poshape$ and therefore induces restriction isofibrations $A^\pbshape\tfib A\times A$ and $A^\poshape\tfib A\times A$. So we may push forward our fibred adjunctions along these isofibrations to obtain a composable pair of adjunctions fibred over $A\times A$. Composing these and pulling back  along $(*,*)\colon \Del^0\to A\times A$, we obtain an adjunction
\begin{equation}\label{eq:loop.susp.var}
  \adjdisplay \overline\Sigma -| \overline\Omega : A^\pbshape_* -> A^\poshape_*.
\end{equation}
where $A^\pbshape_*\subseteq A^\pbshape$ and $A^\poshape_*\subseteq A^\poshape$ are the sub-quasi-categories of pullback and pushout diagrams whose outer vertices are pinned at the zero object $*$. 

The family of points $\rho\colon A\to A^\cattwo$ discussed in Observation~\ref{obs:loops.diag.fam} factorises through the sub-quasi-category $*\comma A\subseteq A^\cattwo$; hence,  the family of diagrams $\bar\rho\colon A\to A^\pbshape$ for the loop space construction also factorises through $A^\pbshape_*\subseteq A^\pbshape$. Furthermore, it is clear that the pullback expressing $A^\pbshape$ in terms of two copies of $A^\cattwo$ restricts to the pullback expressing $A^\pbshape_*$ in terms of two copies of $*\comma A$ in the following diagram:
\begin{equation*}
  \xymatrix@=1.5em{ 
    A\ar[dr]|*+{\scriptstyle\bar\rho}\ar@/^1.5ex/[drr]^\rho
    \ar@/_1.5ex/[ddr]_\rho &&\\
    & A^\pbshape_* \pbexcursion \ar@{->>}[d]
    \ar@{->>}[r] & {*\comma A} \ar@{->>}[d]^-{p_1} \\ 
    & {* \comma A} \ar@{->>}[r]_-{p_1} & A
  }
\end{equation*}

We claim that each functor in this diagram is an equivalence. To show this start by observing that the initiality of $*$ in $A$ implies that the isofibration $p_1$ is an equivalence, as is its right inverse $\rho$ by the 2-of-3 property. Trivial fibrations are stable under pullback, so the two projections from $A^\pbshape_*$ are equivalences, as is $\bar\rho$ by the 2-of-3 property. Observe also that the functor which restricts each pullback diagram to its inner vertex is an isofibration left inverse to $\bar\rho$ and so, by the 2-of-3 property, it too is an equivalence. The dual argument shows that the family of diagrams $\bar\xi\colon A\to A^\poshape$ for the suspension construction also factorises through $A^\poshape_*\subseteq A^\poshape$ to give an equivalence $\bar\xi\colon A\to A^\poshape_*$ with left inverse the isofibration that restricts each pullback diagram to its inner vertex.

  Now we may promote the equivalences $\bar\rho$ and $\bar\xi$  to adjoint equivalences and compose them with the adjunction~\eqref{eq:loop.susp.var}. The right adjoint in this composite adjunction is equal to the composite $\xymatrix@1{{A}\ar[r]^-{\bar\rho} & {A^\pbshape}\ar[r]^-{\ran} & {A^{\Del^1\times\Del^1}}\ar[r]^-{\res} & {A}}$ in which the last map is the restriction functor associated with the inclusion of $\Del^0$ as the vertex $(0,0)$ of $\Del^1\times\Del^1$. The composite of these last two functors is the pullback functor $\lim\colon A^\pbshape\to A$, so pre-composing it with $\bar\rho\colon A\to A^\pbshape$ produces a functor which picks out limits of the diagrams in the family $\bar\rho$. This must therefore be isomorphic to the loop space functor $\Omega$ by Definition~\ref{defn:loop.susp}. A dual argument demonstrates that the left adjoint in the composite adjunction is isomorphic to the suspension functor $\Sigma$, thus completing the verification that the adjunction we have constructed is the one asked for in the statement.
\end{proof}

\subsection{Geometric realisations of simplicial objects}

A classical result from simplicial homotopy theory states that if a simplicial object admits an augmentation together with a splitting, also called a contracting homotopy or simply ``extra degeneracies'', then the augmentation is homotopy equivalent to its geometric realisation. More precisely, the augmented simplicial object, a diagram of shape $\Del+\op$, defines a colimit cone over the restriction of this diagram to $\Del\op$. 

In this section, we import these ideas into the quasi-categorical context, proving that if a simplicial object in a quasi-category admits an augmentation and a splitting then the augmentation is its quasi-categorical colimit.
Again, the result is not new (cf.~\cite[6.1.3.16]{Lurie:2009fk}), but our proof closely mirrors the classical one (see, e.g.,~\cite{Meyer:84ba}). Specifically, we show that the structure of the contracting homotopies define an absolute left extension diagram in $\Cat$. Furthermore, this universal property is witnessed equationally and so is preserved by any 2-functor.  Dual remarks apply to cosimplicial objects admitting a coaugmentation and a splitting.

The first step is to describe the shape of a split simplicial object. There are two choices, distinguished by whether we choose a ``forwards'' or ``backwards'' contracting homotopy. The corresponding categories are opposites. Let $\Del[t]$ and $\Del[b]$ denote the subcategories of $\Del$ consisting of those maps that preserve the top or bottom element respectively in each ordinal. There is an inclusion  $[0]\oplus -\colon \Del+ \inc \Del[b]$ which freely adjoins a bottom element. Note the degree shift: this functor sends the initial object $[-1] \in \Del+$ to the zero object $[0]\in\Del[b]$. 

A simplicial object is \emph{augmented} if it admits an extension to $\Del+\op$ and \emph{split} if it admits a further extension to $\Del[t] \cong \Del[b]\op$. Evaluating at $[0] \in \Del[t]$ yields the augmentation. Restriction along the inclusion $\Del\op \inc \Del+\op\inc \Del[t]$ yields the original diagram. We will prove:

\begin{thm}\label{thm:splitgeorealizations} For any quasi-category $B$, the canonical diagram \[ \xymatrix{ \ar@{}[dr]|(.7){\Uparrow} & B \ar[d]^c \\ B^{\Del[t]} \ar[ur]^{\ev_0} \ar[r]_{\res} & B^{\Del\op}}\] is an absolute left lifting diagram. Hence, given any simplicial object admitting an augmentation and a splitting, the augmented simplicial object defines a colimit cone over the original simplicial object. Furthermore, such colimits are preserved by any functor.
\end{thm}

Our proof uses a 2-categorical lemma.

\begin{lem}\label{lem:doms2catlemma} Suppose given an adjunction in a slice 2-category $C\slice\tcat{C}$
\[\vcenter{\xymatrix@R=30pt{ & C \ar[dl]|b_{\rotatebox{45}{$\labelstyle\perp$}} \ar[dr]^a & \\ \ar@/^3ex/@{-->}[ur]^c B \ar@/^1ex/[rr]^f  \ar@{}[rr]|\perp & & A \ar@/^1ex/[ll]^u }}\] If $b$ admits a left adjoint $c$ in $\tcat{C}$ with unit $\iota$, then the 2-cell $f\iota \colon f \Rightarrow fbc=ac$ exhibits $c$ as an absolute left lifting of $f$ through $a$.
\end{lem}
\begin{proof}
Let $\nu$ be the counit of $c \dashv b$, and write $\eta$ and $\epsilon$ for the unit and counit of the adjunction $f\dashv u$; because this adjunction is under $C$ we have $\epsilon a = \id_a$ and $\eta b =\id_b$. Any 2-cell $\chi$ of the form displayed below factorises through $f\iota$ as follows
\[\xymatrix{ X \ar[d]_x \ar[r]^y \ar@{}[dr]|{\Uparrow\chi} & C \ar[d]^a \ar@{}[dr]|{\displaystyle =}  & X \ar[d]_x \ar[r]^y \ar@{}[dr]|{\Uparrow\chi} & C \ar[d]^a   \ar@{}[dr]|{\displaystyle =}   &X \ar@{}[dr]|{\Uparrow\chi} \ar[d]_x \ar[r]^y & C \ar[d]_a \ar[r]^a \ar[dr]|b & A   \ar@{}[dr]|{\displaystyle =} & X \ar[d]_x \ar[r]^y \ar@{}[dr]|{\Uparrow\chi} & C \ar[d]_a \ar[dr]^b \ar@{=}[rr] & \ar@{}[d]|{\Uparrow\nu} &  C \ar[d]^a  \\  B \ar[r]_f & A & B \ar[r]^f \ar@{=}[dr] & A \ar[d]|u \ar@{}[dl]|(.3){\Uparrow\eta} \ar@{=}[dr]  &  B \ar[r]^f \ar@{=}@/_3.5ex/[rr]^{\Uparrow\eta} & A \ar[r]^u & B \ar[u]_f &  B \ar[r]^f \ar@{=}@/_3.5ex/[rr]^{\Uparrow\eta} & A \ar[r]^u  & B \ar[ur]^c \ar[r]_f & A \ar@{}[ul]|(.3){\Uparrow f\iota} \\ & & & B \ar[r]_f \ar@{}[ur]|(.3){\Uparrow\epsilon} & A }\]
using a triangle identity for each adjunction and the fact that $\epsilon a = \id_a$. Such factorisations are unique because the 2-cell $\zeta$ can be recovered from the pasted composite with $f\iota$: \[\xymatrix{  X \ar[d]_x \ar[r]^y \ar@{}[dr]|(.3){\Uparrow\zeta}|(.7){\Uparrow f\iota} & C \ar[d]_a \ar[dr]^b \ar@{=}[rr] & \ar@{}[d]|{\Uparrow\nu} &  C  \\ B \ar[ur]|c \ar[r]^f \ar@{=}@/_3.5ex/[rr]^{\Uparrow\eta} & A \ar[r]^u & B \ar[ur]^c & {\displaystyle =} } \xymatrix{ X \ar[d]_x \ar[r]^y \ar@{}[dr]|(.3){\Uparrow\zeta}|(.7){\Uparrow\iota} & C \ar[d]|b \ar[dr]|a \ar[drr]^b \ar@{=}[rr]& & C \ar@{}[dl]|(.4){\Uparrow\nu} \\B \ar[ur]|c \ar@{=}[r] & B  \ar[r]_f \ar@{=}@/_3.5ex/[rr]^{\Uparrow\eta} & A \ar[r]_u & B \ar[u]_c} = \xymatrix{ X \ar[d]_x \ar[r]^y \ar@{}[dr]|(.3){\Uparrow\zeta}|(.7){\Uparrow\iota} & C \ar[d]|b \ar@{}[dr]|(.3){\Uparrow\nu} \ar@{=}[r] & C    \\  B \ar[ur]|c \ar@{=}[r] & B \ar[ur]_c & {\displaystyle =}  }
\xymatrix{ X \ar[d]_x \ar[r]^y \ar@{}[dr]|(.3){\Uparrow\zeta} & C \\ B \ar[ur]_c & }\qedhere
\] 
\end{proof}

\begin{proof}[Proof of Theorem \ref{thm:splitgeorealizations}] The inclusion $\Del\op \hookrightarrow \Del[t]$ admits a left adjoint. One way to define it is to present  $\Del\op$ via the ``interval representation'': after employing a degree shift $[n] \mapsto [n+1]$, $\Del\op$ is the subcategory of $\Del+$ consisting of ordinals with distinct top and bottom elements and maps that preserve these. Most generally, we might think of the interval representation as the diagonal composite functor in the pullback diagram \[\xymatrix{ \Del+\op \pbexcursion \ar[d] \ar[r] & \Del[t] \ar@{_(->}[d] \\ \Del[b] \ar@{^(->}[r] & \Del+}\] The arrows $\Del[b] \leftarrow \Del+\op \to \Del[t]$ extend the category indexing augmented simplicial objects by introducing extra maps that define ``extra degeneracies'' either on the left or on the right. The restricted functor $\Del\op \to \Del[t]$ is the inclusion described above. It has a left adjoint: a map $\alpha \colon [k] \to [n+1]$ in $\Del[t]$ is given by a map $[n] \to [k]$ in $\Del$ that sends $i\in [n]$, thought of as a ``gap'' between adjacent elements in $[n+1]$, to the minimal $j \in [k]$ so that $\alpha(j) = i+1$. 

 For any quasi-category $B$, the 2-functor $B^{(-)} \colon \Cat_2\op \to \qCat_2$ carries the
adjoint functors 
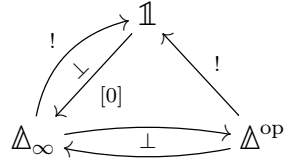
\[\xymatrix@R=30pt{ & \catone \ar[dl]^{\labelstyle[0]} & \\ \Del[t] \ar@/^1ex/[rr] \ar@{}[rr]|\perp  \ar@/^2.5ex/[ur]^{!} \ar@{}[ur]^*-{\rotatebox{45}{$\labelstyle\perp$}} & & \Del\op \ar@/^/[ll] \ar[ul]_{!}}\]
to an adjunction in the slice 2-category $B\slice\qCat_2$
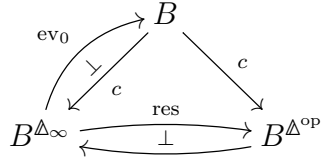
\[\xymatrix@R=30pt{ & B \ar[dl]^{\labelstyle c} \ar[dr]^c & \\ B^{\Del[t]} \ar@/^1ex/[rr]^{\res} \ar@{}[rr]|\perp  \ar@/^2.5ex/[ur]^{\ev_0} \ar@{}[ur]^*-{\rotatebox{45}{$\labelstyle\perp$}} & & B^{\Del\op} \ar@/^/[ll] }\] The 2-cell defined by whiskering $\res$ with the unit of $\ev_0\dashv c$ is the 2-cell $\res \Rightarrow c \cdot \ev_0$ obtained by applying the 2-functor $B^{-}$ to the unique 2-cell
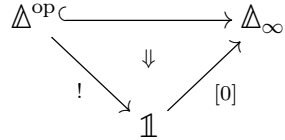
\[\xymatrix{ \Del\op \ar@{^(->}[rr] \ar[dr]_{!} & \ar@{}[d]|(.35){\Downarrow} & \Del[t] \\ & \catone \ar[ur]_{[0]} & } \] that exists because $[0] \in \Del[t]$ is terminal. The result now follows from Lemma \ref{lem:doms2catlemma}.

It remains only to prove the last statement. Given any functor $f \colon B \to A$, the diagrams \[ \vcenter
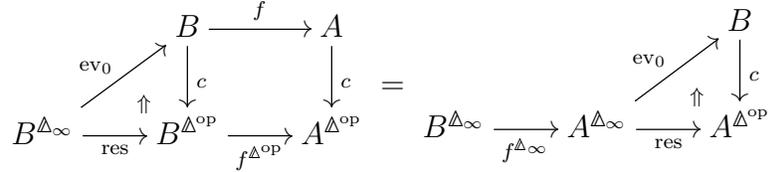
{\xymatrix{ \ar@{}[dr]|(.7){\Uparrow} & B \ar[d]^c \ar[r]^f & A \ar[d]^c \\ B^{\Del[t]} \ar[ur]^{\ev_0} \ar[r]_{\res} & B^{\Del\op} \ar[r]_{f^{\Del\op}} & A^{\Del\op}}} = \vcenter{ \xymatrix{ &  \ar@{}[dr]|(.7){\Uparrow} & B \ar[d]^c \\ B^{\Del[t]} \ar[r]_{f^{\Del[t]}} & A^{\Del[t]} \ar[ur]^{\ev_0} \ar[r]_{\res} & A^{\Del\op}}}\] coincide by bifunctoriality of the internal hom 2-functor in $\qCat_2$. In particular, the left-hand side inherits the universal property of the right-hand side.
\end{proof}

\begin{ex} Theorem \ref{thm:splitgeorealizations} can be used to prove that any object in the quasi-category of algebras associated to a coherent monad is a homotopy colimit of a canonical simplicial object of free algebras. See \cite{RiehlVerity:2012hc} and \cite{RiehlVerity:2013cp}.
\end{ex}

%!TEX root = all.tex
% ******************************************************************
% ** Title:            The 2-category theory of quasi-categories
% **                   adjunctions
% ** Precis:        
% ** Author:           Emily Riehl and Dominic Verity
% ** Commenced:        2/3/2012
% ******************************************************************

\section{Pointwise universal properties}\label{sec:pointwise}

We have seen that limits and adjunctions can be encoded as absolute lifting diagrams in $\qCat_2$. In this section, we prove a theorem that allows such diagrams to be identified in practice: we show that absolute left or right lifting diagrams can be defined ``pointwise''  by specifying initial or terminal objects, respectively, in the appropriate comma or slice quasi-categories; the definition of Joyal's slice quasi-categories is recalled in \ref{defn:slices}.

 We conclude by proving a corollary of this result: that simplicial Quillen adjunctions between simplicial model categories are adjunctions of quasi-categories.  Adjunctions in homotopical contexts are commonly presented as Quillen adjunctions,  which can be replaced  by adjunctions of this type in good set-theoretical cases \cite{RezkSchwedeShipley:2001ss}. This result implies that such adjunctions can be imported into the quasi-categorical context.

\subsection{Pointwise absolute lifting}

Immediately from Definition \ref{defn:families.of.diagrams}, absolute lifting diagrams are preserved by pre-composition by all functors and, in particular, under evaluation at a vertex in the domain quasi-category.

\begin{defn}[pointwise universal property of absoluting lifting diagrams]\label{defn:pointwise-abs-lifting}
If the left-hand diagram
  \begin{equation}\label{eq:pointwise-absRlifting}
    \vcenter{\xymatrix{ \ar@{}[dr]|(.7){\Downarrow\lambda} & B \ar[d]^f \\ C \ar[r]_g \ar[ur]^\ell & A}} \qquad \rightsquigarrow \qquad   \vcenter{\xymatrix{ \ar@{}[dr]|(.7){\Downarrow\lambda c} & B \ar[d]^f \\ \Del^0 \ar[r]_{gc} \ar[ur]^{\ell c} & A}}
  \end{equation}
  is an absolute lifting diagram and $c$ is an object of $C$ then pre-composition by the functor $c\colon\Del^0\to C$ gives a 2-cell $\lambda c\colon f\ell c\Rightarrow gc$ which displays $\ell c\colon\Del^0\to B$ as an absolute right lifting of $gc\colon\Del^0\to A$ through $f\colon B\to A$. The family of absolute lifting diagrams as displayed on the right encode  the \emph{pointwise universal property} of the absolute lifting diagram displayed on the left.
\end{defn}

A special case of Proposition \ref{prop:right.liftings.as.fibred.terminal.objects} provides an alternate characterisation of a pointwise absolute lifting property:

\begin{lem}\label{lem:pointwise-terminal}
Given functors $g \colon C \to A$ and $f \colon B \to A$ the data of a pointwise absolute right lifting diagram at a vertex $c \in C$ is equally the data of a terminal object in the comma or slice quasi-categories $f \comma gc \simeq \slicer{f}{gc}$.
\end{lem}

Lemma~\ref{lem:slice-equiv-comma} supplies an equivalence $f \comma gc \simeq \slicer{f}{gc}$ along which we may transport terminal objects. Lemma \ref{lem:pointwise-terminal} demonstrates that if $g$ admits an absolute right lifting through $f$, then $f\comma gc \simeq \slicer{f}{gc}$ has a terminal object, for each vertex $c$ in the domain of $g$. In fact, these terminal objects suffice to demonstrate the existence of an absolute right lifting:

\begin{thm}\label{thm:pointwise} The functor $g\colon C\to A$ admits an absolute right lifting through the functor $f\colon B\to A$ if and only if for all objects $c$ of $C$ the quasi-category $f\comma gc\simeq\slicer{f}{gc}$ has a terminal object.
\end{thm}

\begin{proof}[Proof of Theorem \ref{thm:pointwise}]
Suppose each $\slicer{f}{gc}$ has a terminal object $\lambda_c \colon fb \to gc$, i.e., suppose we can fill any sphere $\partial\Delta^n \to \slicer{f}{gc}$ with $n \geq 1$ whose final vertex is $\lambda_c$. Unpacking the definition, we have assumed that we can solve any lifting problem 
\begin{equation}\label{eq:term.obj.assumption} \vcenter{\xymatrix{\boundary \Delta^n \ar[d] \ar[r] & \slicer{f}{gc} \\ \Delta^n \ar@{-->}[ur]}}\qquad\leftrightsquigarrow \qquad \vcenter{\xymatrix@C=25pt@R=30pt@!0{ \boundary\Delta^n \ar[rrrr] \ar[dd] \ar[dr] & && & B \ar[dr]^f \\ & \Lambda^{n+1,n+1} \ar[dd] \ar[rrrr] & & & & A \\ \Delta^n \ar[dr]_{\face^{n+1}} \ar@{-->}[uurrrr] \\ & \Delta^{n+1} \ar@{-->}[uurrrr]}}\end{equation} 
in $\qCat^\cattwo$ for which the $\fbv{n,n+1}$ edge of the $\Lambda^{n+1,n+1}$-horn in $A$ is $\lambda_c$.

It follows that we can solve any extension problem 
\begin{equation}\label{eq:term.obj.extension}\xymatrix@C=70pt@R=30pt@!0{  \boundary\Delta^n \times \Delta^{\fbv{0}} \ar[rr] \ar[dd] \ar[dr] &  & B \ar[dr]^f \\ &  \boundary\Delta^n \times \Delta^1\cup \Delta^n \times \Delta^{\fbv{1}}  \ar[dd] \ar[rr] &  & A \\ \Delta^n \times \Delta^{\fbv{0}} \ar[dr] \ar@{-->}[uurr]|\hole \\ & \Delta^n \times \Delta^1 \ar@{-->}[uurr]}\end{equation} for which the image of the edge between the vertices $(n,0)$ and $(n,1)$ is $\lambda_c$: The filler is constructed by inductively choosing images for the shuffles of $\Delta^n \times \Delta^1$  starting from the filled end of the specified cylinder. The images for all but the last shuffle are defined by filling the obvious inner horns in $A$. The final shuffle is attached by filling a $\Lambda^{n+1,n+1}$-horn in $A$ precisely of the form \eqref{eq:term.obj.assumption}.

We are interested in extension problems \eqref{eq:term.obj.extension} where the $n$-simplex in $A$ given as one end of the cylinder is in the image of some specified $n$-simplex of $C$ under $g$; these are precisely the data specified by a lifting problem \begin{equation}\label{eq:term.obj.lifting} \xymatrix{\Delta^0 \ar[r]_-{\fbv{n}} \ar@/^2ex/[rr]^{\lambda_c} & \boundary\Delta^n \ar[d] \ar[r] & f \comma g \ar[d]^{q_1} \\ &  \Delta^n \ar@{-->}[ur] \ar[r] & C}\end{equation} in which case the extension of \eqref{eq:term.obj.extension} provides a solution. We have just shown that any lifting problem \eqref{eq:term.obj.lifting} in which the final vertex of the sphere maps to a terminal object $\lambda_c \in \slicer{f}{gc}$ has a solution. By Lemma \ref{lem:RARI-lifting}, this tells us that $q_1 \colon f \comma g \to C$ admits a right adjoint right inverse $t \colon C \to f \comma g$, which by Proposition \ref{prop:right.liftings.as.fibred.terminal.objects} encodes the data of an absolute right lifting diagram, as displayed on the bottom right.
\[ 
    \vcenter{\xymatrix@=0.8em{
      & \save []+<0pt,1em>*+{C}\ar[d]^-{t}\ar@{=}@/_1.5ex/[ddl]
      \ar@/^1.5ex/[ddr]^{\ell}\restore & \\
      & {f\comma g}\ar[dr]_{q_0}\ar[dl]^{q_1} & \\
      {C}\ar[dr]_{g} & {\scriptstyle\Leftarrow\psi} &
      {B}\ar[dl]^{f} \\
      & {A} &
    }}
    \mkern20mu = \mkern20mu
    \vcenter{\xymatrix@=0.7em{
      & {C}\ar@{=}[dl]\ar[dr]^{\ell} & \\
      {C}\ar[dr]_{g} & {\scriptstyle\Leftarrow\lambda} & {B}\ar[dl]^{f} \\
      & {A} &
    }}
\]
\end{proof}

Theorem \ref{thm:pointwise} provides a useful criterion for the existence of absolute lifting diagrams. The following corollary supplies the corresponding detection result, identifying when a candidate lifting diagram has the desired universal property. The lifting property implies that each of its fibres admit terminal objects, a definition that will be introduced in the next section.

\begin{cor}\label{cor:pointwise}
  A triangle 
  \[    \xymatrix{ \ar@{}[dr]|(.7){\Downarrow\lambda} & B \ar[d]^f \\ C \ar[r]_g \ar[ur]^\ell & A}\] displays $\ell$ as an absolute right lifting of $g$ through $f$ if and only if it has that property pointwise.
\end{cor}

\begin{proof}
Necessity of the pointwise absolute lifting property of Definition \ref{defn:pointwise-abs-lifting} is immediate. Conversely, the assumed pointwise lifting tells us, in particular, that for each object $c$ in $C$ the slice quasi-category $f\comma gc\simeq\slicer{f}{gc}$ has a terminal object. Consequently, we may apply Theorem \ref{thm:pointwise} to construct a functor $\ell'\colon C\to A$ and 2-cell $\lambda'\colon f\ell'\Rightarrow g$ which displays $\ell'$ as an absolute right lifting of $g$ through $f$. 
  
The universal property of $(\ell',\lambda')$ applied to the triangle $(\ell,\lambda)$ provides us with a unique 2-cell $\tau\colon\ell\Rightarrow\ell'$ with the defining property that $\lambda'\cdot f\tau =\lambda$. Now both of the 2-cells $\lambda$ and $\lambda'$ possess the pointwise lifting property, the first by assumption and the second by construction. In other words, for all objects $c$ in $C$ the 2-cell $\lambda c\colon f\ell c\Rightarrow g c$ (respectively $\lambda' c\colon f\ell' c\Rightarrow g c$) displays $\ell c$ (respectively $\ell' c$) as an absolute right lifting of $g c$ through $f$ for all objects $c$ of $C$. Furthermore, the defining property of $\tau$ whiskers to tell us that $\lambda' c\cdot f(\tau c)=\lambda c$, so since $\lambda c$ and $\lambda' c$ both possess the absolute right lifting property it follows that $\tau c$ is an isomorphism. Applying Observation~\ref{obs:pointwise-iso-reprise}, we find that $\tau\colon\ell\Rightarrow\ell'$ is an isomorphism and thus that the given triangle is isomorphic to the absolute right lifting that we constructed and is thus itself an absolute right lifting.
\end{proof}

Proposition \ref{prop:families.of.diagrams}, which states that a quasi-category admits limits of a family of diagrams of a fixed shape if and only if it admits limits of each individual diagram in the family, is a special case of Theorem \ref{thm:pointwise}.

\begin{proof}[Proof of Proposition~\ref{prop:families.of.diagrams}]
If $A$ admits limits of each diagram in a family $k \colon K \to A^X$, then Proposition~\ref{prop:limits.are.limits} implies that for each vertex $\overline{d} \in K$, $\slicer{c}{k\overline{d}}$ has a terminal object. By Theorem~\ref{thm:pointwise}, it follows that $k$ admits an absolute right lifting along $c \colon A \to A^X$, i.e., $A$ admits limits of the family of diagrams $k \colon K \to A^X$.
\end{proof}

%Corollary \ref{cor:pointwise} also allows us to prove that Lurie's definition of adjunction given in \cite[5.2.2.8]{Lurie:2009fk} is equivalent to the 2-categorical definition.

%\begin{prop}\label{prop:joyal-adj-equals-lurie-adj} A pair of functors $f \colon B \to A$ and $u \colon A \to B$ together with a 2-cell $\epsilon \colon fu \To \id _A \in A^A$ define an adjunction $f \dashv u$ with counit $\epsilon$ if and only if the functor $B \comma u \to f \comma A$ over $A \times B$ induced by $\epsilon$, as described in proposition \ref{prop:adjointequiv}, pulls back over any pair of vertices $(a,b) \in A \times B$  to define an equivalence $b \comma ua \simeq fb \comma a$ of hom-spaces.
%\end{prop}
%\begin{proof}
%Observation \ref{obs:pointwise-adjoint-correspondence} demonstrated that the counit of an adjunction in the 2-category of quasi-categories has this property. 
%\end{proof}

\subsection{Simplicial Quillen adjunctions are adjunctions of quasi-categories}

Now we use Theorem \ref{thm:pointwise} to prove the assertions made in Example \ref{ex:simp.quillen.adj}: namely that any simplicial Quillen adjunction between simplicial model categories descends to an adjunction of quasi-categories. Another proof of this result is given in \cite[5.2.4.6]{Lurie:2009fk}. 

Recall that the quasi-category associated to a simplicial model category $\lcat{A}$ is defined by restricting to the full simplicial subcategory $\lcat{A}_{cf}$ of fibrant-cofibrant objects and then applying the homotopy coherent nerve $\nrvhc\colon \sSet\text{-}\Cat \to \sSet$. 

\begin{thm}\label{thm:simplicial-Quillen-adjunction} A simplicial Quillen adjunction \[\adjdisplay f -| u : \lcat{A} -> \lcat{B}.\] between simplicial model categories gives rise to an adjunction between the quasi-categories $\nrvhc\lcat{A}_{cf}$ and $\nrvhc\lcat{B}_{cf}$.
\end{thm}
\begin{proof}
We introduce a pair of simplicial categories  $\coll(f,\lcat{A})$ and $\coll(\lcat{B},u)$, with $\lcat{B}$ and $\lcat{A}$ as full subcategories that are jointly surjective on objects. Declare the hom-spaces from $a \in \lcat{A}$ to $b \in \lcat{B}$ to be empty and define \[ \coll(f,\lcat{A})(b,a) \defeq \lcat{A}(fb,a) \qquad \coll(\lcat{B},u)(b,a) \defeq \lcat{B}(b,ua).\] The simplicial adjunction $f\dashv u$ is encoded in the proposition that the simplicial categories $\coll(f,\lcat{A})$ and $\coll(\lcat{B},u)$ are isomorphic under $\lcat{B} \coprod \lcat{A}$.

Now write $\coll(f,\lcat{A})_{cf} \cong \coll(\lcat{B},u)_{cf}$ for the full simplicial sub-categories spanned by the fibrant-cofibrant objects of $\lcat{A}$ and $\lcat{B}$.  Via these restrictions, we obtain a diagram \[ \lcat{B}_{cf} \hookrightarrow \coll(f,\lcat{A})_{cf} \cong \coll(\lcat{B},u)_{cf} \hookleftarrow \lcat{A}_{cf}\] of locally Kan simplicial categories. Applying the homotopy coherent nerve, we have a pair of isomorphic cospans in $\qCat_2$: \[\xymatrix{ \ar@{}[dr]|(.7){\Uparrow\psi} & \nrvhc\lcat{A}_{cf} \ar[d] & &  \ar@{}[dr]|(.7){\Downarrow\beta} & \nrvhc\lcat{B}_{cf} \ar[d]^i  \\ \nrvhc\lcat{B}_{cf} \ar@{-->}[ur]^{\overline{f}} \ar[r] & \nrvhc\coll(f,\lcat{A})_{cf} & & \nrvhc\lcat{A}_{cf} \ar@{-->}[ur]^{\overline{u}} \ar[r]_-j &\nrvhc\coll(\lcat{B},u)_{cf} }\] Our objective is to define an absolute left lifting $(\overline{f}, \psi)$ and an absolute right lifting $(\overline{u},\beta)$. Proposition \ref{prop:absliftingtranslation2} and its dual then provides a fibred equivalence \[ \overline{f} \comma \nrvhc\lcat{A}_{cf} \simeq \nrvhc\lcat{B}_{cf} \comma \overline{u}\] over $\nrvhc\lcat{A}_{cf} \times \nrvhc\lcat{B}_{cf}$, which by Proposition \ref{prop:adjointequivconverse} implies that $\adjinline \overline{f} -| \overline{u} : \nrvhc\lcat{A}_{cf} -> \nrvhc\lcat{B}_{cf}.$ is an adjunction of quasi-categories.

The arguments building the absolute right lifting diagram $(\overline{u},\beta)$ and the absolute left lifting diagram $(\overline{f},\psi)$ are entirely dual. Interpreting the statement of Theorem \ref{thm:pointwise} in this context, we are asked to produce, for each fibrant-cofibrant object $a \in \lcat{A}$, a terminal object in $\slicer{i}{ja}$, defined to be the pullback of the slice  quasi-category $\slicer{(\nrvhc\coll(\lcat{B},u)_{cf})}{a}$ along the natural inclusion $i \colon N\lcat{B}_{cf} \to N\coll(\lcat{B},u)_{cf}$. To that end, choose a  cofibrant replacement $q \colon t \to ua$ in the model category $\lcat{B}$ such that the map $q$ is a trivial fibration. It follows that whenever $b \in \lcat{B}$ is cofibrant, the natural map $q_*\colon \lcat{B}(b,t) \to \lcat{B}(b,ua)$ is a trivial fibration between Kan complexes. We claim that $q$ is terminal in $\slicer{i}{ja}$.

Let $\gC$ denote the left adjoint to the homotopy coherent nerve. Unpacking the definition, an $n$-simplex in $\slicer{i}{ja}$ is \[\xymatrix{  \Delta^n \ar[d]_{\face^{n+1}} \ar[r] & \nrvhc\lcat{B}_{cf} \ar[d]  & & \gC\Delta^n\ar[d]_{\face^{n+1}} \ar[r] & \lcat{B}_{cf} \ar[d]^i\\ \Delta^n\join\Delta^0\cong  \Delta^{n+1} \ar[r] & \nrvhc\coll(\lcat{B},u)_{cf} & \leftrightsquigarrow & \gC\Delta^{n+1} \ar[r] & \coll(\lcat{B},u)_{cf} \\ \Delta^{\fbv{n+1}} \ar[u] \ar[ur]_a  &  & & \catone \ar[u]^{\mathrm{last}} \ar[ur]_a } \]  The vertex $q \in \nrvhc\coll(\lcat{B},u)_{cf}$ is terminal if and only if we can extend any diagram of simplicial functors 
\begin{equation}\label{eq:simp.adj.extension}\xymatrix@C=30pt@R=35pt@!0{ \gC\boundary\Delta^n \ar[rrrr] \ar[dd] \ar[dr] & && & \lcat{B}_{cf} \ar[dr] \\ & \gC\Lambda^{n+1,n+1} \ar[dd] \ar[rrrr] & & & & \coll(\lcat{B},u)_{cf} \\ \gC\Delta^n \ar[dr]_{\face^{n+1}} \ar@{-->}[uurrrr] \\ & \gC\Delta^{n+1} \ar@{-->}[uurrrr]}\end{equation} in which the unique vertex in the hom-space between the objects $n$ and $n+1$ in the simplicial category $\gC\Lambda^{n+1,n+1}$ is mapped to $q \in \lcat{B}(t,ua)$. 

The simplicial categories $\gC\Lambda^{n+1,n+1}$ and $\gC\Delta^{n+1}$ have objects $0,\ldots, n+1$ and all but two of the same hom-spaces, the only exceptions being the hom-spaces from $0$ to $n$ and to $n+1$. We have $\gC\Delta^{n+1}(0,n) \cong (\Delta^1)^{n-1}$ and $\gC\Delta^{n+1}(0,n+1)\cong(\Delta^1)^n$, while $\gC\Lambda^{n+1,n+1}(0,n) \cong \boundary(\Delta^1)^{n-1}$ and $\gC\Lambda^{n+1,n+1}(0,n+1)$ is the open box $B \hookrightarrow (\Delta^1)^n$ with the interior of the $n$-cube and one face removed \cite[1.1.5.10]{Lurie:2009fk} and \cite[16.5.10]{Riehl:2014kx}. In this way, writing $b \in \lcat{B}$ for the image of the object 0, the extension problem \eqref{eq:simp.adj.extension} in the category of simplicial categories reduces to an extension problem
\[\xymatrix@C=30pt@R=35pt@!0{ \boundary(\Delta^1)^{n-1} \ar[rrrr] \ar@{ >->}[dd] \ar[dr] & && & \lcat{B}(b,t) \ar@{->>}[dr]^{q_*}_(.4){\rotatebox{135}{$\labelstyle\sim$}} \\ & B \ar@{ >->}[dd]^{\rotatebox{90}{$\labelstyle\sim$}} \ar[rrrr] & & & & \lcat{B}(b,ua) \\ (\Delta^1)^{n-1} \ar[dr] \ar@{-->}[uurrrr] \\ & (\Delta^1)^n \ar@{-->}[uurrrr]}\] in the category of simplicial sets. For the reader's convenience, we have used the standard decorations to mark cofibrations, fibrations, and weak equivalences in Quillen's model structure on simplicial sets.

The extension \eqref{eq:simp.adj.extension} may be achieved by first  extending along the map $B \hookrightarrow (\Delta^1)^n$ in the Kan complex $\lcat{B}(b,ua)$. This chooses an image under the map $q_*$ for the $(n-1)$-cube missing from the box $B$. An $(n-1)$-cube in $\lcat{B}(b,t)$ with this image can be found by lifting the cofibration $\boundary(\Delta^1)^{n-1} \hookrightarrow (\Delta^1)^{n-1}$ against the trivial fibration $q_*$.
\end{proof}

%\appendix
%\input{geometry}

\bibliographystyle{abbrv}
\bibliography{index}

\begin{thebibliography}{10}

\bibitem{Bergner:2007fk}
J.~E. Bergner.
\newblock A model category structure on the category of simplicial categories.
\newblock {\em Transactions of the American Mathematical Society},
  359:2043--2058, 2007.

\bibitem{Boardman:1973xo}
J.~Boardman and R.~Vogt.
\newblock {\em Homotopy Invariant Algebraic Structures on Topological Spaces},
  volume 347 of {\em Lecture Notes in Mathematics}.
\newblock Springer-Verlag, 1973.

\bibitem{Cordier:1982:HtyCoh}
J.-M. Cordier.
\newblock Sur la notion de diagramme homotopiquement coh{\'e}rent.
\newblock In {\em Proc. 3{\'e}me Colloque sur les Cat{\'e}gories, Amiens
  (1980)}, volume~23, pages 93--112, 1982.

\bibitem{Cordier:1986:HtyCoh}
J.-M. Cordier and T.~Porter.
\newblock Vogt's theorem on categories of homotopy coherent diagrams.
\newblock {\em Mathematical Proceedings of the Cambridge Philosophical
  Society}, 100:65--90, 1986.

\bibitem{DuggerSpivak:2011ms}
D.~Dugger and D.~I. Spivak.
\newblock Mapping spaces in quasi-categories.
\newblock {\em Algebr. Geom. Topol.}, 11(1):263--325, 2011.

\bibitem{GepnerHaugseng:2013ec}
D.~Gepner and R.~Haugseng.
\newblock Enriched $\infty$-categories via non-symmetric $\infty$-operads.
\newblock arXiv:1312.3178 [math.AT], 2013.

\bibitem{GepnerHaugsengNikolaus:2015lc}
D.~Gepner, R.~Haugseng, and T.~Nikolaus.
\newblock Lax colimits and free fibrations in $\infty$-categories.
\newblock arXiv:1501.02161 [math.CT].

\bibitem{Hovey:1999fk}
M.~Hovey.
\newblock {\em Model Categories}, volume~63 of {\em Mathematical Surveys and
  Monographs}.
\newblock American Mathematical Society, Providence, RI, 1999.

\bibitem{Joyal:2002:QuasiCategories}
A.~Joyal.
\newblock Quasi-categories and {K}an complexes.
\newblock {\em Journal of Pure and Applied Algebra}, 175:207--222, 2002.

\bibitem{Joyal:2008tq}
A.~Joyal.
\newblock {\em The theory of quasi-categories and its applications}.
\newblock Quadern 45, Vol. II, Centre de Recerca Matem\`{a}tica Barcelona.
  2008.

\bibitem{Joyal:2007kk}
A.~Joyal and M.~Tierney.
\newblock Quasi-categories vs {S}egal spaces.
\newblock In A.~D. et~al, editor, {\em Categories in Algebra, Geometry and
  Physics}, volume 431 of {\em Contemporary Mathematics}, pages 277--326.
  American Mathematical Society, 2007.

\bibitem{Kelly:1989fk}
G.~M. Kelly.
\newblock Elementary observations on 2-categorical limits.
\newblock {\em Bulletin of the Australian Mathematical Society}, 49:301--317,
  1989.

\bibitem{Kelly:2005:ECT}
G.~M. Kelly.
\newblock Basic concepts of enriched category theory.
\newblock {\em Reprints in Theory and Applications of Categories}, 10, 2005.

\bibitem{kelly.street:2}
G.~M. Kelly and R.~H. Street.
\newblock {\em Review of the Elements of 2-Categories}, volume 420 of {\em
  Lecture Notes in Math.}, pages 75--103.
\newblock Springer-Verlag, 1974.

\bibitem{Lurie:2009fk}
J.~Lurie.
\newblock {\em {Higher Topos Theory}}, volume 170 of {\em Annals of
  Mathematical Studies}.
\newblock Princeton University Press, Princeton, New Jersey, 2009.

\bibitem{Lurie:2012uq}
J.~Lurie.
\newblock {Higher Algebra}.
\newblock \verb!http://www.math.harvard.edu/~lurie/papers/HigherAlgebra.pdf!,
  September 2014.

\bibitem{Maclane:1971:CWM}
S.~{Mac Lane}.
\newblock {\em Categories for the Working Mathematician}.
\newblock Springer-Verlag, New York, 1971.

\bibitem{Meyer:84ba}
J.-P. Meyer.
\newblock Bar and cobar constructions. {I}.
\newblock {\em J. Pure Appl. Algebra}, 33(2):163--207, 1984.

\bibitem{NicholsBarrer:2007oq}
J.~P. Nichols-Barrer.
\newblock {\em On quasi-categories as a foundation for higher algebraic
  stacks}.
\newblock PhD thesis, Massachusetts Institute of Technology. Dept. of
  Mathematics., 2007.

\bibitem{Quillen:1967:Model}
D.~G. Quillen.
\newblock {\em Homotopical Algebra}, volume~43 of {\em Lecture Notes in
  Mathematics}.
\newblock Springer-Verlag, 1967.

\bibitem{Rezk:2001sf}
C.~Rezk.
\newblock A model for the homotopy theory of homotopy theory.
\newblock {\em Transactions of the American Mathematical Society},
  353(3):973--1007, 2001.

\bibitem{RezkSchwedeShipley:2001ss}
C.~Rezk, S.~Schwede, and B.~Shipley.
\newblock Simplicial structures on model categories and functors.
\newblock {\em Amer. J. Math}, 123(3):551--575, 2001.

\bibitem{Riehl:2014kx}
E.~Riehl.
\newblock {\em Categorical homotopy theory}, volume~24 of {\em New Mathematical
  Monographs}.
\newblock Cambridge University Press, 2014.

\bibitem{RiehlVerity:2012hc}
E.~Riehl and D.~Verity.
\newblock Homotopy coherent adjunctions and the formal theory of monads.
\newblock 2013.
\newblock arXiv:1310.8279 [math.CT].

\bibitem{RiehlVerity:2013cp}
E.~Riehl and D.~Verity.
\newblock Completeness results for quasi-categories of algebras, homotopy
  limits, and related general constructions.
\newblock pages 1--34, 2014.
\newblock arXiv:1401.6247 [math.CT].

\bibitem{RiehlVerity:2013kx}
E.~Riehl and D.~Verity.
\newblock The theory and practice of {Reedy} categories.
\newblock {\em Theory and Applications of Categories}, 29(9):256--301, 2014.
\newblock arXiv:1304.6871 [math.CT].

\bibitem{RiehlVerity:2015tt-v3}
E.~Riehl and D.~Verity.
\newblock The 2-category theory of quasi-categories.
\newblock arXiv:1306.5144v3 [math.CT], 2015.

\bibitem{RiehlVerity:2015fy}
E.~Riehl and D.~Verity.
\newblock Fibrations and {Y}oneda's lemma in an $\infty$-cosmos.
\newblock in preparation, 2015.

\bibitem{Roberts:1978:Complicial}
J.~E. Roberts.
\newblock Complicial sets.
\newblock handwritten manuscript, 1978.

\bibitem{Street:1987:Oriental}
R.~H. Street.
\newblock The algebra of oriented simplexes.
\newblock {\em Journal of Pure and Applied Algebra}, 49:283--335, 1987.

\bibitem{Street:2003:WomCats}
R.~H. Street.
\newblock Weak omega-categories.
\newblock In {\em Diagrammatic Morphisms and Applications}, volume 318 of {\em
  Contemporary Mathematics}, pages 207--213. American Mathematical Society,
  2003.

\bibitem{Verity:2008sr}
D.~Verity.
\newblock {\em Complicial Sets, Characterising the Simplicial Nerves of Strict
  $\omega$-Categories}, volume 193 of {\em Memoirs of the AMS}.
\newblock American Mathematical Society, May 2008.

\bibitem{Verity:2007:wcs1}
D.~Verity.
\newblock Weak complicial sets {I}, basic homotopy theory.
\newblock {\em Advances in Mathematics}, 219:1081--1149, September 2008.

\end{thebibliography}

% \listoffixmes

\end{document}